\documentclass[12pt]{amsart}
\usepackage{amsmath, amsthm, amssymb}
\usepackage{fullpage}
\usepackage{color}
\usepackage{hyperref}
\usepackage{enumitem}
\usepackage{algorithm}
\usepackage{bm}

\usepackage{ulem}

\newtheorem{theorem}{Theorem}
\newtheorem{lemma}{Lemma}
\newtheorem{proposition}[lemma]{Proposition}
\newtheorem{corollary}[lemma]{Corollary}

\numberwithin{lemma}{section}

\newcommand{\eps}{\epsilon}
\newcommand{\half}{\frac{1}{2}}
\newcommand{\D}{\partial}
\newcommand{\Elin}{E_{lin}}

\newcommand{\As}{{A^\sharp}}

\newcommand{\w}{{\mathfrak{w}}}
\newcommand{\W}{{\mathbf W}}
\renewcommand{\r}{{\mathfrak r}}
\newcommand{\q}{{\mathfrak{q}}}

\newcommand{\uu}{{\mathbf u}}

\newcommand{\qq}{{\mathbf q}}
\newcommand{\ww}{{\mathbf w}}
\renewcommand{\ggg}{{\mathbf g}}

\numberwithin{equation}{section}

\newcommand{\R}{{\mathbb R}}

\newcommand{\Z}{{\mathbb Z}}

\newcommand{\tW}{{\tilde W}}
\newcommand{\tQ}{{\tilde Q}}

\renewcommand{\H}{{\mathcal H }}
\renewcommand{\AA}{\mathbf A}

\newcommand{\sWH}{{\mathcal{WH}^\sharp }}
\newcommand{\tWH}{{\tilde{\mathcal{WH} }}}
\newcommand{\dH}{{\dot{\mathcal H} }}

\newcommand{\sgn}{\mathop{\mathrm{sgn}}}
\newcommand{\tG}{\tilde{G}}
\newcommand{\tK}{\tilde{K}}
\newcommand{\tS}{\tilde \S}
\newcommand{\tF}{\tilde F}

\newcommand{\bchi}{\bm{\chi}}

\renewcommand{\S}{{\mathbf S}}
\newcommand{\err}{\text{\bf err}}

\usepackage{xcolor}

\begin{document}
\normalem

\title{Two dimensional gravity waves at low regularity II: Global solutions}

\author{Albert Ai}
\address{Department of Mathematics, University of Wisconsin, Madison}
%\thanks{}
\email{aai@math.wisc.edu}
\author{Mihaela Ifrim}
 \address{Department of Mathematics, University of Wisconsin, Madison}
\email{ifrim@wisconsin.edu}
%\thanks{The second author was supported by }
\author{ Daniel Tataru}
\address{Department of Mathematics, University of California at Berkeley}
 %\thanks{The third author was partially supported by }
\email{tataru@math.berkeley.edu}

\begin{abstract}
  This article represents the second installment of a series of papers concerned with low regularity solutions for the  water wave equations in two space dimensions. Our focus here is on global solutions for small
  and localized data. Such solutions have been proved to exist 
  earlier in \cite{ip,AD,HIT,IT-global} in much higher regularity.
  Our goal in this paper is to improve these results and prove
  global well-posedness under minimal regularity and decay assumptions for the initial data. One key ingredient here is represented by the \emph{balanced cubic estimates} in our first paper. Another is 
  the nonlinear vector field Sobolev inequalities, an idea first introduced by the last two authors in the context of the Benjamin-Ono equations \cite{BO}.
  
\end{abstract}

\setcounter{tocdepth}{1}
\maketitle

\tableofcontents

\section{Introduction}
We consider the two dimensional water wave equations with infinite
depth, with gravity but without surface tension.  This is governed by
the incompressible Euler's equations with boundary conditions on the
water surface. Under the additional assumption that the flow is
irrotational, the fluid dynamics can be expressed in terms of a
one-dimensional evolution of the water surface coupled with the trace
of the velocity potential on the surface.  

The choice of the parametrization of the free
boundary plays an important role here, and can be viewed as a form of gauge freedom.  Historically there are three such choices of coordinates; the first two, namely  the Eulerian and Lagrangian coordinates, arise in the broader context  of fluid
dynamics. The third employs the so-called  conformal method, which is specific to two dimensional irrotational flows; this leads to what we call the holomorphic coordinates, which play a key role in the present paper. 

Our objective in this series of papers is to improve, 
streamline, and simplify the analysis of the two dimensional
gravity wave equations. This is a challenging quasilinear, nonlocal,
non-diagonal system. We aim to develop its analysis in multiple ways,
including:

\begin{enumerate}
\item  prove better, scale invariant energy estimates,

\item improve the existing results on long time solutions,

\item  refine the study of the dispersive properties and improve the low regularity theory.

\end{enumerate}

The first step of this program was carried out in \cite{AIT}, where 
we have developed a new class of estimates, which we called \emph{balanced 
energy estimates}, which led to drastic improvements in the study of the low 
regularity well-posedness for this problem.

In the present article we carry out the second step of this program,
and obtain an enlarged class of global solutions, with decaying initial data of minimal regularity. In a nutshell, our result reads as follows:

\begin{theorem}
Small and localized data leads to global solutions, which exhibit dispersive 
$t^{-\frac12}$ uniform decay.
\end{theorem}

Compared to the prior work of the last two authors \cite{HIT}, \cite{IT-global},
in this paper we bring forth several key improvements:

\begin{enumerate}[label=\roman*)]
\item We lower the regularity requirements for the initial data 
both at low and at high frequency, to almost optimal levels. In other words,
our global well-posedness results are nearly scale invariant, at almost the 
same regularity level that would be required for an equivalent, semilinear,  
cubic NLS problem.

\item We use the sharp, cubic balanced energy estimates of \cite{AIT}, as well as 
the Alazard-Delort idea in \cite{AD} of performing a partial normal form transformation in order to 
further simplify and streamline the proof.

\item At a technical level, we develop in this context the idea of nonlinear 
paradifferential vector field Sobolev inequalities, which was first introduced by the last two authors
in the Benjamin-Ono context \cite{BO}.
\end{enumerate}

\subsection{Holomorphic coordinates}

It has been known since the work of Zakharov~\cite{zak} that under an irrotationality 
condition, the water wave equations can be viewed as a self contained system for the water surface
together with the trace of the velocity potential on the free surface. For a two dimensional fluid, this yields a fully nonlinear first order system in one space dimension.

In addition, one has the freedom of choosing the parametrization of the free surface 
in a favourable manner. Classical parametrizations rely on either the Eulerian or the Lagrangian coordinates. But in the two dimensional case, there is a better choice, that is
the holomorphic (conformal) coordinates, which are based on conformally representing 
the two dimensional fluid domain as a half-plane. These coordinates were independently introduced by 
Wu~\cite{wu} and Zakharov \& al. \cite{zakharov} in the study of the dynamical problem, though 
conformal coordinates of various types had been used before in the study of traveling and solitary waves.

In this article we will use the holomorphic coordinates, but in an alternative\footnote{This should be compared with the use of real 
valued functions in \cite{zakharov}, or with a second order evolution formulation in \cite{wu}.}
formulation developed in the last two authors prior work \cite{HIT}, jointly with Hunter. Denoting by $\alpha$
the variable on the real line and by $\alpha + i \beta$ the complex, conformal coordinates in the lower half space, the water wave equations are written as a system for a pair of complex valued functions $(W,Q)$ on the real line, as follows:
\begin{itemize}
    \item $\alpha \to \alpha +W(\alpha)$ represents the conformal parametrization of the fluid surface, which is a non-self-intersecting curve but not necessarily a graph. 
    
    \item $\alpha \to Q(\alpha)$ represents the complex velocity potential on the free surface, where the real part of $Q$ is the real velocity potential and its imaginary part is its harmonic conjugate, namely the stream 
    function. It is only defined modulo constants.
\end{itemize}

Here $(W,Q)$ are further restricted to the class of functions that by a slight abuse we call holomorphic, i.e. which admit holomorphic extensions to the lower half-space, with suitable decay conditions in depth. In the infinite 
depth case these are exactly the functions which are frequency localized to negative frequencies. 
One significant advantage of this choice is that this class of functions forms an algebra.

 With this choice of variables, following \cite{HIT}, the nonlinear water waves system takes the form
\begin{equation}\label{nlin}
\left\{
\begin{aligned}
&W_t + F(1 + W_\alpha) = 0, \\
&Q_t + FQ_\alpha - iW + P[\bar R R] = 0,
\end{aligned}
\right.
\end{equation}
\begin{equation*}
%\label{rww2d}
 F = P\left[\frac{Q_\alpha - \bar Q_\alpha}{J}\right], \qquad J = |1+W_\alpha|^2, \qquad R=\frac{Q_\alpha}{1+W_\alpha},
 \end{equation*}
where $P$ is the projector to negative frequencies. The factor $R$ above has an intrinsic meaning, namely it is the complex velocity on the water surface. Also note that $J$ represents the Jacobian of the conformal change of coordinates. We can also re-express $F$ in terms of $Y$, where the function $Y$, given by
\begin{equation}\label{Y-def}
Y := \frac{\W}{1+\W}, \qquad \mbox{and}\qquad  \W:=W_{\alpha},
\end{equation}
is introduced in order to avoid rational expressions above and in many
places in the sequel; then we have
\[
F=R+P\left[ \bar RY -R\bar Y\right].
\]
This system admits a conserved energy (Hamiltonian)
\begin{equation}\label{energy}
E(W, Q) = \int \frac12 |W|^2 + \frac{1}{2i}(Q\bar Q_\alpha - \bar Q Q_\alpha) - \frac14 (\bar W^2 W_\alpha + W^2 \bar W_\alpha) \, d\alpha.
\end{equation}

We also consider the system for the differentiated good variables $(\W, R)$,
which are what we call  \emph{the  diagonal variables} 
\[
(\W,R): = \left(W_\alpha, \frac{Q_\alpha}{1+W_\alpha}\right).
\]
Differentiating \eqref{nlin} yields a self-contained system in $(\W, R)$:
\begin{equation} \label{ww2d-diff}
\left\{
\begin{aligned}
 &D_t \W + \frac{(1+\W) R_\alpha}{1+\bar \W}   =  (1+\W)M,
\\
&D_t R = i\left(\frac{\W - a}{1+\W}\right),
\end{aligned}
\right.
\end{equation}
which is satisfied in full but is equivalent to its projected version onto the holomorphic class. 
Here $D_t = \partial_t + b\partial_\alpha$ plays the role of the material derivative, 
$b$ is the {\em advection velocity} and is given by
\begin{equation*}
b = P \left[\frac{R}{1+\bar \W}\right] +  \bar P\left[\frac{\bar R}{1+\W}\right],
%\label{defb1}
\end{equation*}
and $1+a$ is the {\em Taylor coefficient}, which represents the normal derivative of the pressure on the free surface, and is given by
\begin{equation}\label{a-def}
a := i\left(\bar P \left[\bar{R} R_\alpha\right]- P\left[R\bar{R}_\alpha\right]\right).
\end{equation}
Finally the auxiliary function $M$, closely related to the material derivative of $a$,
has the expression
\begin{equation}\label{M-def}
M :=  \frac{R_\alpha}{1+\bar \W}  + \frac{\bar R_\alpha}{1+ \W} -  b_\alpha =
\bar P [\bar R Y_\alpha- R_\alpha \bar Y]  + P[R \bar Y_\alpha - \bar R_\alpha Y].
\end{equation}

To complete our description of the equations we also need to add the linearized equations
which are best seen not as an evolution for the linearized variables $(w,q)$ associated to $(W,Q)$,
but rather as an evolution for the \emph{good linearized variables}
\begin{equation}
(w,r) = \AA(w,q) = (w,q-Rw).    
\end{equation}
These equations have the form
\begin{equation}\label{lin(wr)}
\left\{
\begin{aligned}
& PD_t w + P \left[ \frac{1}{1+\bar \W} r_\alpha\right]
+  P \left[ \frac{R_{\alpha} }{1+\bar \W} w \right] = P \mathcal{G}_0(  w, r),
 \\
&PD_t r - i P\left[ \frac{1+a}{1+\W} w\right]  =
 P \mathcal{K}_0( w,r),
\end{aligned}
\right.
\end{equation}
where $(\mathcal{G}_0, \mathcal{K}_0)$ represent perturbative terms, see \cite{HIT}.

\subsection{Sobolev spaces and local well-posedness}
The well-posedness for the water wave system \eqref{nlin} is naturally considered
on a scale of Sobolev spaces inspired by the conserved energy in \eqref{energy}.
Its quadratic part corresponds to the Hilbert space $\H$ with norm
\[
\| (w,r)\|_{\H}^2 = \|w\|_{L^2}^2 + \|r\|_{\dot H^\frac12}^2.
\]
For higher regularity we use  the scale of Sobolev spaces $\H^s$,
which we recall from \cite{HIT} and \cite{AIT},
endowed with the norm 
\[
\| (w,r) \|_{\H^s} :=  \| \langle D \rangle^s (w, r)\|_{ L^2 \times \dot H^\frac12},
\]
where $s \in \R$. 

 Since many of the estimates in both this paper and its predecessor 
 \cite{AIT} are scale invariant,
to describe them it is very useful to also have homogeneous versions 
of the above spaces, namely the spaces $\dH^s$ endowed with the norm 
\[
\| (w,r) \|_{\dH^s} :=  \| |D|^s (w, r)\|_{ L^2 \times \dot H^\frac12}.
\]
We caution the reader that, in order to streamline the exposition here, our 
notation for the energy spaces differs slightly from the notation used in \cite{HIT}.

For the local well-posedness problem, it suffices to work with the 
differentiated system \eqref{ww2d-diff}.  For this we have the following result:

\begin{theorem}[\cite{AIT}]
The differentiated water wave system \eqref{ww2d-diff} is locally well-posed
in $\H^s$ for $s \geq \frac34$.
\end{theorem}
For reference one should compare from below with scaling which corresponds to $s_0=\frac12$. This result represents the current best result, following a succession of several other results. This started with  the work of Alazard-Burq-Zuily~\cite{abz}, who proved energy estimates and well-posedness  roughly for $s = 1+\delta$ with $\delta > 0$. Using the holomorphic
setting and further structural properties of the equations, the energy
estimates were improved by the last two authors together with Hunter \cite{HIT} to
the case $\delta = 0$. This is an important threshold as it is
where the Lipschitz property for the velocity is lost.  Further
improvements  were obtained in subsequent work of Alazard-Burq-Zuily~\cite{abz1, abz-c1},
who proved and used appropriate Strichartz
estimates for this system. Their result in $2$-d yields local well-posedness 
in  for $\delta = -1/24$. This was followed by the results of the first author,
who was able to further improve this first to $\delta = - 1/10$ in \cite{Ai1} and then to $\delta = - 1/8$ in \cite{Ai2}, and finally to the result above.

A family of energy estimates developed by the authors in \cite{AIT}, which we call \emph{balanced energy estimates}, played the key role in the proof of this result. The same estimates play an essential role in the present paper, as they are part of what allows us to reach the optimal regularity threshold. They are described in detail in Section~\ref{s:ee}.

\subsection{Global solutions and the main result}
In order to state our main result, we introduce appropriate weighted norms
which are based on the scaling symmetry of the problem. 
Precisely, the equations \eqref{nlin} are invariant with respect to the scaling law 
\[
(W(t, \alpha), Q(t, \alpha)) \rightarrow (\lambda^{-2} W(\lambda t, \lambda^2 \alpha), \lambda^{-3} Q(\lambda t, \lambda^2 \alpha)).
\]
The generator of this symmetry is the scaling operator
\[
\S(W, Q) = ((S - 2)W, (S - 3)Q),
\]
where we define the scaling vector field by
\[
S = t\D_t + 2\alpha \D_\alpha.
\]

Writing 
\begin{equation} \label{def:bwr}
(\w, \r) = \AA\S(W, Q),
\end{equation}
where $\AA$ represents the diagonalization operator 
\begin{equation}\label{def:AA}
\AA(\w, \q) := (\w, \q - R\w),\qquad R = \frac{Q_\alpha}{1 + W_\alpha},
\end{equation}
and $\sigma > 11/4$, we define the weighted energy norm
\begin{equation}
\|(W, Q)(t) \|_{\sWH} = \|(W, Q)(t)\|_{\dH^{\frac14}} + \|(\W, R)(t)\|_{\dH^{\sigma-1}} + \|(\w, \r)(t)\|_{\dH^{\frac14}}.
\end{equation}
We remark that at time $t = 0$ this simply becomes 
\begin{equation}
\|(W, Q)(0) \|_{\sWH} \approx \|(W, Q)(0)\|_{\dH^{\frac14}} + \|(\W, R)(0)\|_{\dH^{\sigma-1}} + \|\alpha (\W, R)(0)\|_{\dH^{\frac14}}.
\end{equation}

In order to track the uniform, dispersive decay of the solutions,
we will also use a  pointwise control norm, namely 
\[
\|(\W, R)(t)\|_{X} = \||D|^{-\frac12} \W\|_{L^\infty} + \|R\|_{L^\infty} + \| \W\|_{\dot B^\frac14_{\infty,2}} + \| R\|_{\dot B^\frac34_{\infty,2}},
\]
where the above homogeneous Besov norms are defined as
\[
\| u \|_{\dot B^s_{\infty,2}}^2 = 
\sum_{k \in \Z} 2^{2ks} \| P_k u \|_{L^\infty}^2
\]
with $P_k$ denoting the standard spatial Littlewood-Paley projectors at dyadic frequency $2^k$.

Given these definitions, our main result is as follows:

\begin{theorem}\label{t:global}
Assume that the initial data for the water wave system \eqref{nlin} satisfies
\begin{equation}\label{data}
\|(W, Q)(0) \|_{\sWH}   \leq \epsilon \ll 1.
\end{equation}

Then the solution $(W,Q)$ is global in time, and satisfies 
the global energy bounds
\begin{equation}\label{ee}
\|(W, Q)(t) \|_{\sWH}   \leq \epsilon  \langle t \rangle^{c\epsilon^2} 
\end{equation}
with a universal constant $c$,
as well as the global pointwise bounds
\begin{equation}\label{point}
    \|(\W, R)(t) \|_{X}   \lesssim \epsilon  \langle t \rangle^{-\frac12}.
\end{equation}    
\end{theorem}

To place this result into context, one should start with Wu's almost global result \cite{wu2},
which was based on a mix of conformal and Lagrangian coordinates. Her work was further developed 
by Ionescu-Pusateri \cite{ip} to a global result. Independently, Alazard-Delort \cite{AD} obtained a different proof of the global result, based on a new idea which they called \emph{paradiagonalization},
which combines a partial normal form transformation with a microlocal diagonalization of the remaining 
system, which is done at the paradifferential level. Both of these results required extensive arguments, as well as very high regularity for the initial data.

Shortly afterward, the last two authors' work \cite{HIT}, \cite{IT-global}, the first also joint
with Hunter, brought a new perspective and a new proof of the global result for this problem, with shorter, simpler arguments at far lower regularity, which corresponds to $\H^6$ with the notations above. These advances were primarily due to two new ideas, implemented in the  context of holomorphic coordinates:

\begin{enumerate} [label=\roman*)]
    \item The \emph{modified energy method}, which asserts that, in quasilinear problems, it is 
    more efficient to construct normal form inspired modified energies which are accurate to quartic order, rather than trying to directly apply a normal form transformation.

   \item The \emph{wave packet testing}, which is an efficient way to capture asymptotic equations
   in a modified scattering scenario.
\end{enumerate}

Another key idea in \cite{HIT} was that the main estimate, and the bulk of the analysis, should be carried out at the level of the linearized equations rather than on the full equations. This 
contributed to both strengthen the results and to streamline the arguments.

 The aim of the present paper is to take advantage of further gains in understanding the 
 best ideas and methods  that can be applied to this class of problems, in order to obtain 
 a \emph{ near optimal} result. Compared to \cite{HIT}, \cite{IT-global}, there are four
 such improvements:
 
 \begin{enumerate}[label=\roman*)]
     \item In terms of energy estimates, we are able to replace the cubic energy estimates 
of \cite{HIT} with sharper ones, which we call \emph{balanced energy estimates}. These estimates, recently proved by the authors in \cite{AIT}, are still cubic, akin to \cite{HIT}, but have a better balance of regularity in the control norms, which allows us to lower the required data regularity 
in the result. Notably, these estimates hold both at the level of the full equation and at the level of the linearized equation.
     
    \item In terms of normal form analysis, we borrow an idea from Alazard-Delort \cite{AD}, 
    which is to ``prepare'' the problem with a partial normal form transformation. This allows 
    us to ultimately reduce a good portion of the analysis to a more favourable, paradifferential setting, without losing any regularity in the process.
    
    \item In order to convert vector field energy estimates to pointwise bounds, 
we use an idea inspired from the last two authors' work \cite{BO} on the Benjamin-Ono equation, and 
prove the pointwise bounds in a nonlinear, paradifferential setting, rather than in a linear setting as in \cite{IT-global}. This is important because the reduction to the linear setting inherently loses derivatives.

 \item The wave packet testing, which uses the same principle as in \cite{IT-global}, is now 
 also applied in the paradifferential setting rather than in a more NLS-like scenario, as in 
 \cite{IT-global}. This creates additional difficulties, but ultimately does not affect the 
 asymptotic equation.
 
 \end{enumerate}

Our refined analysis in this article allows us not only to relax the initial data regularity 
at high frequency to the nearly optimal level $\dH^\sigma$ with $\sigma > 11/4$, but also to relax the initial data regularity at low frequency to $\dH^\frac14$, which in particular allows for initial 
data with infinite energy. An improvement of this type has been previously obtained by Wang~\cite{MR3730012}, but only to $\dH^\frac15$.

\subsection{ On optimality}
Our goal here is to heuristically explain why our result is nearly optimal, 
by comparing it with its sharp counterpart for the cubic NLS problem.

We begin by recalling the optimal result for cubic NLS,
\[
i u_t - \Delta u = \pm u|u|^2, \qquad u(0) = u_0.
\]
Small and localized data for this problem leads to global solutions.
A good starting point here for instance is the the result of the 
last two authors in \cite{NLS}, which asserts that 
an appropriate smallness condition is 
\[
\| u_0\|_{L^2} + \| x u_0\|_{L^2} \ll 1.
\]
This is not scale invariant, but by scaling one can replace it with a 
scale invariant counterpart
\begin{equation}\label{nls-small}
\| u_0\|_{L^2} \| x u_0\|_{L^2} \ll 1,
\end{equation}
which roughly corresponds to 
\[
x^\frac12 u_0 \in L^2.
\]

On the other hand, for the water wave problem,  our smallness assumption for the initial data reads
\[
\| (\W,R)\|_{\dH^{\sigma-1}} + \| (W,Q)\|_{\dH^\frac14} + \| \alpha (\W,R)\|_{\dH^\frac14} \ll 1.
\]
Consider the limiting case $\sigma = \frac{11}4$. Then by scaling one can replace this 
smallness condition with 
\begin{equation}\label{ww-small}
\| (\W,R)\|_{\dH^{\sigma-1}} (\| (W,Q)\|_{\dH^\frac14} + \| \alpha (\W,R)\|_{\dH^\frac14}) \ll 1.
\end{equation}
Here the last two norms were kept together, as they have the same scaling and are in effect related via a Hardy type inequality at the linear level.

We will argue that, in a suitable interpretation, the two smallness relations \eqref{nls-small} and 
\eqref{ww-small} are essentially equivalent in a frequency localized setting. To see why this is so 
one should think in terms of NLS approximation results for water waves, for which we refer the reader 
to \cite{NLS-approx} and references therein, and also \cite{Dull}. In a nutshell, these results 
assert that water waves are well approximated by the (focusing) cubic NLS in well chosen regimes
as follows:
\begin{itemize}
    \item The frequency of the solutions is well localized near a given frequency $\xi_0$,
    around which the water waves linear dispersion relation $\tau = \pm \sqrt{|\xi|}$ is well approximated by its quadratic approximation.
    
    \item The water wave to NLS connection is given via a normal form transformation, which eliminates
    quadratic interactions and leaves only cubic interactions, as in the NLS case. 
\end{itemize}

For water waves, after diagonalization and normal form analysis 
we have a cubic nonlinearity, which for a diagonal variable
\[
v = \tilde W + i |D|^\frac12 \tilde Q
\]
has roughly the form
\begin{equation}\label{ww-cubic}
i v_t - |D|^\frac12 v = |D|^{\frac52}( v |v|^2).
\end{equation}
Here we neglect the exact placement of derivatives, only counting the total 
number, as this approximation is valid anyway only near a fixed frequency.
At this level, our smallness assumption \eqref{ww-small} becomes
\begin{equation}\label{v-small}
\| v(0)\|_{\dot H^{\sigma}}  \| \alpha v_\alpha(0) \|_{\dot H^\frac14} \ll 1.
\end{equation}

To relate this problem with the cubic NLS, we consider solutions $v$
at a fixed frequency $\lambda$. For the dispersion relation, 
we approximate our relation with a quadratic one, neglecting the constant and the linear part (as in Galilean invariance). At frequency $\lambda$
we have 
\[
\frac{\partial^2}{\partial \xi^2} |\xi|^{\frac12} \approx \lambda^{-\frac32}.
\]
Then our reduced equation \eqref{ww-cubic} should be compared with the NLS type problem
\[
i v_t - \lambda^{-\frac32} \Delta v = \lambda^{\frac52} v |v|^2.
\]
To eliminate the scaling parameters without changing the $v$ frequency $\lambda$ we substitute
\[
v(t,x) = \lambda^{-2} u(\lambda^{-\frac32} t,x).
\]
Now $u$ solves the cubic NLS.

It remains to compare the smallness assumptions. At frequency $\lambda$, the smallness
condition \eqref{ww-small} for $v$ reads
\[
\|\lambda^{\frac{11}4} v(0)\|_{L^2} \| \lambda^{\frac54} x v(0)\|_{L^2} \ll 1,
\]
or equivalently 
\[
\| v(0)\|_{L^2} \|  \alpha v(0)\|_{L^2} \ll \lambda^{-4}.
\]
Translated to $u$, we have arrived exactly at  \eqref{nls-small}.

\subsection*{Acknowledgements}
The first author was supported by the Henry Luce Foundation. The second author was supported by a Luce Associate Professorship, by the Sloan Foundation, and by an NSF CAREER grant DMS-1845037. The third author was supported by the NSF grant DMS-1800294 as well as by a Simons Investigator grant from the Simons Foundation.  The authors are very grateful 
to the anonymous referee for the careful reading of the paper, which led to many clarifications and improvements in exposition.

\section{An overview of the proof}

By our prior results in \cite{HIT}, \cite{AIT}, the water wave system \eqref{nlin}
expressed in holomorphic coordinates is locally well-posed in the space 
\[
(W,Q) \in \dH^\frac14, \qquad (\W,R) \in \dH^{\sigma-1}.
\]
The objective of the proof is to use a continuity argument to extend these 
local solutions to global in time solutions, by 
simultaneously tracking the Sobolev $\sWH$ norm and the uniform $X$ norm of the solutions.

Our energy estimates are based on \cite{AIT}, where we construct
cubic energy functionals equivalent to $\|  (\W,R) \|_{\dH^{s}}$ for
all  $s \geq 0$. Unfortunately, in \cite{AIT} there is no cubic energy
estimate at the level of $\|(W,Q)\|_{\dH^\frac14}$, so we need to do
this here. Our remedy is to use instead the cubic $\dH^\frac14$ energy estimates
proved in \cite{AIT} for the linearized equation.  To make such an
argument possible, we will work with a one parameter family of
solutions instead of a single solution. Precisely, for $h \in [0,1]$,
we consider the family of initial data
\[
(W^h_0,Q^h_0) = h (W_0,Q_0)
\]
and the corresponding solutions $(W^h,Q^h)$, and we will
simultaneously track the energy and the pointwise size for the entire
family of solutions. To avoid cumbersome notations, we will omit the index $h$
for the rest of the paper. The $h$ dependence will be important, and
indeed, critically used in a single place in the paper, namely in the
proof of Proposition~\ref{p:ee-low}.

In order for us to be able to provide a modular proof, it is convenient
to make in the beginning the following bootstrap assumption in a time interval $[0,T]$,
\begin{equation}\label{pointwise-bootstrap}
\|(\W, R)(t)\|_X \leq C\eps \langle t \rangle^{-1/2}, \qquad |t| \leq T.
\end{equation}
This will be assumed to hold uniformly for $h \in [0,1]$. Then the main steps of our argument are as follows:

\bigskip

\emph{1. Energy estimates.} Using the bootstrap assumption, as well as the balanced energy estimates of \cite{AIT} (recalled here in Theorems \ref{t:ee-s} and \ref{t:balancedenergy}) we obtain the energy estimates with a slight growth
\begin{equation}\label{ee-global}
\|(W, Q)(t) \|_{\sWH}   \lesssim \epsilon \langle t \rangle^{C\epsilon^2}.
\end{equation}
Here, the notation $\lesssim$ indicates a universal implicit constant, which in particular does not depend on $C$ in \eqref{pointwise-bootstrap}.
This is done in Section~\ref{s:ee}.

\bigskip

\emph{2. Normal form reduction.}  In Section~\ref{s:nf}, we apply a partial normal form reduction, whose primary goal is to eliminate the balanced quadratic interactions from the equations.   Using a partial normal form transformation, the variables $(W,Q)$ are replaced by normal form alternates $(\tW,\tQ)$, for which we obtain an equation with paradifferential quadratic terms and full cubic terms, modulo quartic error terms; see Proposition~\ref{p:gk-est}.
    
We  re-express the bounds \eqref{pointwise-bootstrap} and \eqref{ee-global} 
in a paradifferential fashion in terms of $(\tW,\tQ)$,
\begin{equation}\label{pointwise-bootstrap-nf}
\|(\tW_\alpha, \tQ_\alpha)(t)\|_X \approx \|(\W, R)(t)\|_X \lesssim C \eps \langle t \rangle^{-1/2}, \qquad |t| \leq T ,\end{equation}
respectively
\begin{equation}\label{ee-global-nf}
\|(\tW, \tQ)(t) \|_{\tWH}  \lesssim \|(W, Q)(t) \|_{\sWH} \lesssim \epsilon \langle t \rangle^{C\epsilon^2}.
\end{equation}
Here the nonlinear $\sWH$ energy functional is replaced
by a linear counterpart $\tWH$, defined later in \eqref{def-tWH}. This step is carried out in Section~\ref{s:nf}.
At the conclusion of this step,
the problem has been reduced to the study of the evolution of the normal form variables
$(\tW,\tQ)$, for which we need to improve the counterpart of the bootstrap assumption
\eqref{pointwise-bootstrap}, and show that
\begin{equation}\label{pointwise-nf}
\|(\tW_\alpha, \tQ_\alpha)(t)\|_X \lesssim  \eps \langle t \rangle^{-1/2}, \qquad |t| \leq T. \end{equation}

\bigskip

\emph{3. Nonlinear vector field Sobolev inequalities.} The goal in Section~\ref{s:wp} is to derive a preliminary pointwise bound 
for the normal form variables $(\tW,\tQ)$ starting from the weighted Sobolev bound $\tWH$ in \eqref{ee-global-nf}. Precisely, under the same bootstrap bound \eqref{ee-global-nf} we show that
\begin{equation}\label{KS-intro}
\|(\tW_\alpha, \tQ_\alpha)(t) \|_{X^\sharp}   \lesssim  \|(\tW, \tQ)(t) \|_{\tWH}.
\end{equation}
Here the $X^\sharp$ norm, defined later in \eqref{def:Xsharp}, is a microlocal improvement of the $X$ norm, which provides an additional frequency gain away from waves  of frequency $1$ which propagate with unit speed. Precisely, the $X^\sharp$ norm is stronger than the $X$ norm in two ways:
\begin{itemize}
    \item It has additional gains away from the frequency $1$.
    \item It has additional gains away from the hyperbolic region.
\end{itemize}
One could think of this akin to Sobolev embeddings, with the key caveat that 
the norm $\tWH$ is not a classical, elliptic norm, and instead has a ``hyperbolic''
component in a certain subset of the phase space.

We interpret the estimate \eqref{KS-intro} as a linear paradifferential estimate, which generalizes in a \emph{nonperturbative} fashion a corresponding linear vector field Sobolev bound in \cite{HIT}. We remark that the idea of replacing linear  bounds with 
more robust (though also more difficult to prove) nonlinear 
vector field Sobolev bounds
was first introduced by the last two authors in the Benjamin-Ono context 
in \cite{BO}.

\bigskip 

\emph{4. Pointwise bounds via wave packet testing.}  In this final step in Section~\ref{s:wp2}, we use the method 
of wave packet testing (see \cite{NLS}, \cite{IT-global}) to propagate sharp pointwise bounds along rays, in order to 
prove the desired pointwise bound \eqref{pointwise-nf} and close the argument. By virtue of the $X^\sharp$ bound, this is needed only in a time dependent range of frequencies around frequency $1$.

In a nutshell, the idea is to use well-chosen wave packets in order 
to define a good \emph{asymptotic profile} $\gamma(t,v)$ which describes
the leading order evolution of the solution at infinity along rays $x = vt$, and then to show that $\gamma$ is an approximate solution 
to an appropriate \emph{asymptotic equation}.

\section{The energy estimates} 
\label{s:ee}
Our goal here is to recall first the energy estimates of \cite{HIT} and \cite{AIT},
and then to use them to prove the bound \eqref{ee-global}, which contains the 
global energy bounds for our time dependent weighted norm $\sWH$.

The energy estimates for the solutions in both \cite{HIT} and \cite{AIT}
are described in terms of the (time dependent) uniform control norms.
The two control norms in \cite{HIT}, denoted by $A$ and $B$, and redenoted by 
$A_0$ and $A_{1/2}$ in \cite{AIT}, are
\begin{equation}\label{A-def}
A_0 = A := \|\W\|_{L^\infty}+\| Y\|_{L^\infty} + \||D|^\frac12 R\|_{L^\infty \cap B^{0}_{\infty, 2}},
\end{equation}
respectively
\begin{equation}\label{B-def}
A_{\frac12} =  B :=\||D|^\frac12 \W\|_{BMO} + \| R_\alpha\|_{BMO}.
\end{equation}
By contrast, in \cite{AIT} the leading role was played  by
an intermediate control norm interpolating between $A_0$
and $A_{\frac12}$,
\begin{equation}\label{A14-def}
A_{\frac14} :=\| \W\|_{\dot B^{\frac14}_{\infty,2}} + \| R \|_{\dot B^{\frac34}_{\infty,2}}.
\end{equation}
Here the subscript of $A$ represents the difference in terms of derivatives between our control norm and scaling.
In particular  $A_{s}$ corresponds to and is controlled by the homogeneous $\dH^{\frac12+s}$ norm of $(\W, R)$, and $A_0$ is a scale invariant quantity. Concerning $A_{\frac14}$, we note the following inequality, 
\begin{equation}\label{A14-def+}
\||D|^\frac14 \W\|_{BMO} + \| |D|^\frac34 R \|_{BMO} \lesssim A_{\frac14}.
\end{equation}

In addition to the pointwise scale invariant norm measured by $A$,
we will also need a  secondary stronger scale invariant Sobolev control norm $\As$ defined by 
\begin{equation}\label{Asharp-def}
\As :=\|D^\frac14 \W\|_{L^4} + \| D^\frac34 R\|_{L^4}.
\end{equation}
In \cite{AIT} this is used to control implicit constants in 
some of the energy estimates.

We now recall from \cite{AIT} the balanced cubic  energy estimates.
We begin with  the  full differentiated system \eqref{ww2d-diff}:

\begin{theorem}\label{t:ee-s}
For each $s \geq 0$ there exists an energy functional $E_s$
associated to the differentiated equation \eqref{ww2d-diff}
with the following two properties:

(i) Energy equivalence if $A \ll 1$:
\begin{equation}
E_s(\W,R) \approx \|(\W,R)\|_{\dH^s}^2   , 
\end{equation}

(ii) Balanced cubic energy bound:
\begin{equation}
\frac{d}{dt} E_s(\W,R) \lesssim_A A_{\frac14}^2 \|(\W,R)\|_{\dH^s}^2.  
\end{equation}
\end{theorem}
 Here the notation $\lesssim_A$ indicates that 
the implicit constant is allowed to depend on $A$;
this has no impact here, as $A$ will be shown to stay small for all solutions we work with.
We continue with the bounds for the linearized system \eqref{lin(wr)}, respectively:

\begin{theorem}\label{t:balancedenergy}
Assume $A \lesssim 1$ and $A_\frac14 \in L^2$. Then the linearized equation~\eqref{lin(wr)} is well-posed in $\dH^{\frac14}$. Furthermore, there 
exists an energy functional $\Elin^\frac14(w,r)$ so that we have

a) Norm equivalence:
\[
\Elin^\frac14(w,r) \approx_{\As} \| (w,r)\|_{\dH^{\frac14}}^2,
\]

b) Energy estimates:
\[
\frac{d}{dt} \Elin^\frac14(w,r) \lesssim_{\As} A_\frac14^2 \|(w,r)\|_{\dH^\frac14}^2.
\]
\end{theorem}

In the present work,  we use both theorems above combined with the pointwise bootstrap assumption \eqref{pointwise-bootstrap} in order to establish the energy estimate \eqref{ee}.

\begin{theorem}\label{t:wh-ee}
Assume that in a time interval $[-T, T]$ we have a solution $(W, Q)$ to \eqref{nlin} with small energy
\[
\|(W, Q)(0)\|_{\sWH} \leq \eps \ll 1,
\]
and satisfying the pointwise bootstrap estimate \eqref{pointwise-bootstrap}. Then we have the energy estimate
\begin{equation}\label{pre-energy-est}
\|(W, Q)(t)\|_{\sWH} \lesssim \eps \langle t \rangle^{C_1 \eps^2}, \qquad t \in [-T, T]
\end{equation}
for some universal $C_1 \gg C^2$.
\end{theorem}

\begin{proof}

a) For the $\dH^{\sigma - 1}$ bound for $(\W,R)$ we use the energy estimates in Theorem~\ref{t:ee-s} with $s = \sigma - 1$. The same energy estimates can be applied with $s=0$, which yields 
the bound 
\[
\|(\W,R)(t)\|_{\H^0} \lesssim \epsilon \langle t \rangle^{C_1 \epsilon^2}.
\]
Interpolating these energy estimates with the pointwise bootstrap bound \eqref{pointwise-bootstrap},
using suitably chosen intermediate norms in both cases, this gives a bound for the $A^\sharp$
control norm, e.g. by writing
\[
A^\sharp \lesssim 
(A \|(W,R)\|_{\dH^\frac12})^\frac12 \lesssim   \
C^\frac12 \epsilon \langle t \rangle^{C_1 \epsilon^2-\frac14}.
\]
For small enough $\epsilon$ this yields in particular the uniform in time smallness
\begin{equation}\label{Asharp-small}
A^\sharp \ll 1.     
\end{equation}
This will  be needed in parts (b), (c) below in order to control the implicit constants 
in the energy estimates for the linearized equation.

\bigskip

b) For the energy bounds in $\sWH$ on $(\w, \r) = \AA\S(W, Q)$, it suffices to use the balanced cubic energy estimates for the linearized equations in Theorem~\ref{t:balancedenergy}. Using the Gronwall inequality, the pointwise bootstrap assumption \eqref{pointwise-bootstrap}, together with \eqref{Asharp-small}, we have
\[
\|(\w, \r)(t) \|_{\dH^{1/4}} \lesssim e^{\int_0^t A_{1/4}^2 \, ds} \|(\w, \r)(0)\|_{\dH^{1/4}} \lesssim e^{\int_0^t C^2\eps^2\langle s \rangle^{-1} \, ds} \|(\w, \r)(0)\|_{\dH^{1/4}} \lesssim \eps e^{C^2 \eps^2 \log t}.
\]

\bigskip

c) The energy bounds in $\dH^\frac14$ do not follow from Theorem~\ref{t:ee-s}, 
 so we need to prove them here. We will show that
\begin{proposition}\label{p:ee-low}
Assume the bootstrap bound \eqref{pointwise-bootstrap} holds. Then we have the estimate
\begin{equation}
\| (W,Q)(t) \|_{\dH^\frac14} \lesssim \langle t \rangle^{C\epsilon^2} \|(W,Q)(0)\|_{\dH^\frac14}.   
\end{equation}
\end{proposition}

\begin{proof}
This is the only place in the article where we use the fact that we
work with a one parameter family of solutions depending on the
parameter $h \in [0,1]$. We will denote
\[
(w,q) = \frac{d}{dh}(W,Q), \qquad r = q - Rw. 
\]  
Then for each $h$, $(w,q)$ solves the linearized equation around
$(W,Q)$, and in particular we can apply the energy estimates
in Theorem~\ref{t:balancedenergy}, which, in view of our bootstrap
assumption \eqref{pointwise-bootstrap}, yield the estimate
\begin{equation}\label{ee-low}
\| (w,r)(t) \|_{\dH^\frac14} \lesssim \langle t \rangle^{C\epsilon^2}\|(w,r)(0)\|_{\dH^\frac14},
\qquad t \in [0,T],\quad h \in [0,1].
\end{equation}  

Our task is now to first estimate the initial data $(w,r)(0)$, and
show that
\begin{equation}\label{low-data}
  \|(w,r)(0)\|_{\dH^\frac14} \lesssim_A \|(W_0,Q_0)\|_{\dH^\frac14}.
\end{equation}
Secondly, we want the reverse estimate at times $t \in [0,T]$,
\begin{equation}\label{low-sln}
  \sup_{h \in [0,1]}  \|(W,Q)(t)\|_{\dH^\frac14} \lesssim_A \sup_{h
    \in [0,1]} \|(w,r)\|_{\dH^\frac14}.
\end{equation}
Together with \eqref{ee-low}, these two bounds imply the conclusion of
the proposition.

\bigskip
\emph{ Proof of \eqref{low-data}.} We have 
\[
(w,r)(0) = (W_0, Q_0- R_0 W_0).
\]
Hence the only nontrivial expression to estimate is $R_0 W_0$, for
which we use Coifman-Meyer type estimates to write
\[
\| R_0 W_0 \|_{\dot H^\frac34} \lesssim \| W_0\|_{L^4} \| |D|^\frac34
R_0 \|_{L^4} + \| R_0 \|_{L^3} \| |D|^ \frac34  W_0\|_{L^6}  ,
\]  
where the first term accounts for the high-low interactions, the
second for the low-high interactions, and the balanced interactions
can go either way. 

The first term is estimated directly using a Sobolev embedding,
\[
\| W_0\|_{L^4} \| |D|^\frac34 R_0 \|_{L^4}
\lesssim \| W_0\|_{\dot H^\frac14} \| |D|^\frac34 R_0 \|_{L^4} \lesssim A^\sharp \| (W_0,Q_0)\|_{\dH^\frac14}.
\]
For the second term we use interpolation instead,
\[
\|R_0\|_{L^3} \lesssim \|R_0\|_{\dot H^{-\frac14}}^\frac23
\| |D|^\frac12 R_0\|_{BMO}^\frac13
\lesssim \|R_0\|_{\dot H^{-\frac14}}^\frac23
\| |D|^\frac34 R_0\|_{L^4}^\frac13,
\]
respectively
\[
\| |D|^ \frac34  W_0\|_{L^6} \lesssim \|W\|_{\dot{H^\frac14}}^\frac13 \| \partial_\alpha W \|_{BMO}^\frac23
\lesssim \|W\|_{\dot{H^\frac14}}^\frac13 \| |D|^\frac14 \partial_\alpha W \|_{L^4}^\frac23.
\]

Combining the two and assuming the equivalence (to be proved shortly)
\begin{equation}\label{R-like-Q}
\|R\|_{\dot H^{-\frac14}} \approx_{A,A^\sharp} \|Q_\alpha\|_{\dot H^{-\frac14}},
\end{equation}
we obtain
\[
\| R_0 W_0 \|_{\dot H^\frac34} \lesssim A^\sharp \| (W_0,Q_0)\|_{\dH^\frac14},
\]
which suffices.

\bigskip

\emph{ Proof of \eqref{R-like-Q}.} We have the relations $R=(1-Y)Q_\alpha$ and 
$Q_\alpha = (1+W_\alpha) R$, where 
\[
\| W_\alpha\|_{W^{\frac14,4} \cap L^\infty} \lesssim A+\As \ll 1, 
\]
and, by the algebra property for the space $W^{\frac14,4} \cap L^\infty$,
\[
\| Y \|_{W^{\frac14,4} \cap L^\infty} \lesssim A+\As \ll 1.
\]
Hence, using also duality, it remains to show that we have a bound of the form
\[
\| Yf \|_{\dot H^\frac14} \lesssim \| Y \|_{W^{\frac14,4} \cap L^\infty} 
\| f\|_{\dot H^\frac14}.
\]
But this is a standard multiplicative estimate, which is left for the reader.

\bigskip
\emph{ Proof of \eqref{low-sln}.}
There is nothing to do for $W$, since it is the antiderivative of $w$,
\[
W(h) = \int_0^h w(h_1)\, dh_1.
\]  

It remains to consider $Q$, where we write
\[
Q(h) = \int_0^h r(h_1) + (Rw)(h_1)\, dh_1 .
 \] 
The first term is straightforward, but we still need to estimate the
second, where there is an apparent loss of derivatives. To rectify this
we replace $Q$ by $Q - T_R W$, which is akin to a good variable.
Computing
\[
\frac{d}{dh} R = (q_\alpha - w_\alpha R)(1-Y)= r_\alpha(1-Y) + w R_\alpha (1-Y)
 \]  
we see that
\[
(Q-T_R W)(h) = \int_0^h r + T_w R + \Pi(w,R) - T_{r_\alpha(1-Y) + w R_\alpha(1-Y)} W \,dh_1 .
\]
The second term on the left plays a perturbative role, in view of the bound
\[
\| T_R W \|_{\dH^\frac34} \lesssim_A A^\sharp \| Q\|_{\dot H^\frac34} \ll \| Q\|_{\dot H^\frac34},
\]  
where $R$ is related to $Q_\alpha$ via \eqref{R-like-Q}.

Hence it remains to estimate the nonlinear terms under the integral in $\dot H^\frac34$.
For the first two we have a Coifman-Meyer type bound
\[
\|T_w R\|_{\dH^\frac34} +\| \Pi(w,R)\|_{\dH^\frac34} \lesssim A^\sharp
\| w \|_{\dH^\frac14}.
\]  
This leaves us with the last one. There, all frequencies in the factors 
of the para-coefficient are negative so must be smaller in size than the frequency of $W$.
Hence we can bound the full expression as
 \[
 \begin{aligned}
\|T_{ r_\alpha(1-Y) + w R_\alpha(1-Y)}  W\|_{\dot H^\frac34}  &\lesssim  \ 
\| |D|^{\frac14} W_\alpha\|_{L^4}
\| |D|^{-\frac12} ( r_\alpha(1-Y) + w R_\alpha(1-Y))\|_{L^4})
\\
&\lesssim  \ A^\sharp (\| |D|^\frac12 r\|_{L^4} \|1-Y\|_{L^\infty} + \| w \|_{L^4} \||D|^\frac12 R\|_{L^\infty}\|1-Y\|_{L^\infty} )
\\
& \lesssim_A  \ A^\sharp \|(w,r)\|_{\dH^\frac14}.
\end{aligned}
 \]  
Here, in estimating the parafactor $r_\alpha(1-Y) + w R_\alpha(1-Y)$ we took advantage of the fact that all factors
 are holomorphic, so the $|D|^{-\frac12}$ operator always acts at the highest frequency.
 This concludes the proof of \eqref{low-sln} and thus the proof of the proposition.
 \end{proof}

Finally, Proposition~\ref{p:ee-low} completes the proof of Theorem~\ref{t:wh-ee}.

\end{proof}

\section{The paradifferential normal form}
\label{s:nf}

We begin by recalling from \cite{HIT} the classical normal form variables,
\begin{equation}\label{classicalNF}
\left\{
\begin{aligned}
\tW &= W - P[2\Re W \cdot W_\alpha], \\
\tQ &= Q - P[2\Re W \cdot R],
\end{aligned}
\right.
\end{equation}
which solve an equation of the form 
\begin{equation}\label{twq-eqn}
\left\{
\begin{aligned}
\tW_t + \tQ_\alpha &= \tG, \\
\tQ_t - i\tW &= \tK,
\end{aligned}
\right.
\end{equation}
with sources $(\tG, \tK)$ which contain only cubic and higher order terms.

We also recall the linear scaling operator
\[
\tS_0(\tW,\tQ) := (2 \alpha \D_\alpha \tW - t\D_\alpha \tQ, 2 \alpha \D_\alpha \tQ + it\tW ),
\]
which was also used in \cite{HIT} as the main vector field at the level of the normal form
variables.

Using \eqref{twq-eqn}, this can be expressed 
in terms of the scaling vector field $S$ as follows:
\begin{equation}\label{ts0-eqn}
\begin{aligned}
\tS_0(\tW,\tQ) &= (S\tW, S\tQ) - t(\tG, \tK).
\end{aligned}
\end{equation}

\bigskip

In this paper, we will instead use a paradifferential substitute of the normal form \eqref{classicalNF}, defining
\begin{equation}\label{tWQ}
\left\{
\begin{aligned}
&\tW = W - T_{W_\alpha}W - \Pi(W_\alpha, 2\Re W), \\
&\tQ = Q - T_R W - \Pi(R, 2\Re W),
\end{aligned}
\right.
\end{equation}
where here and throughout, we let both paradifferential operators $T$ and $\Pi$ include an implicit projection $P$, so that $T=PT$ and $\Pi = P\Pi$. This somewhat unusual convention is motivated by the fact that our 
flow evolves in spaces of holomorphic functions.

This can no longer be seen as a full normal form 
transformation, but, instead, only as a partial 
normal form. This idea was introduced by Alazard-Delort in \cite{AD},
in the context of the Eulerian formulation of the equations.

Here and throughout, we fix a self-adjoint quantization for the paraproduct operator $T$
viewed as a pseudodifferential operator.
For instance, we may use the Weyl quantization, or simply the average 
\[
\half (T + T^*).
\]
In the following sections, we will use several classical multilinear estimates for $T$ and $\Pi$. We refer the reader to Appendix B in \cite{HIT} and Section 2 in \cite{AIT} for such estimates.

Our objectives in this section are as follows:
\bigskip

(i) To transfer the $\H^s$
bounds from $(W,Q)$ to $(\tW,\tQ)$:

\begin{proposition}\label{p:energy-equiv}
Assume \eqref{pointwise-bootstrap}. Then we have 
\begin{equation}\label{nf-energy-est}
\|(\tW_\alpha, \tQ_\alpha) \|_{\dH^{\sigma - 1}} \approx \|(\W, R) \|_{\dH^{\sigma - 1}}
\end{equation}
as well as 
\begin{equation}\label{nf-energy-est-lo}
\|(\tW, \tQ) \|_{\dH^{\frac14}} \approx \|(W, Q) \|_{\dH^{\frac14}}.
\end{equation}
\end{proposition}

We also prove the similar bound associated to the scaling vector field $S$:

\begin{proposition}\label{p:SWQ}
Assume \eqref{pointwise-bootstrap}. Then we have 
\begin{equation}\label{vf-energy-est}
 \|(S\tW, S\tQ)\|_{\dH^{\frac14}} \lesssim \| (\w, \r)\|_{\dH^{\frac14}}
 + \|(W,Q) \|_{\dH^{\frac14}},
 \end{equation}
 with $(\w,\r)$ as in \eqref{def:bwr}.
\end{proposition}

Given the two propositions above, it is natural to define the linear energy 
functional of $(\tW,\tQ)$ as 
\begin{equation}\label{def-tWH}
\|(\tW,\tQ)\|_{\tWH}^2 =    \|(\tW, \tQ) \|_{\dH^{\frac14}\cap \dH^\sigma}^2
+ \|(S\tW, S\tQ)\|_{\dH^{\frac14}}^2.
\end{equation}
Then, as a consequence of the last two propositions, it follows that 
the energy bound \eqref{ee-global} for the original variables $(W,Q)$ implies the corresponding bound \eqref{ee-global-nf} for the normal form variables $(\tW,\tQ)$.

\bigskip 

(ii) To allow the transfer of the pointwise bounds between $(W,Q)$ and
$(\tW,\tQ)$:

\begin{proposition}\label{p:nf-pointwise-est}
Assume \eqref{pointwise-bootstrap}. Then
\begin{equation}
\|(\tW_\alpha - \W, \tQ_\alpha - R)\|_X \lesssim \epsilon \langle t \rangle^{-1/2} \|(\W, R)\|_{X}.
\end{equation}
\end{proposition}

On one hand, this bound allows us to transfer the bootstrap assumption
to $(\tW,\tQ)$, i.e. show that \eqref{pointwise-bootstrap} implies
\eqref{pointwise-bootstrap-nf}.

On the other hand, it shows that it suffices to improve the bootstrap condition for $(\tW,\tQ)$, i.e. prove \eqref{pointwise-nf}; this in turn
implies a similar improvement for \eqref{pointwise-bootstrap}.

\bigskip

(iii) To compute the paradifferential equation \eqref{twq-eqn} for $(\tW,\tQ)$, which is written in the form
\begin{equation}\label{tWQ-system}
\left\{
\begin{aligned}
& \tW_t + \tQ_\alpha + T_{2 \Re \tW_\alpha} \tQ_\alpha - T_{2 \Re \tQ_\alpha} \tW_\alpha =  \tG,
    \\
& \tQ_t - i \tW - T_{2\Re  \tQ_\alpha} \tQ_\alpha =  \tK,
\end{aligned}
\right.
\end{equation}
with a good description of the source terms $(\tG,\tK)$.
 Precisely, we identify the leading order cubic terms in $(\tG, \tK)$, while proving improved decay estimates for the remaining quartic terms:

\begin{proposition}\label{p:gk-est}
Assume that \eqref{pointwise-bootstrap} and \eqref{ee-global} hold.
Then $(\tW,\tQ)$ solve an equation of the form \eqref{tWQ-system},
where the source terms $(\tG,\tK)$ satisfy the bound
\begin{equation}\label{tGK-est}
 \|(\tG, \tK)\|_{\dH^{\frac14}} \lesssim_{A_0} 
 \|(\W, R)\|_{X}^2\|(\tW, \tQ)\|_{\dH^\frac14 \cap \dH^\sigma}.
\end{equation}
Furthermore, $(\tG,\tK)$ can be split into 
\[
(\tG,\tK) = (\tG^{(3)},\tK^{(3)}) + (\tG^{(4+)},\tK^{(4+)}),
\]
where $(\tG^{(3)},\tK^{(3)})$ are explicit cubic expressions in $(\tW,\tQ)$
given by \eqref{tGK3} which satisfy the bound
\begin{equation}\label{tGK-est+}
 \|(\tG^{(3)}, \tK^{(3)})\|_{\dH^{\frac14}} \lesssim_{A_0}\|(\W, R)\|_{X}^2\|(\tW, \tQ)\|_{\dH^\frac14 \cap \dH^\sigma},
\end{equation}
while $(\tG^{(4+)},\tK^{(4+)})$ are quartic and higher 
order expressions which satisfy the better bound
\begin{equation}\label{tGK4-est}
\|(\tG^{(4+)}, \tK^{(4+)})\|_{\dH^{\frac14}} \lesssim \|(\W, R)\|_{X}^3\|(\tW, \tQ)\|_{\dH^\frac14 \cap \dH^\sigma}.
\end{equation}
\end{proposition}
We remark that, in view of the equations \eqref{tWQ-system}, it is natural 
to replace the linear scaling operator $\tS_0$ used in \cite{HIT} via 
\eqref{ts0-eqn} with a nonlinear, paradifferential counterpart
\[
\tS(\tW,\tQ) := (2 \alpha \D_\alpha \tW - t\D_\alpha \tQ + t(T_{2 \Re \tW_\alpha} \tQ_\alpha - T_{2 \Re \tQ_\alpha} \tW_\alpha), 2 \alpha \D_\alpha \tQ + it\tW - tT_{2\Re \tQ_\alpha} \tQ_\alpha),
\]
so that we have
\begin{equation}\label{ts-eqn}
\begin{aligned}
\tS(\tW,\tQ) &= (S\tW, S\tQ) - t(\tG, \tK).
\end{aligned}
\end{equation}
This system of equations for $(\tW,\tQ)$ will play a key role in the next section.
\

We record the explicit cubic expressions for $(\tG^{(3)}, \tK^{(3)})$ below. Here we 
have partitioned the terms of $\tG^{(3)}$ into three components: 
\begin{itemize}
\item $\tG_1^{(3)}$ corresponds to the perturbative cubic terms that arise from the time differentiation $\D_t \tW$,
\item $\tG_2^{(3)}$ corresponds to additional cubic terms that arise after cancellations with $ \tQ_\alpha$,
\item $\tG_3^{(3)}$ corresponds to cubic terms that arise from rewriting the remaining quadratic expressions in terms of the normal form variables.
\end{itemize}
The partition of $\tK^{(3)}$ is similar. We remark that these decompositions are 
consistent with the computations later in this section, but not so much with the resonant/nonresonant/null decomposition in Section~\ref{s:wp}; for this reason, in 
Section~\ref{s:wp} we reorganize them in a more useful way.

It is convenient to denote the quadratic component of $F$, rewritten in terms of the normal form variables $(\tW_\alpha, \tQ_\alpha)$, as follows:
\[
\tilde F^{(2)} = P\left[ \bar \tQ_\alpha \tW_\alpha -\tQ_\alpha \bar \tW_\alpha\right].
\]  
Then we have
\begin{equation}\label{tGK3}
\begin{aligned}
\tG^{(3)} &= \tG_1^{(3)} + \tG_2^{(3)} + \tG_3^{(3)}, \\
\tilde G_1^{(3)} &= T_{\tW_\alpha} (\tQ_\alpha \tW_\alpha) + T_{(\tQ_\alpha \tW_\alpha)_\alpha}\tW + \Pi(\tW_\alpha, 2\Re [\tQ_\alpha \tW_\alpha]) + \Pi((\tQ_\alpha \tW_\alpha)_\alpha, 2\Re \tW), \\
\tilde G_2^{(3)} &=- \tW_\alpha \tilde F^{(2)} + T_{  \tilde F^{(2)}_\alpha}\tW 
+ \Pi( \D_\alpha \tilde F^{(2)}, 2\Re W) + \Pi(\tilde F^{(2)}, W_\alpha) +  \Pi(\tW_\alpha, \bar {\tilde F}^{(2)}) \\ 
&\quad - \Pi(\bar \tW_\alpha^2, \tQ_\alpha) + \Pi(\bar \tQ_\alpha, \tW_\alpha^2) 
- T_{\bar \tW_\alpha^2}\tQ_\alpha - T_{\bar \tW_\alpha} \tilde F^{(2)} + T_{\bar \tQ_\alpha}\tW_\alpha^2, \\
\tilde G_3^{(3)} &= T_{2 \Re (T_{\tW_\alpha}\tW + \Pi(\tW_\alpha, 2\Re \tW))_\alpha} \tQ_\alpha \\
&\quad + T_{2 \Re \tW_\alpha} ( - \tQ_\alpha \tW_\alpha + \tilde F^{(2)}) + T_{2 \Re \tW_\alpha}(T_{\tQ_\alpha} \tW + \Pi(\tQ_\alpha, 2\Re \tW))_\alpha \\
&\quad + T_{2 \Re (\tQ_\alpha \tW_\alpha - (T_{\tQ_\alpha} \tW + \Pi(\tQ_\alpha, 2\Re \tW))_\alpha)}\tW_\alpha - T_{2 \Re \tQ_\alpha} ( T_{\tW_\alpha}\tW + \Pi(\tW_\alpha, 2\Re \tW))_\alpha,\\
\tK^{(3)} &= \tK_1^{(3)} + \tK_2^{(3)} + \tK_3^{(3)}, \\
\tilde K_1^{(3)} &= T_{ (\frac12\tQ_\alpha^2   +  P[|\tQ_\alpha|^2])_\alpha} \tW + T_{\tQ_\alpha} (T_{\tW_\alpha} \tQ_\alpha + \Pi(\tW_\alpha, \tQ_\alpha)) \\
&\quad + \Pi((\frac12\tQ_\alpha^2   +  P[|\tQ_\alpha|^2])_\alpha, 2\Re \tW) +  \Pi(\tQ_\alpha, 2\Re [\tQ_\alpha \tW_\alpha]) \\
&\quad - \Pi(\tW_\alpha \tQ_\alpha, \tQ_\alpha) + \Pi(\tQ_\alpha, \bar {\tilde F}^{(2)}) - T_{\tQ_\alpha \tW_\alpha} \tQ_\alpha, \\
\tilde K_2^{(3)} &=  i T_{\tW_\alpha^2} \tW  + i\Pi( \tW_\alpha^2, 2\Re \tW) - T_{P\left[ \bar \tQ_\alpha \tW_\alpha - \tQ_\alpha \bar \tW_\alpha \right]} \tQ_\alpha, \\
 \tilde K_3^{(3)} &=  - T_{2\Re (T_{\tQ_\alpha} \tW + \Pi(\tQ_\alpha, 2\Re\tW))_\alpha}\tQ_\alpha - T_{2\Re \tQ_\alpha} (T_{\tQ_\alpha} \tW + \Pi(\tQ_\alpha, 2\Re \tW))_\alpha \\
&\quad + T_{2\Re (\tQ_\alpha \tW_\alpha)} \tQ_\alpha + T_{\bar \tQ_\alpha} (\tQ_\alpha \tW_\alpha) + T_{\tQ_\alpha} T_{\tQ_\alpha}\tW_\alpha.
\end{aligned}
\end{equation}

\bigskip

\subsection{The paradifferential equation}

We begin by computing the cubic and higher order perturbative source terms $(\tG, \tK)$ for the paradifferential normal form variables \eqref{tWQ} above, simultaneously identifying the paradifferential quadratic potentials that we will collect on the left hand side in the equation \eqref{tWQ-system}. First we compute the time derivative, 
\begin{equation*}
\begin{aligned}
\D_t \tW &= W_t - \D_t T_{W_\alpha} W - \D_t \Pi(W_\alpha, 2\Re W) \\
&= -F(1 + W_\alpha) \\
&\quad + T_{W_\alpha} (F(1 + W_\alpha)) + T_{(F(1 + W_\alpha))_\alpha}W \\
&\quad +  \Pi(W_\alpha, 2\Re [F(1 + W_\alpha)]) + \Pi((F(1 + W_\alpha))_\alpha, 2\Re W).
\end{aligned}
\end{equation*}
We collect the cubic terms below in $\tilde G_1$: 
\begin{equation*}
\begin{aligned}
\D_t \tW &= -F(1 + W_\alpha) + T_{W_\alpha} F + T_{F_\alpha}W + \Pi(W_\alpha, 2\Re F) + \Pi(F_\alpha, 2\Re W) + \tilde G_1
\end{aligned}
\end{equation*}
where
\begin{equation}\label{G1}
\begin{aligned}
\tilde G_1 &= T_{W_\alpha} (FW_\alpha) + T_{(FW_\alpha)_\alpha}W + \Pi(W_\alpha, 2\Re [FW_\alpha]) + \Pi((F W_\alpha)_\alpha, 2\Re W).
\end{aligned}
\end{equation}

On the other hand, we have
\begin{equation*}
\begin{aligned}
\D_\alpha \tQ &= Q_\alpha - \D_\alpha T_R W - \D_\alpha \Pi(R, 2\Re W) \\
&= R(1 + W_\alpha) - \D_\alpha T_R W - \D_\alpha \Pi(R, 2\Re W),
\end{aligned}
\end{equation*}
so that
\begin{equation*}
\begin{aligned}
\D_t \tW + \D_\alpha \tQ &= -F(1 + W_\alpha) + T_{W_\alpha} F + T_{F_\alpha}W + \Pi(W_\alpha, 2\Re F) + \Pi(F_\alpha, 2\Re W) + \tilde G_1 \\
&\quad + R(1 + W_\alpha) - \D_\alpha T_R W - \D_\alpha \Pi(R, 2\Re W) \\
&= T_{2 \Re W_\alpha} F - T_{2\Re R} W_\alpha + \tilde G_1 + \tilde G_2,
\end{aligned}
\end{equation*}
where $\tG_2$ contains only cubic and higher order terms,
\begin{equation}\label{G2}
\begin{aligned}
\tilde G_2 &= (R - F)W_\alpha + T_{(F - R)_\alpha}W + \Pi((F - R)_\alpha, 2\Re W) + \Pi(F - R, W_\alpha) \\
&\quad + \Pi(\bar Y - \bar W_\alpha, R) + P [\Pi(W_\alpha, \bar F) - \Pi(\bar R, Y)] \\
&\quad +( T_{\bar Y}R - T_{\bar W_\alpha} F) + (T_{\bar R} W_\alpha - T_{\bar R}Y).
\end{aligned}
\end{equation}

Lastly, we exchange the variables in the quadratic potentials for their normal form counterparts $(\tW, \tQ)$:
\begin{equation*}
\begin{aligned}
\D_t \tW + \D_\alpha \tQ &= T_{2 \Re \tW_\alpha} \tQ_\alpha - T_{2 \Re \tQ_\alpha} \tW_\alpha + \tilde G_1 + \tilde G_2 + \tilde G_3 =: T_{2 \Re \tW_\alpha}\tQ_\alpha - T_{2 \Re \tQ_\alpha} \tW_\alpha + \tG ,
\end{aligned}
\end{equation*}
where $\tG_3$ contains the cubic and higher terms in the difference,
\begin{equation}\label{G3}
\begin{aligned}
\tilde G_3 &= T_{2 \Re W_\alpha} F - T_{2 \Re \tW_\alpha} \tQ_\alpha + T_{2 \Re \tQ_\alpha}\tW_\alpha - T_{2 \Re R} W_\alpha. 
\end{aligned}
\end{equation}

\

For the second equation, we have
\begin{equation*}
\begin{aligned}
\D_t \tQ &= Q_t - \D_t T_R W - \D_t \Pi(R, 2\Re W) \\
&= iW - FQ_\alpha - P[|R|^2]  \\
&\quad + T_{bR_\alpha - i(1 + a)Y + ia} W + T_{R} [F(1 + W_\alpha)] \\
&\quad + \Pi(bR_\alpha - i(1 + a)Y + ia, 2\Re W) + \Pi(R, 2\Re [F(1 + W_\alpha)]).
\end{aligned}
\end{equation*}
We collect the cubic terms below in $\tilde K_1$: 
\begin{equation*}
\begin{aligned}
\D_t \tQ &= iW - T_FQ_\alpha  - T_{\bar R} R -i T_{Y} W +T_R T_F W_\alpha - i\Pi(Y, 2\Re W) + \tilde K_1,
\end{aligned}
\end{equation*}
where
\begin{equation}\label{K1}
\begin{aligned}
\tilde K_1 &= T_{bR_\alpha - iaY + ia} W + T_R (T_{W_\alpha} F + \Pi(W_\alpha, F)) \\
&\quad + \Pi(bR_\alpha - iaY + ia, 2\Re W) +  \Pi(R, 2\Re [FW_\alpha]) \\
&\quad +  [\Pi(R, 2\Re F)- \Pi(F, Q_\alpha) - \Pi(\bar R, R)] + T_{R - Q_\alpha} F.
\end{aligned}
\end{equation}
As a result, we have
\begin{equation*}
\begin{aligned}
\D_t \tQ - i\tW &= iW - T_FQ_\alpha  - T_{\bar R} R -i T_{Y} W +T_R T_F W_\alpha - i\Pi(Y, 2\Re W) + \tilde K_1 \\
&\quad - i(W - T_{W_\alpha}W - \Pi(W_\alpha, 2\Re W)) \\
&= - T_{\bar R} R - T_R (Q_\alpha - T_F W_\alpha) + \tilde K_1 + \tilde K_2,
\end{aligned}
\end{equation*}
where
\begin{equation}\label{K2}
\begin{aligned}
\tilde K_2 &=  i(T_{W_\alpha}W - T_{Y} W)  + i\Pi(W_\alpha - Y, 2\Re W) + T_{R - F} Q_\alpha.
\end{aligned}
\end{equation}

Lastly, we exchange the variables on the right hand side for their normal form counterparts $(\tW, \tQ)$:
\begin{equation*}
\begin{aligned}
\D_t \tQ - i\tW &= - T_{2\Re \tQ_\alpha} \tQ_\alpha+ \tK_1 + \tK_2 + \tK_3 := \tK,
&\quad
\end{aligned}
\end{equation*}
where
\begin{equation}\label{K3}
\begin{aligned}
 \tilde K_3 &=  T_{2\Re \tQ_\alpha}\tQ_\alpha - T_{\bar R} R - T_R (Q_\alpha - T_F W_\alpha).
\end{aligned}
\end{equation}

We conclude
\begin{equation}\label{paradiffeqn}
\left\{
\begin{aligned}
&\D_t \tW + \D_\alpha \tQ - T_{2 \Re \tW_\alpha} \tQ_\alpha + T_{2 \Re \tQ_\alpha} \tW_\alpha = \tG, \\
&\D_t \tQ - i\tW +  T_{2\Re \tQ_\alpha} \tQ_\alpha = \tK.
\end{aligned}
\right.
\end{equation}

\subsection{Energy bounds on the normal form}

Here we prove Proposition \ref{p:energy-equiv}, transferring the energy estimates from $(W, Q)$ to the normal form variables $(\tW, \tQ)$, together with Proposition \ref{p:SWQ}, which is concerned with estimates on $(S\tW, S\tQ)$.

\begin{proof}
For $\tW$, we estimate the quadratic corrections to $W$. We have
\begin{equation*}
\begin{aligned}
\|T_{W_\alpha}W + \Pi(W_\alpha, 2\Re W)\|_{\dot H^{\sigma}} \lesssim \|W_\alpha\|_{L^\infty} \|W_\alpha\|_{\dot H^{\sigma - 1}} 
\end{aligned}
\end{equation*}
so that using  \eqref{pointwise-bootstrap} suffices. The $\dot H^{\frac14}$ estimate is similar.
\smallskip

For $\tQ$, we first write
\begin{equation*}
\begin{aligned}
\tQ_\alpha &= R + RW_\alpha - \D_\alpha(T_{R} W + \Pi(R, 2\Re W)) \\
&= R + T_{W_\alpha}R + \Pi(R, W_\alpha) - T_{R_{\alpha}} W - \D_\alpha \Pi(R, 2\Re W).
\end{aligned}
\end{equation*}
We estimate the quadratic terms. First we have, using both \eqref{pre-energy-est} and \eqref{pointwise-bootstrap},
\[ \| T_{W_\alpha}R + \Pi(R, W_\alpha) -\Pi(R, 2\Re W_\alpha) \|_{\dot H^{\sigma - \half}} \lesssim \|W_\alpha\|_{L^\infty} \|R\|_{\dot H^{\sigma - \half}} \lesssim  \|(\W, R)\|_X \|(W, Q)\|_{\sWH}. \]
Similarly, we have 
\[\| T_{R_{\alpha}} W + \Pi(R_\alpha, 2\Re W)\|_{\dot H^{\sigma - \half}} \lesssim \||D|^{1/2} R\|_{L^\infty} \|W\|_{\dot H^{\sigma}} \lesssim  \|(\W, R)\|_X \|(W, Q)\|_{\sWH}. \]
For the $\dH^{\frac14}$ bound, we directly estimate
\[
\|T_R W +\Pi(R, 2\Re W)\|_{\dot H^{\frac34}} \lesssim \||D|^{-1/4} R\|_{L^2} \|W_\alpha\|_{L^\infty} \lesssim \|(W, Q)\|_{\sWH}\|(\W, R)\|_X,
\]
where at the second step we have used \eqref{R-like-Q}.

\smallskip

It remains to show the bounds for  $(S\tW, S\tQ)$. Writing
\[(\w, \r) = \AA \S(W, Q) = ((S - 2)W, (S - 3)Q - R(S - 2)W), \]
we have
\begin{equation*}
\begin{aligned}
S\tW &= \w - (T_{\w_\alpha}W + T_{W_\alpha} \w + \Pi(\w_\alpha, 2\Re W) +  \Pi(W_\alpha, 2\Re \w) \\
&\quad \phantom{w - (}  - \tilde T_{W_\alpha} W - P\tilde \Pi(W_\alpha, 2 \Re W)) \\
&\quad + 2W - (T_{2W_\alpha}W + T_{W_\alpha} 2W + \Pi(2W_\alpha, 2\Re W) +  \Pi(W_\alpha, 4\Re W)),
\end{aligned}
\end{equation*}
where the $\tilde T$ and $\tilde \Pi$ arise from commutators with $\alpha$,
\begin{equation*}
\begin{aligned}
\tilde T_{W_\alpha} W &= 2(T_{(\alpha W_\alpha)_\alpha} W + T_{W_\alpha} \alpha W_\alpha) - 2\alpha (T_{W_{\alpha\alpha}} W + T_{W_\alpha} W_\alpha), \\
\tilde\Pi(W_\alpha, 2\Re W) &= 2(\Pi((\alpha W_{\alpha})_\alpha, 2\Re W) + \Pi(W_\alpha, 2\Re (\alpha W_\alpha)) ) \\
&\quad- 2\alpha (\Pi(W_{\alpha\alpha}, 2\Re W) + \Pi(W_\alpha, 2\Re W_\alpha)).
\end{aligned}
\end{equation*}
Then it is straightforward to estimate each term on the right hand side in $\dot H^{1/4}$, combining \eqref{pre-energy-est} and \eqref{pointwise-bootstrap} as before. We remark that we gain a derivative in the commutator, so both of these terms are estimated 
in a scale invariant fashion by
\[
\|\tilde T_{W_\alpha} W\|_{\dot H^\frac14}
+ \|\tilde\Pi(W_\alpha, 2\Re W)\|_{\dot H^\frac14}
\lesssim \|W_\alpha \|_{L^\infty} \| W\|_{\dot H^\frac14}.
\]

\

For $S\tQ$, we write
\begin{equation}\label{SQeqn}
\begin{aligned}
S\tQ &= SQ - (T_{SR} W + T_R \w + \Pi(SR, 2\Re W) + \Pi(R, 2\Re \w) - \tilde T_R W - P\tilde \Pi(R, 2\Re W )) \\
&\quad - (T_R 2W + \Pi(R, 4\Re W)) \\
&= \r + T_\w R - (T_{SR} W + \Pi(SR, 2\Re W) + \Pi(R, \bar \w) - \tilde T_R W - P\tilde \Pi(R, 2\Re W )) \\
&\quad - (T_R 2W + \Pi(R, 4\Re W)) + 3Q - 2RW. 
\end{aligned}
\end{equation}

We estimate the unbalanced terms on the right hand side of \eqref{SQeqn}, as the corresponding balanced terms are similar. First, we have by Sobolev embeddings,
\[\|T_\w R\|_{\dot H^{3/4}} \lesssim \|\w\|_{L^4} \||D|^{3/4}R\|_{L^4} \lesssim_{\As} \|\w\|_{\dot H^{1/4}}. \]
Using \eqref{R-like-Q}, we also have
\[ 
\|T_R W \|_{\dot H^{3/4}} \lesssim \||D|^{-1/4}R\|_{L^2} \|W_\alpha \|_{L^\infty} \lesssim_{\As} \|Q\|_{\dot H^{3/4}}. 
\]

Next, we estimate the last two terms on the right hand side. The estimate for $Q$ is immediate. For $RW$, $T_RW$ and $\Pi(R,W)$ may be estimated as before, and we have
\[
\|T_W R\|_{\dot H^{3/4}} \lesssim \|W\|_{L^4} \||D|^{3/4}R\|_{L^4} \lesssim_{\As} \|W\|_{\dot H^{1/4}}.
\]

To estimate $T_{SR} W$, we use the identity
\[(S - 1)R = \frac{\r_\alpha + R_\alpha \w}{1 + W_\alpha}.\]
Consider the contribution $T_{(1 - Y)\r_\alpha} W$ from $\r_\alpha$. The cases of three unequal frequencies may be estimated as follows:
\begin{equation*}
\begin{aligned}
\|T_{T_{1 - Y} \r_\alpha} W \|_{\dot H^{3/4}} + \|T_{T_{\r_\alpha} Y} W \|_{\dot H^{3/4}} &\lesssim (1 + \|W_\alpha\|_{L^\infty})\|\r\|_{\dot H^{3/4}} \|W_\alpha\|_{L^\infty} \lesssim_A \|\r\|_{\dot H^{3/4}}.
\end{aligned}
\end{equation*}
Since $Y$ and $\r_\alpha$ are both holomorphic, there is no frequency cancellation in the balanced case $\Pi(Y, r_\alpha)$, and so this case may be treated in the same way.

For the contribution from $R_\alpha \w$, we have $T_{(1 - Y)R_\alpha \w} W$. We consider the cubic term; the quartic term is similar, measuring $Y \in L^\infty$ in all cases. We have
\begin{equation*}
\begin{aligned}
\|T_{T_{R_\alpha} \w} W \|_{\dot H^{3/4}} &\lesssim \||D|^{1/2} R\|_{L^\infty} \|\w\|_{\dot H^{1/4}} \|W_\alpha\|_{L^\infty} \lesssim_A  \|\w\|_{\dot H^{1/4}}, \\
\|T_{T_{\w} R_\alpha} W \|_{\dot H^{3/4}} &\lesssim \|\w\|_{L^4} \||D|^{1/2}R\|_{L^\infty} \||D|^{5/4} W\|_{L^4} \lesssim_{\As} \|\w\|_{\dot H^{1/4}}.
\end{aligned}
\end{equation*}
As with the $\r_\alpha$ contribution, the balanced frequency case has no cancellation and may be treated as either of these two cases.

\end{proof}

\subsection{Pointwise bounds on the normal form}

Here we prove Proposition \ref{p:nf-pointwise-est}, transferring pointwise estimates from $(\tW, \tQ)$ to the original variables $(W, Q)$.

\begin{proof}
Recall our objective is to show that
\[
\|(\tW_\alpha - \W, \tQ_\alpha - R)\|_X \lesssim \eps \langle t \rangle^{-1/2} \|(\W, R)\|_{X}.
\]
For the $\tW$ bound, we first consider the unbalanced paraproduct quadratic terms. For the high frequency estimate, we have using \eqref{pointwise-bootstrap},
\begin{equation*}
\|\D_\alpha T_{W_\alpha} W\|_{\dot B^\frac14_{\infty,2}} \lesssim \|W_\alpha\|_{L^\infty} \|\W\|_{\dot B^\frac14_{\infty,2}} \lesssim \eps \langle t \rangle^{-1/2}\|(\W, R)\|_{X}
\end{equation*}
so it suffices to consider the low frequency $L^\infty$ estimate with $W_{\leq 1}$. In this case we may gain derivatives from the low frequency $W_\alpha$ in the paraproduct. For instance, we may use the estimate
\begin{equation*}
\begin{aligned}
\||D|^{-1/2} \D_\alpha T_{W_\alpha} W_{\leq 1}\|_{L^\infty}  &\lesssim \||D|^{3/4} W\|_{L^\infty} \sum_{\lambda \leq 1} \| |D|^{3/4} W_\lambda\|_{L^\infty} \\
&\lesssim \||D|^{3/4} W\|_{L^\infty} \||D|^{1/2}W\|_{L^\infty} \lesssim \eps \langle t \rangle^{-1/2}\|(\W, R)\|_{X}.
\end{aligned}
\end{equation*}

Next, we estimate the balanced quadratic corrections in $\tW$. Here, we likewise consider the high and low frequencies separately. For the high frequencies,
\begin{equation*}
\|\D_\alpha \Pi(W_\alpha, \Re W)_{\geq 1}\|_{\dot B^\frac14_{\infty,2}} \lesssim \||D|^{1/8}W_\alpha\|_{L^\infty} \|\W\|_{\dot B^\frac14_{\infty,2}} \lesssim \eps \langle t \rangle^{-1/2}\|(\W, R)\|_{X}.
\end{equation*}
For the low frequency estimate, we have room as before to rebalance derivatives from the first instance of $W_\alpha$. Here, note that the summation permitted by the rebalancing likewise allows for the estimate of the implicit projection $P$ in our notation for $\Pi$. Precisely,
\begin{equation*}
\begin{aligned}
\||D|^{-1/2} \D_\alpha \Pi(W_\alpha, \Re W)_{\leq 1}\|_{L^\infty} &\lesssim \sum_{\lambda \leq 1} \||D|^{-1/2} \D_\alpha \Pi(W_\alpha, \Re W)_{\lambda}\|_{L^\infty}\\
 &\lesssim \||D|^{3/4} W\|_{L^\infty} \||D|^{1/2}W\|_{L^\infty} \lesssim \eps \langle t \rangle^{-1/2}\|(\W, R)\|_{X}.
\end{aligned}
\end{equation*}

%%%%%
%Using \eqref{pointwise-bootstrap},
%\red{
%Red formulas below have issues; I don't think paraproducts are bounded in $L^\infty$. Perhaps
%we can use the fact that $X$ contains a range of norms
%in order to unbalance the bounds and break the scaling.
%\begin{equation*}
%\||D|^{1/2} T_{W_\alpha} W\|_{L^\infty} \lesssim \|W_\alpha\|_{L^\infty} \||D|^{1/2}W\|_{L^\infty}\lesssim \eps \langle t \rangle^{-1/2}\|(\W, R)\|_{X},
%\end{equation*}}
%and similarly,
%%%%

%\red{Need to clarify and expand this paragraph.} The frequency balanced terms are similar. For the $L^\infty$ bound, we may address the boundedness of the implicit $P$ in $\Pi$ by using the extra derivatives in $X$.

For $\tQ_\alpha$, we have
\[
\tQ_\alpha - R = T_{W_\alpha} R + \Pi(W_\alpha, R) - T_{R_\alpha}W - \D_\alpha\Pi(R, 2\Re W).
\]
The estimates for the quadratic errors are similar to before: We consider the high and low frequency estimates separately, and observe that the balance of derivatives is favorable for each term so that we have room to rebalance as necessary. We present here the analysis for the unbalanced paraproduct terms. Considering first the high frequency estimate, we have 
\begin{equation*}
\begin{aligned}
\|T_{W_\alpha} R\|_{\dot B^\frac34_{\infty,2}} &\lesssim \|W_\alpha\|_{L^\infty} \|R\|_{\dot B^\frac34_{\infty,2}} \lesssim \eps \langle t \rangle^{-1/2}\|(\W, R)\|_{X}, \\
\|T_{R_\alpha} W\|_{\dot B^\frac34_{\infty,2}} &\lesssim \||D|^{1/2} R\|_{L^\infty}\|W_\alpha\|_{\dot B^\frac14_{\infty,2}}  \lesssim \eps \langle t \rangle^{-1/2}\|(\W, R)\|_{X}.
\end{aligned}
\end{equation*}
For the low frequency $L^\infty$ estimate, we have
\begin{equation*}
\begin{aligned}
\|T_{W_\alpha} R_{\leq 1}\|_{L^\infty}  &\lesssim \||D|^{3/4} W\|_{L^\infty} \sum_{\lambda \leq 1} \| |D|^{1/4}R_\lambda\|_{L^\infty} \\
&\lesssim \||D|^{3/4} W\|_{L^\infty} \|R\|_{L^\infty} \lesssim \eps \langle t \rangle^{-1/2}\|(\W, R)\|_{X}, \\
\|T_{R_\alpha} W_{\leq 1}\|_{L^\infty}  &\lesssim \||D|^{1/4} R\|_{L^\infty} \sum_{\lambda \leq 1} \| |D|^{3/4}W_\lambda\|_{L^\infty} \\
&\lesssim \||D|^{1/4} R\|_{L^\infty} \||D|^{1/2}W\|_{L^\infty} \lesssim \eps \langle t \rangle^{-1/2}\|(\W, R)\|_{X},
\end{aligned}
\end{equation*}
\end{proof}

\subsection{Bounds on the source term}
Next, we estimate the cubic and higher source terms $(\tG, \tK)$ in the equation \eqref{tWQ-system}, and thus prove Proposition \ref{p:gk-est}.

\begin{proof}
To estimate $\tG$, we observe that all the terms in $\tG_i$, $i = 1, 2, 3$ (see \eqref{tGK3}) are cubic 
and higher order expressions with variables $(W_\alpha, R)$ or their respective normal form counterparts $(\tW_\alpha, \tQ_\alpha)$, and also possibly $Y$ instead of $W_\alpha$,
within the following set of rules:
\begin{enumerate}[label=(\roman*)]
\item $R$ or its equivalent $\tQ_\alpha$ appears exactly once.
\item $W_\alpha$ or its equivalents $Y,\tW_\alpha$ appears at least twice.
\end{enumerate}
Similarly, after extracting the cubic terms $\tG^{(3)}$ as a trilinear expression
in $\tG^{(3)}(\tQ_\alpha, \tW_\alpha,\tW_\alpha)$, the remaining quartic and higher 
order terms share a similar description, but with (ii) replaced by 
\begin{enumerate}
\item[(ii)'] $W_\alpha$ or its equivalents $Y,\tW_\alpha$ appears at least three times.
\end{enumerate}
In terms of estimates, these equivalent sets of variables are interchangeable by Propositions \ref{p:energy-equiv} and \ref{p:nf-pointwise-est}. To bound the cubic terms,
we may rebalance the derivatives such that the lowest frequency variable is estimated by
\[
\|(\tW, \tQ)\|_{\dH^{\frac14}} \approx_{A_0} \| (W_\alpha, R)\|_{\dH^{-\frac34}},
\]
while the remaining two variables are controlled by $A_{\frac14}^2$. To bound the 
quartic and higher terms, we argue in the same fashion, with the adjustment that 
the two highest frequency factors are controlled by $A_{\frac14}^2$, while the intermediate frequencies are controlled by $A_0$.

\

To estimate $\tK$, the discussion is similar but slightly more complex.
As we will see below,  its cubic terms may be placed into three cases: 
\begin{enumerate}[label=(\roman*)]
\item cubic terms with $\tW_\alpha, \tW_\alpha, \tW$ where $\tW$ is in the place of the highest frequency, 
\item cubic terms with $\tQ_\alpha, \tQ_\alpha, \tW_\alpha$ where a $\tQ_\alpha$ is in the place of the highest frequency,
\item cubic terms with $\tQ_\alpha, \tQ_{\alpha\alpha}, \tW$ where $\tW$ is in the place of the highest frequency.
\end{enumerate}
We allow as above for substitutions with equivalent variables $\tW_\alpha \sim W_\alpha \sim Y$ respectively $\tQ_\alpha \sim R$. Then to classify the quartic and higher order terms, we obtain contributions similar to the three above, but with an 
additional $W_\alpha$ factor which is at or below the highest frequency as described 
in (i)-(iii). In addition to this, we obtain one more term, namely
\begin{enumerate}
\item[(iv)] a quartic term with $R, R, W_\alpha, W_\alpha$ and highest frequency $W_\alpha$.
\end{enumerate}
This categorization is immediate for the terms of $\tK_1$ and $\tK_2$. For $\tK_3$, write
\begin{equation*}
\begin{aligned}
 \tilde K_3 &=  T_{2\Re R}\tQ_\alpha - T_{\bar R} R - T_R (Q_\alpha - T_F W_\alpha) + T_{2\Re (\tQ_\alpha - R)}\tQ_\alpha \\
&= T_{2\Re R}(Q - T_RW - \Pi(R, 2\Re W))_\alpha - T_{\bar R} R - T_R (Q_\alpha - T_F W_\alpha) \\
&\quad  + T_{2\Re (Y Q_\alpha - (T_R W + \Pi(R, 2\Re W))_\alpha)}\tQ_\alpha \\
&= T_R[T_FW_\alpha - (T_RW)_\alpha] - T_{2\Re R} (\Pi(R, 2\Re W))_\alpha + T_{\bar R}[RW_\alpha -( T_RW)_\alpha] \\
&\quad  + T_{2\Re (Y Q_\alpha - (T_R W + \Pi(R, 2\Re W))_\alpha)}\tQ_\alpha \\
&= T_R[T_{F - R}W_\alpha - T_{R_\alpha}W] - T_{2\Re R} (\Pi(R, 2\Re W))_\alpha + T_{\bar R}[T_{W_\alpha} R + \Pi(W_\alpha, R) - T_{R_\alpha} W] \\
&\quad  + T_{2\Re (Y Q_\alpha - (T_R W + \Pi(R, 2\Re W))_\alpha)}\tQ_\alpha.
\end{aligned}
\end{equation*}
The first term on the right hand side is quartic, recalling that
\[
R - F = P[R\bar Y - \bar R Y].
\]

In terms of estimates, as with $\tG$, in all of the cubic cases (i)-(iii) we may rebalance derivatives such that the lowest frequency variable is estimated by 
\[
\|(\tW, \tQ)\|_{\dH^{\frac14}},
\] 
while the remaining two variables are controlled by $A_{\frac14}^2$. For the corresponding quartic terms we add $A_0$ bounds for the additional $W_\alpha$ terms. On the other hand, for the remaining quartic term in (iv) we use $W_\alpha \in \dot H^{3/4}$ for the highest frequency $W_\alpha$:
\[
\|T_RT_{P[R\bar Y - \bar R Y]}W_\alpha\|_{\dot H^{3/4}} \lesssim_{A_0} \|R\|_{L^\infty}\|R\|_{L^\infty}\|W_\alpha\|_{L^\infty} \|W_\alpha\|_{\dot H^{3/4}}.
\]

\bigskip

Next, we identify the leading order cubic source terms $(\tG^{(3)}, \tK^{(3)})$ in $(\tG, \tK)$, with respect to the normal form variables $(\tW, \tQ)$. 

\

First, we identify the cubic terms in $\tG$. Using the identities
\begin{equation}
\begin{aligned}
&F = Q_\alpha - Q_\alpha Y + P\left[ \bar RY -R\bar Y\right], \\
&W = \tW + T_{W_\alpha}W + \Pi(W_\alpha, 2\Re W), \\
&Q = \tQ + T_R W + \Pi(R, 2\Re W). 
\end{aligned}
\end{equation}
The first term of $\tG_1$ may be written
\begin{equation}
\begin{aligned}
T_{W_\alpha} (FW_\alpha) &= T_{\tW_\alpha} (\tQ_\alpha \tW_\alpha) + g,
\end{aligned}
\end{equation}
where $g$ consists of quartic terms with variables $R,W_\alpha, W_\alpha, W_\alpha$, possibly interchanged with their normal form counterparts. By a similar re-expression of the remaining terms of $\tG_1$, we may write
\begin{equation*}
\begin{aligned}
\tilde G_1 &= T_{\tW_\alpha} (\tQ_\alpha \tW_\alpha) + T_{(\tQ_\alpha \tW_\alpha)_\alpha}\tW + \Pi(\tW_\alpha, 2\Re [\tQ_\alpha \tW_\alpha]) + \Pi((\tQ_\alpha \tW_\alpha)_\alpha, 2\Re \tW) + \tG_1^{(4+)}.
\end{aligned}
\end{equation*}

From the last two lines of $\tG_2$, we observe that the quadratic components cancel,
\begin{equation*}
\begin{aligned}
&\Pi(\bar Y - \bar W_\alpha, R) + P [\Pi(W_\alpha, \bar F) - \Pi(\bar R, Y)] +( T_{\bar Y}R - T_{\bar W_\alpha} F) + (T_{\bar R} W_\alpha - T_{\bar R}Y) \\
&= - \Pi(\bar W_\alpha \bar Y, R) +  P [\Pi(W_\alpha, \bar P[R \bar Y - \bar R Y]) + \Pi(\bar R, W_\alpha Y)]\\
&\quad - T_{\bar W_\alpha \bar Y}R - T_{\bar W_\alpha} (P[\bar R Y - R \bar Y]) + T_{\bar R}W_\alpha Y
\end{aligned}
\end{equation*}

so that
\begin{equation*}
\begin{aligned}
\tilde G_2 &=- \tW_\alpha P\left[ \bar \tQ_\alpha \tW_\alpha -\tQ_\alpha \bar \tW_\alpha\right] + T_{ P\left[ \bar \tQ_\alpha \tW_\alpha - \tQ_\alpha \bar \tW_\alpha\right]_\alpha}\tW \\
&\quad + \Pi( P\left[ \bar \tQ_\alpha \tW_\alpha - \tQ_\alpha \bar \tW_\alpha\right]_\alpha, 2\Re W) + \Pi( P\left[ \bar \tQ_\alpha \tW_\alpha - \tQ_\alpha \bar \tW_\alpha\right], W_\alpha) \\ 
&\quad - \Pi(\bar \tW_\alpha^2, \tQ_\alpha) +  P [\Pi(\tW_\alpha, \bar P[\tQ_\alpha \bar \tW_\alpha - \bar \tQ_\alpha \tW_\alpha]) + \Pi(\bar \tQ_\alpha, \tW_\alpha^2)] \\
&\quad - T_{\bar \tW_\alpha^2}\tQ_\alpha - T_{\bar \tW_\alpha} (P[\bar \tQ_\alpha \tW_\alpha - \tQ_\alpha \bar \tW_\alpha]) + T_{\bar \tQ_\alpha}\tW_\alpha^2 + \tG_2^{(4+)}.
\end{aligned}
\end{equation*}

For $\tG_3$, we likewise observe that the quadratic components cancel,
\begin{equation*}
\begin{aligned}
\tilde G_3 &= T_{2 \Re (T_{W_\alpha}W + \Pi(W_\alpha, 2\Re W))_\alpha} Q_\alpha \\
&\quad + T_{2 \Re W_\alpha} ( - Q_\alpha Y + P\left[ \bar RY -R\bar Y\right]) + T_{2 \Re \tW_\alpha}(T_R W + \Pi(R, 2\Re W))_\alpha \\
&\quad + T_{2 \Re (RW_\alpha - (T_R W + \Pi(R, 2\Re W))_\alpha)}\tW_\alpha - T_{2 \Re R} ( T_{W_\alpha}W + \Pi(W_\alpha, 2\Re W))_\alpha,
\end{aligned}
\end{equation*}
so that
\begin{equation*}
\begin{aligned}
\tilde G_3 &= T_{2 \Re (T_{\tW_\alpha}\tW + \Pi(\tW_\alpha, 2\Re \tW))_\alpha} \tQ_\alpha \\
&\quad + T_{2 \Re \tW_\alpha} ( - \tQ_\alpha \tW_\alpha + P\left[ \bar \tQ_\alpha \tW_\alpha - \tQ_\alpha \bar \tW_\alpha \right]) + T_{2 \Re \tW_\alpha}(T_{\tQ_\alpha} \tW + \Pi(\tQ_\alpha, 2\Re \tW))_\alpha \\
&\quad + T_{2 \Re (\tQ_\alpha \tW_\alpha - (T_{\tQ_\alpha} \tW + \Pi(\tQ_\alpha, 2\Re \tW))_\alpha)}\tW_\alpha - T_{2 \Re \tQ_\alpha} ( T_{\tW_\alpha}\tW + \Pi(\tW_\alpha, 2\Re \tW))_\alpha + \tG_3^{(4+)}.
\end{aligned}
\end{equation*}

\

Next, we identify the cubic terms in $\tK$. From the last line of $\tK_1$, we first observe the quadratic cancellations,
\begin{equation*}
\begin{aligned}
\ [\Pi(R, 2\Re F) &- \Pi(F, Q_\alpha) - \Pi(\bar R, R)] + T_{R - Q_\alpha} F \\
&= - \Pi(YQ_\alpha, F) + \Pi(R, \bar P\left[\bar R Y - R \bar Y\right]) - T_{Q_\alpha Y} F
\end{aligned}
\end{equation*}
so that
\begin{equation*}
\begin{aligned}
\tilde K_1 &= T_{2(\Re \tQ_\alpha ) \tQ_{\alpha\alpha} + 2i\Im P[\tQ_\alpha \bar \tQ_{\alpha\alpha}]} \tW + T_{\tQ_\alpha} (T_{\tW_\alpha} \tQ_\alpha + \Pi(\tW_\alpha, \tQ_\alpha)) \\
&\quad + \Pi(2(\Re \tQ_\alpha ) \tQ_{\alpha\alpha} + 2i\Im P[\tQ_\alpha \bar \tQ_{\alpha\alpha}], 2\Re \tW) +  \Pi(\tQ_\alpha, 2\Re [\tQ_\alpha \tW_\alpha]) \\
&\quad - \Pi(\tW_\alpha \tQ_\alpha, \tQ_\alpha) + \Pi(\tQ_\alpha, \bar P\left[\bar \tQ_\alpha \tW_\alpha - \tQ_\alpha \bar \tW_\alpha \right]) - T_{\tQ_\alpha \tW_\alpha} \tQ_\alpha + \tK_1^{(4+)}.
\end{aligned}
\end{equation*}

For $\tK_2$, the quadratic cancellations are straightforward from the definition of $Y$ and we obtain
\[
\tilde K_2 =  i T_{\tW_\alpha^2} \tW  + i\Pi( \tW_\alpha^2, 2\Re \tW) - T_{P\left[ \bar \tQ_\alpha \tW_\alpha - \tQ_\alpha \bar \tW_\alpha \right]} \tQ_\alpha + \tK_2^{(4+)}.
\]

For $\tK_3$, we have the quadratic cancellations
\begin{equation*}
\begin{aligned}
\tilde K_3 &=  T_{2\Re \tQ_\alpha}\tQ_\alpha - T_{\bar R} R - T_R (Q_\alpha - T_F W_\alpha) \\
&= - T_{2\Re (T_R W + \Pi(R, 2\Re W))_\alpha}\tQ_\alpha - T_{2\Re Q_\alpha} (T_R W + \Pi(R, 2\Re W))_\alpha \\
&\quad + T_{2\Re (Q_\alpha Y)} Q_\alpha + T_{\bar R} (Q_\alpha Y) + T_R T_F W_\alpha ,
\end{aligned}
\end{equation*}
and hence
\begin{equation*}
\begin{aligned}
\tilde K_3 &= - T_{2\Re (T_{\tQ_\alpha} \tW + \Pi(\tQ_\alpha, 2\Re\tW))_\alpha}\tQ_\alpha - T_{2\Re \tQ_\alpha} (T_{\tQ_\alpha} \tW + \Pi(\tQ_\alpha, 2\Re \tW))_\alpha \\
&\quad + T_{2\Re (\tQ_\alpha \tW_\alpha)} \tQ_\alpha + T_{\bar \tQ_\alpha} (\tQ_\alpha \tW_\alpha) + T_{\tQ_\alpha} T_{\tQ_\alpha}\tW_\alpha + \tK_3^{(4+)}.
\end{aligned}
\end{equation*}

\end{proof}

\section{The pointwise estimates}
\label{s:wp}
So far, we have used the pointwise bootstrap assumption \eqref{pointwise-bootstrap}
in order to derive the energy estimates \eqref{ee-global} with $t^{C\epsilon^2}$ loss for $(W,Q)$, which we then transferred to  normal form variables $(\tW,\tQ)$, see \eqref{ee-global-nf},
with a similar $t^{C\epsilon^2}$ loss. The remaining objective is to obtain an improvement of the bootstrap assumption \eqref{pointwise-bootstrap}, which has in turn been reduced to proving its counterpart for the normal form variables, namely the bound \eqref{pointwise-nf}. All
the work in the last two sections of the paper happens at the level of the normal form variables $(\tW,\tQ)$.

Our primary objective in this section is to consider $(\tW,\tQ)$ at fixed time, 
and to convert the energy estimates \eqref{ee-global-nf} into pointwise bounds, via 
vector field Sobolev type inequalities.

This will in particular 
yield the pointwise bound
\begin{equation}\label{get-nf-X}
\| (\tW_\alpha,\tQ_\alpha)\|_{X} \lesssim \epsilon t^{-\frac12+C\epsilon^2}.   
\end{equation}  
However, this does not suffice in order to prove \eqref{pointwise-nf} because 
of the $t^{C\epsilon^2}$ loss. For this reason, we will instead obtain a sharper 
version of \eqref{get-nf-X}, where the loss is replaced with a gain for 
most components of $(\tW,\tQ)$.  To describe this gain, we will produce 
an elliptic/hyperbolic decomposition 
\begin{equation}\label{tWQ=ell+hyp}
(\tW,\tQ) = (\tW,\tQ)_{ell} + (\tW,\tQ)_{hyp}.
\end{equation}
Here the elliptic component contains a nearly full range of frequencies, but
satisfies stronger, elliptic bounds, and in particular has better decay,
\begin{equation}\label{get-nf-X-ell}
\| (\tW_\alpha,\tQ_\alpha)_{ell}\|_{X} \lesssim \epsilon t^{-\frac12-\delta}, \qquad \delta > 0,
\end{equation} 
which suffices for \eqref{pointwise-nf}.

The hyperbolic component, on the other hand, is frequency localized on a scale which depends on the velocity $v = \alpha/t$. While retaining the $t^{C\epsilon^2}$ loss, it has another redeeming feature, namely 
a gain of $\min\{|v|^{-b},|v|^b\}$ away from velocity $|v| \approx 1$, with a universal small $b$. This will defeat the $t^{C\epsilon^2}$ loss outside a small region of the form
\begin{equation}\label{Omega-delta}
\Omega^\delta = \{ t^{-\delta} \lesssim |v| \lesssim t^{\delta} \}, \qquad \delta \ll 1.
\end{equation}

That will leave us, at the conclusion of this section, with the remaining task of improving 
the pointwise bounds for $(\tW,\tQ)_{hyp}$ within the above region $\Omega^\delta$. This can no longer 
be done via a fixed time analysis, and instead has to be accomplished dynamically. 
That will be the objective of the last section of the paper, where we use our wave packet testing method to capture a good asymptotic parameter $\gamma(t,v)$ and its associated asymptotic equation.   

\subsection{ A fixed time system for \texorpdfstring{$(\tW,\tQ)$}{}}
We recall from \eqref{ts-eqn} that we have
\[
\tS(\tW, \tQ) = S(\tW, \tQ) - t(\tG, \tK),
\]
where the nonlinear paradifferential vector field $\tS$ is given by
\[
\tS(\tW,\tQ) = (2 \alpha \D_\alpha \tW - t\D_\alpha \tQ + t(T_{2 \Re \tW_\alpha} \tQ_\alpha - T_{2 \Re \tQ_\alpha} \tW_\alpha), 2 \alpha \D_\alpha \tQ + it\tW - tT_{2\Re  \tQ_\alpha} \tQ_\alpha).
\]

Thus we can think of $(\tW,\tQ)$ at fixed time as the solutions to a system governed by the 
operator $\tS$,
\begin{equation}\label{KS-system}
\left\{
\begin{aligned}
& 2 \alpha \D_\alpha \tW - t\D_\alpha \tQ + t(T_{2 \Re \tW_\alpha} \tQ_\alpha - T_{2 \Re \tQ_\alpha} \tW_\alpha) = G,
\\
& 2 \alpha \D_\alpha \tQ + it\tW - tT_{2\Re  \tQ_\alpha} \tQ_\alpha = K,
\end{aligned}  
\right.
\end{equation}
where, by Proposition~\ref{p:energy-equiv} and Proposition~\ref{p:gk-est}, we control
$L^2$ type norms as follows:
\begin{equation}\label{KS-control}
\| (\tW,\tQ)\|_{\dH^\frac14 \cap \dH^\sigma} 
+ \|(G,K)\|_{\dH^\frac14} 
\lesssim \epsilon t^{C \epsilon^2}.
\end{equation}

At first we will regard this as a linear system for $(\tW,\tQ)$, where the 
paradifferential coefficients are decoupled from the main variables, and are instead assumed to have $t^{-\frac12}$ decay in the uniform $X$ norm,
\begin{equation}\label{boot-nf}
\| (\tW_\alpha,\tQ_\alpha)\|_{X} \lesssim C \epsilon t^{-\frac12}.   
\end{equation}
The pointwise bounds we will prove
for solutions to this system hold irrespective of the origin of $\tW$ and $\tQ$.
To emphasize this, we will more generally consider any solution $(w,q)$
to the system $\tS(w,q)=(g,k)$ or in expanded form
\begin{equation}\label{S-system}
\left\{
\begin{aligned}
   & 2 \alpha  w_\alpha - t q_\alpha + t(T_{2 \Re \tW_\alpha} q_\alpha - T_{2 \Re \tQ_\alpha} w_\alpha) =  g,
    \\
&    2 \alpha q_\alpha + it w - tT_{2\Re  \tQ_\alpha} q_\alpha =  k.
\end{aligned}
\right.
\end{equation}
In a nutshell, our  goal will be to 
obtain pointwise bounds for $(w,q)$ in terms of Sobolev bounds
for $(w,q)$, respectively $(g,k)$. A simplified version of our main estimate
is as follows:

\begin{proposition}\label{p:KS-para}
Assume that $(\tW,\tQ)$ satisfy the bootstrap bound \eqref{boot-nf}.
Then the following pointwise bound holds for solutions $(w,q)$ to \eqref{S-system}:
\begin{equation}\label{KS-para}
\| (w_\alpha,q_\alpha)\|_X
\lesssim t^{-\frac12} \left(\| (w,q)\|_{\dot \H^\frac14 \cap \dot H^\sigma} 
+ \| (g,k)\|_{\dot \H^\frac14} \right).
\end{equation}
\end{proposition}
However, such a bound does not suffice for our purposes due to the $t^{C\epsilon^2}$
loss in \eqref{KS-control}, so in the next subsection we perform a finer analysis,
where we replace the $X$ norm above with a stronger norm, which we call $X^\sharp$.

\subsection{The elliptic-hyperbolic decomposition and the \texorpdfstring{$X^\sharp$}{} norm.}
To better understand the system \eqref{S-system} we begin with a heuristic discussion. The starting point is to consider a dyadic decomposition for $\alpha$. In a fixed dyadic region 
$\alpha \approx \alpha_0$, the system \eqref{S-system} is microlocally elliptic 
unless the frequency $\xi$ is comparable to
\[
\xi_0 = \frac{t^2}{\alpha_0^2},
\]
in which case the system is microlocally hyperbolic. Thus in this region we distinguish 
between
\begin{itemize}
    \item Elliptic low frequencies, $|\xi| \ll \xi_0$,
    \item Hyperbolic intermediate frequencies, $|\xi| \approx \xi_0$,
    \item Elliptic high frequencies, $|\xi| \gg \xi_0$.
\end{itemize}

The above phase space decomposition applies for a range of $\alpha_0$, but not for all.
We separate two extreme cases:

\medskip

\emph{a) Very low $\alpha$}, namely 
\[
\alpha \ll \alpha_{lo} := t^{\frac34}.
\]
This corresponds to $\xi_0 \gg t^\frac12$ and to $|v| \ll v_{lo}:=t^{-\frac14}$. In this range 
we simply separate frequencies into low and high
relative to the frequency $\xi_{lo} = t^\frac12$:
\begin{itemize}
    \item Elliptic low frequencies, $|\xi| \ll \xi_{lo}$,
    \item High frequencies, $|\xi| \gtrsim \xi_0$
\end{itemize}
where, in the high frequency region, the $\dH^\frac14$ regularity 
of the source terms in \eqref{S-system} is superseded by the $\dH^\sigma$
bound for $(w,q)$.

\medskip

\emph{b) Very high $\alpha$}, namely 
\[
\alpha \gg \alpha_{hi} := t^{2}.
\]
This corresponds to $\xi_0 \ll t^{-2}$ and to $v \gg v_{hi}:= t$. In this range 
we again separate frequencies into low and high:
\begin{itemize}
    \item Low frequencies, $|\xi| \lesssim \xi_{hi}$,
    \item Elliptic high frequencies, $|\xi| \gg \xi_{hi}$
\end{itemize}
where, in the low frequency region, the $\dH^\frac14$ regularity 
of the source terms in \eqref{S-system} is superseded by the $\dH^\frac14$
bound for $(w,q)$.

Corresponding to the above decomposition of the phase space, we consider 
an associated decomposition of $(w,q)$. Our strategy will be to localize 
spatially first, and then in frequency. Some care is required at the level of the spatial localization. At low frequency we have $w \in \dot H^\frac14$, which is a localizable norm.
On the other hand, at low frequency $q \in \dot H^\frac34$, which is not a localizable norm;
in particular $q$ is only defined modulo constants. Hence, rather than localizing $q$
it is better to localize $q_\alpha$.
Thus, given a bump function $\chi$, we define the associated localization operator,
which we denote by $\bchi$, as follows:
\[
\bchi(w,q) = (w_1,q_1) \quad \text{iff} \quad w_1 = \chi w, \ \ q_{1,\alpha} = \chi q_\alpha.
\]

We now use these localization operators to define the $X^\sharp$ norm via a full decomposition of $(w,q)$. We begin with the spatial decomposition, 
\begin{equation}\label{(wq)-dec}
(w,q) = \bchi_{\ll \alpha_{lo}} (w,q) + \bchi_{\gg \alpha_{hi}}(w,q)     
+ \sum_{ \alpha_{lo} \lesssim \alpha_0 \lesssim \alpha_{hi}} \bchi_{\alpha_0}
(w,q).
\end{equation}
Then the summand in the last term is further decomposed in frequency,
\begin{equation}\label{(wq)-dec+}
\bchi_{\alpha_0} (w,q) = P_{\ll\xi_0} \bchi_{\alpha_0} (w,q)+ P_{\xi_0} \bchi_{\alpha_0} (w,q)  + P_{\gg \xi_0}\bchi_{\alpha_0} (w,q),
\end{equation}
into an elliptic low frequency component, a hyperbolic component and an elliptic high frequency component. We note that these truncations do not preserve the spatial localizations; however 
the ensuing tails are smooth and rapidly decreasing away from the original support, and do not have any effect on the arguments that follow.

For later use, we employ these truncations in order to define 
a decomposition of $(w,q)$ into an elliptic and a hyperbolic part,
\begin{equation}\label{wq=ell+hyp}
  (w,q) = (w_{ell},q_{ell})  + (w_{hyp},q_{hyp}),
\end{equation}
where the hyperbolic part is defined as
\begin{equation}\label{hyp-decomp}
   (w_{hyp},q_{hyp}) =  \sum_{ \alpha_{lo} \lesssim \alpha_0 \lesssim \alpha_{hi}} P_{\xi_0} \bchi_{\alpha_0} (w,q).
\end{equation}

As hinted earlier, to measure the size of $(w,q)$ we will not simply use the $X$ norm; 
instead we introduce a stronger norm $X^\sharp$ which we now define. Based on the decomposition above, we set:
\begin{equation}\label{def:Xsharp}
\| (w,q) \|_{X^\sharp} 
= \| \bchi_{\ll \alpha_{lo}} (w,q)\|_{X^\sharp_{lo}} + \|\bchi_{\gg \alpha_{hi}}(w,q)\|_{X^\sharp_{hi}}     
+ \sup_{ \alpha_{lo} \lesssim \alpha_0 \lesssim \alpha_{hi}} \| \bchi_{\alpha_0}
(w,q)\|_{X^\sharp_{\alpha_0}},
\end{equation}
where the first two of the component norms are as follows:
\begin{equation}\label{X-lo}
\|(w,q)\|_{X^\sharp_{lo}} := t^\frac12 \| (w,q)\|_{\dH^\frac34}, 
\end{equation}
\begin{equation}\label{X-hi}
\|(w,q)\|_{X^\sharp_{hi}} :=  t^\frac32  \| (w_\alpha,q_\alpha)\|_{\dH^{\frac14}}.
\end{equation}

For the last component we distinguish between the case
$\xi_0 < 1$ (which corresponds to $\alpha_0 > t$)  and $\xi_0 > 1$ (which corresponds to $\alpha_0 < t$). 
In the first case, we set
\begin{equation}\label{Xsharp-large}
\|(w,q)\|_{X^\sharp_{\alpha_0}} :=t^\frac12 \xi_0^{-\frac12} \| P_{> \xi_{0}}(w,q)_\alpha\|_{\dot \H^{\frac14}} + t^\frac12 \| P_{< \xi_{0}}(w,q)_\alpha\|_{\dot \H^{-\frac14}}
+ \xi_0^{-a} \| P_{\xi_{0}}(w,q)\|_{X^0}, 
\end{equation}
while in the second case we define
\begin{equation}\label{Xsharp-small}
\|(w,q)\|_{X^\sharp_{\alpha_0}} :=
t^\frac12 \xi_0^{-\frac12} \| P_{> \xi_{0}}(w,q)_\alpha\|_{\dot \H^{\frac14}} + t^\frac12 \| P_{< \xi_{0}}(w,q)_\alpha\|_{\dot \H^{-\frac14}}
+ \xi_0^{b} \| P_{\xi_{0}}(w,q)\|_{X^0},
\end{equation}
where
\[
a = \frac54, \qquad b = \frac14(\sigma - \frac{11}4 ).
\]

Here the $X^0$ norm, used above as a reference norm, corresponds exactly to $A_{1/4}$,
\[
\| (w,q)\|_{X^0} = \| w_\alpha\|_{\dot B^{\frac14}_{\infty,2}} + \| q_\alpha \|_{\dot B^{\frac34}_{\infty,2}}.
\]

If $\alpha_0 < t$, which corresponds 
to $\xi_0 > 1$, then in the hyperbolic region the $X^0$ norm controls the full $X$ norm, and so 
it suffices to have a small  gain $b > 0$.  On the other hand if $\alpha_0 > t$, which corresponds to $\xi_0 < 1$, then the $X^0$ norm no longer controls the full $X$ norm at frequency $\xi_0$, so we need a gain $a > \frac34$ for the $X$ norm; we actually get $\frac54$.

The elliptic portion in our decomposition will play a perturbative role in our analysis.
For this reason, it is convenient to separately define a norm $X_{ell}^\sharp$ in order to measure it. Precisely, we set
\begin{equation}\label{def:Xsharp-ell}
  \| (w,q) \|_{X^\sharp_{ell}} 
= \| \bchi_{\ll \alpha_{lo}} (w,q)\|_{X^\sharp_{lo}} + \|\bchi_{\gg \alpha_{hi}}(w,q)\|_{X^\sharp_{hi}}     
+ \sup_{ \alpha_{lo} \lesssim \alpha_0 \lesssim \alpha_{hi}} \| \bchi_{\alpha_0}
(w,q)\|_{X^\sharp_{\alpha_0,ell}}, 
\end{equation}
where
\begin{equation}\label{wq-ell}
  \| (w,q)\|_{X^\sharp_{\alpha_0,ell}}    = t^\frac12 \xi_0^{-\frac12} \| (w,q)_\alpha\|_{\dot \H^{\frac14}} + t^\frac12 \|(w,q)_\alpha\|_{\dot \H^{-\frac14}}.
\end{equation}

\subsection{ Bounds for the linear system \texorpdfstring{\protect\eqref{S-system}}{}}\label{ss:wq}

The main objective here is to use the $X^\sharp$ norm in order to state and prove an enhanced form of Proposition~\ref{p:KS-para}:

\begin{proposition}\label{p:KS-para+}
Assume that $(\tW,\tQ)$ satisfy the bootstrap bound \eqref{boot-nf}.
Then the following bound holds for solutions $(w,q)$ to \eqref{S-system}:
\begin{equation}\label{KS-para+}
\| (w,q)\|_{X^\sharp}
\lesssim t^{-\frac12} \left(\| (w,q)\|_{\dot \H^\frac14 \cap \dot H^\sigma} 
+ \| (g,k)\|_{\dot \H^\frac14} \right).
\end{equation}
\end{proposition}

\begin{proof}[Proof of Proposition~\ref{p:KS-para+}]

In the  constant coefficient case, in the absence of the paradifferential quadratic terms, 
the bound \eqref{KS-para+} essentially follows from \cite{HIT}.
 The argument in \cite{HIT} begins with a frequency localization, and then identifies spatially 
the elliptic and hyperbolic regions. In our case such an argument 
is no longer possible because the frequency localization no longer 
commutes with $\tS$, i.e. the commutator of Littlewood-Paley projectors with the paradifferential terms is not perturbative (does not have enough time decay). But what we can do instead is change the order of the two steps, i.e. 
first localize spatially in dyadic regions and then identify the elliptic and hyperbolic frequency ranges.

\textbf{Step 1: Localization.}
To localize spatially we consider a unit bump function $\chi$ which selects 
a dyadic spatial range $\alpha \approx \alpha_0$  with $\alpha_{lo} \lesssim \alpha_0 \lesssim 
\alpha_{hi}$. Then $|\chi'| \lesssim \alpha_0^{-1}$.
We replace $(w,q)$ by $(w_1,q_{1}) = \bchi( w, q)$, and seek to get a good equation for $(w_1,q_1)$. We begin with
\[
\alpha \partial_\alpha w_1 = \chi \alpha \partial_\alpha w + \alpha \chi' w ,
\]
which is acceptable since multiplication by $\chi$ or $\alpha \chi' $ preserves $\dot H^\frac14$. Similarly, by duality, it also preserves $\dot H^{-\frac14}$ so the $q_\alpha$ truncation is also acceptable.

Next we consider the commutator of $\bchi$ with the paraproducts. For these we need the following

\begin{lemma}\label{l:com}
Assume that the pointwise bound \eqref{boot-nf} holds. Then 
we have the commutator bounds
\begin{equation} \label{com-tQ}
\| [\chi,T_{\tQ_\alpha}] w_\alpha\|_{\dot H^{\frac14}} \lesssim t^{-1} \| w\|_{\dot H^\frac14},   
\end{equation}
\begin{equation} \label{com-tQ1}
\| [\chi,T_{\tQ_\alpha}] q_\alpha\|_{\dot H^{\frac34}} \lesssim t^{-1} \| q\|_{\dot H^{\frac34}},    
\end{equation}
\begin{equation}\label{com-tW}
\| [\chi,T_{\tW_\alpha}] q_\alpha \|_{\dot H^{\frac14}} \lesssim t^{-1} \| q\|_{\dot H^{\frac34}}.  
\end{equation}
\end{lemma}

\begin{proof}
We start with \eqref{com-tQ}. Since $\chi$ is essentially localized at frequencies
$\lesssim \alpha_0^{-1}$, we first eliminate the low frequencies in $w$, namely those below $\alpha_0^{-1}$. Their contribution is estimated directly, without using the commutator structure:
\[
\| [\chi,T_{\tQ_\alpha}] P_{< \alpha_0^{-1}} w_\alpha\|_{L^2}
\lesssim \|\tQ_\alpha\|_{L^\infty} \|  P_{< \alpha_0^{-1}} w_\alpha\|_{L^2}
\lesssim \alpha_0^{-\frac34} \|\tQ_\alpha\|_{L^\infty} \|w\|_{\dot H^\frac14} 
\lesssim t^{-\frac12} \alpha_0^{-\frac34}  \|w\|_{\dot H^\frac14}.
\]
Each additional derivative contributes an $\alpha_0^{-1}$ factor, so we get
\[
\| [\chi,T_{\tQ_\alpha}] P_{< \alpha_0^{-1}} w_\alpha\|_{ \dot H^\frac14}
\lesssim t^{-\frac12} \alpha_0^{-1}  \|w\|_{\dot H^\frac14},
\]
which suffices since $\alpha_0 \gtrsim \alpha_{lo} = t^{\frac34}$. 

Next we replace $\chi$ by $T_\chi$. Suppose $w$ is localized at frequency $\mu > \alpha_0^{-1}$. Then we estimate 
\[
\| P_{\gtrsim \mu} \chi \|_{L^\infty} \lesssim (\mu \alpha_0)^{-N},
\]
and then repeat the computation above to obtain
\[
\|[\chi - T_\chi,T_{\tQ_\alpha}] P_{\mu} w_\alpha\|_{ \dot H^\frac14}
\lesssim t^{-\frac12} \alpha_0^{-1} (\mu \alpha_0)^{-N} \|w\|_{\dot H^\frac14}.
\]
Finally we consider the paraproduct commutator $[T_\chi,T_{\tQ_\alpha}]$, where the input and output frequencies are equal and equal to $\mu > \alpha_0^{-1}$. 
 This vanishes if either $\chi$ or $\tQ_\alpha$ are constant,
so it is natural to think of it as a bilinear form
in their derivatives $\chi_\alpha$ respectively
$\tQ_{\alpha\alpha}$. Indeed,
 we can write the commutator in the form
\[
[T_\chi,T_{\tQ_\alpha}] P_{\mu} w_\alpha = \mu^{-1} L_{llh}(\chi_\alpha,\tQ_{\alpha \alpha}, 
P_\mu w)
\]
where $L_{llh}$ stands for a translation invariant trilinear form with uniformly integrable kernel and lower frequencies in the first two entries. Hence we have the $L^2$ bound
\[
\| [T_\chi,T_{\tQ_\alpha}] P_{\mu} w_\alpha \|_{\dot H^{\frac14}}
\lesssim \|\chi_\alpha\|_{L^\infty} \|\tQ_{ \alpha}\|_{L^\infty} 
\|P_\mu w\|_{\dot H^\frac14} \lesssim t^{-\frac12} \alpha_0^{-1}  \|P_\mu w\|_{\dot H^\frac14},
\]
which suffices, exactly as above.

This concludes the proof of \eqref{com-tQ}. The 
proof of \eqref{com-tQ1} is identical. 

Finally we consider \eqref{com-tW}, where we carry out the same steps.
For the very low frequencies we have
\[
\| [\chi,T_{\tW_\alpha}] P_{< \alpha_0^{-1}} q_\alpha\|_{ \dot H^\frac14}
\lesssim  \alpha_0^{-1} \| D^\frac12 W\|_{L^\infty} \|q\|_{\dot H^\frac34}
\lesssim  \alpha_0^{-1} t^{-\frac12} \|q\|_{\dot H^\frac34},
\]
for the high $\chi$ frequencies we gain extra $(\mu \alpha_0)^{-N}$ factors,
and the paraproduct commutator is the same as above.
\end{proof}

To summarize, we have reduced the problem to three localized settings, i.e. where
$(w,q_\alpha)$ are localized in one of the following three regions:

\begin{enumerate}
    \item Low $\alpha$,  $|\alpha| \ll \alpha_{lo} = t^\frac34$, where it suffices to prove a low frequency
    elliptic bound.
 \item    High $\alpha$, $ |\alpha| \gg \alpha_{hi} = t^2$ where it suffices to prove a high frequency
    elliptic bound.
\item Intermediate dyadic $\alpha$, $\alpha \approx \alpha_0$, with $t^\frac34 \lesssim \alpha_0 \lesssim t^2$. Here we will do a full elliptic-hyperbolic decomposition.
\end{enumerate}
We consider each of these three cases in turn.

\bigskip

\textbf{Step 2: The low $\alpha$ region, $|\alpha| \ll \alpha_{lo}$.}
Here for the high frequencies $|\xi| \gtrsim \xi_{lo}= t^\frac12$, we simply use the $\dH^\sigma$ bound.

The same $\dH^{\sigma}$ bound allows us to treat perturbatively the input of the high frequencies to the equation \eqref{S-system}, 
\[
\| \tS P_{\geq \xi_{lo}} (w,q) \|_{\dH^\frac14} \lesssim \| (w,q)\|_{\dH^{\sigma}}.
\]
Here we use the bound on the size of $\alpha$ within the support of $(w,q)$. 
The frequency projector $P_{\geq \xi_{lo}}$ does not have a localized kernel, but 
it decays rapidly on the $t^{\frac12}$  scale, so it only generates $O(t^{-N})$ errors.

Thus we are left with an equation of the form \eqref{S-system} for the low frequency
component 
$(w_{lo}, q_{lo}) =P_{\ll \xi_{lo}} (w,q)$
\[
L[(w_{lo},q_{lo})] = (g_{lo},k_{lo}),
\]
where the localization is again retained up to negligible tails. 

We rewrite this system as 
\[
\left\{
\begin{aligned}
   &  t q_{lo,\alpha} = 2 \alpha  w_{lo,\alpha} + t(T_{2 \Re \tW_\alpha} q_{lo,\alpha} - T_{2 \Re \tQ_\alpha} w_{lo,\alpha}) -  g_{lo}
    \\
&     it w_{lo}  = -   2 \alpha q_{lo,\alpha} + tT_{2\Re  \tQ_\alpha} q_{lo,\alpha} +  k_{lo},
\end{aligned}
\right.
\]
with source terms $(g_{lo},k_{lo})$ satisfying the same bounds as $(g,k)$.
Here we directly obtain the elliptic bound
\[
\| (q_{lo,\alpha},w_{lo}) \|_{\dH^\frac14} \lesssim t^{-1} \|(g_{lo},k_{lo})\|_{\dH^\frac14},
\]
simply by treating all the terms we have moved to the right in a perturbative manner,
using both the frequency and the spatial\footnote{up to negligible tails} localization.
This in turn can be rewritten as 
\[
\| (w_{lo},q_{lo}) \|_{\dH^\frac34} 
 \lesssim  t^{-1} \|(g_{lo},k_{lo})\|_{\dH^\frac14}  
\lesssim  t^{-1}\left(\| (w,q)\|_{\dot \H^\frac14 \cap \dot H^\sigma} 
+ \| (g,k)\|_{\dot \H^\frac14} \right),
\]
as needed for the $X^{\sharp}_{lo}$ norm.

\bigskip

\textbf{Step 3: The high $\alpha$ region, $|\alpha| \gg \alpha_{hi}$.}
Here we apply the same strategy as in the previous case, but reversing the role 
of high and low frequencies. Precisely, for frequencies below $\xi_{hi} = t^{-2}$
we only retain the starting $\dH^{\frac14}$ bound for $(w,q)$. This suffices 
in order to place the contribution of the low frequencies into the source term in 
\eqref{S-system},
\[
\| \tS P_{\lesssim \xi_{hi}} (w,q)\|_{\dH^\frac14} \lesssim \|(w,q)\|_{\dH^\frac14}.
\]
We note that here the left hand side cannot be estimated directly due to the large 
$\alpha$ factors. Instead, we need to commute $\tS$ and  $P_{\lesssim \xi_hi}$. 
Thus, we obtain a system of the same form \eqref{S-system}
for the high frequencies $(w_{hi},q_{hi}) = P_{\gg \xi_{hi}}(w,q)$. We rewrite
this system in the form
\[
\left\{
\begin{aligned}
   & 2 \alpha  w_{hi,\alpha} = t q_{hi,\alpha} -t(T_{2 \Re \tW_\alpha} q_{hi,\alpha} - T_{2 \Re \tQ_\alpha} w_{hi,\alpha}) +  g_{hi}
    \\
&    2 \alpha q_{hi,\alpha} =- it w_{hi} + tT_{2\Re  \tQ_\alpha} q_{hi,\alpha} +  k_{hi}.
\end{aligned}
\right.
\]
On the left we use the localization to $|\alpha| > \alpha_{hi} = t^2$ to estimate from below
\[
t^{2} \|(w_{hi,\alpha},q_{hi,\alpha})\|_{\dH^\frac14} \lesssim 
\|(\alpha w_{hi,\alpha},\alpha q_{hi,\alpha})\|_{\dH^\frac14} + \|(w,q)\|_{\dH^\frac14}.
\]
Using this bound allows us to estimate perturbatively all the terms we have moved to 
the right, and thus obtain the elliptic bound
\[
 \|(w_{hi,\alpha},q_{hi,\alpha})\|_{\dH^\frac14} \lesssim t^{-2}\left(\| (w,q)\|_{\dot \H^\frac14 \cap \dot H^\sigma} 
+ \| (g,k)\|_{\dot \H^\frac14} \right),
\]
as needed for the $X^{\sharp}_{hi}$ norm.

\bigskip

\textbf{Step 4: The intermediate $\alpha$ region, $\alpha_{lo} \lesssim \alpha_0 \lesssim \alpha_{hi}$.} 
Here we assume that $(w,q_\alpha)$ are localized in the dyadic region $|\alpha| \approx \alpha_0$. The difficulty we have in this region is that the potentials are nonperturbative, at least in the hyperbolic region. To address this difficulty, we first use perturbative analysis to estimate $(w,q)$ in the elliptic region. 

\bigskip
\emph{Step 4(a): The elliptic analysis.}
This has two components:

\medskip

i) High frequency, $ \xi \gg \xi_0$. Here the leading component 
is $\alpha \partial_\alpha$, therefore we would like to prove the bound
\begin{equation}\label{ell-hi}
\| P_{\gg \xi_0} (w,q) \|_{\dot \H^\frac54} \lesssim \frac{1}{\alpha_0} = \frac{1}{t} |\xi_0|^\frac12.
\end{equation}
 With $j > 0$ we  apply the projector $P_{\geq 2^j \xi_0}$ respectively in the equations \eqref{S-system}, and commute to obtain an equation for  $(w_{j}, q_{j}) = P_{\geq 2^j \xi_0} (w,q)$. We obtain
\begin{equation}\label{S-system=high}
\left\{
\begin{aligned}
   & 2 \alpha  w_{j,\alpha} - t q_{j,\alpha} + t(T_{2 \Re \tW_\alpha} q_{j,\alpha} - T_{2 \Re \tQ_\alpha} w_{j,\alpha}) =  g_j
    \\
&    2 \alpha q_{j,\alpha} + it w_{j} - tT_{2\Re  \tQ_\alpha} q_{j,\alpha} =  k_j,
\end{aligned}
\right.
\end{equation}
with source terms 
\[
\left\{
\begin{aligned}
g_j =  & \ P_{\geq 2^j \xi_0} g + t ( [ P_{\geq 2^j \xi_0},T_{2 \Re \tW_\alpha}] q_{j,\alpha} - [P_{\geq 2^j \xi_0},  T_{2 \Re \tQ_\alpha}] w_{j,\alpha}) 
\\
k_j = & \ P_{\geq 2^j \xi_0} k + t [ P_{\geq 2^j \xi_0} , T_{2 \Re \tQ_\alpha}] w_{j,\alpha}.
\end{aligned}
\right.
\]
We estimate the source terms in $\dot H^\frac14$, using \eqref{boot-nf}, as follows:
\[
\begin{aligned}
\| (g_j,k_j) \|_{\dH^\frac14} \lesssim & \  \| (g,k) \|_{\dH^\frac14} 
+ t \xi_0^{-\frac12} ( \|D^\frac12 \tQ_\alpha\|_{L^\infty} + \|\tW\|_{L^\infty}) 
\|(w_{j-1,\alpha},q_{j-1,\alpha}) \|_{\dH^\frac14}
\\
\lesssim & \  \| (g,k) \|_{\dH^\frac14} + t^{\frac12} \xi_0^\frac12 \|(w_{j-1,\alpha},q_{j-1,\alpha}) \|_{\dH^\frac14},
\end{aligned}
\]
where $(w_{j-1},q_{j-1})$ arise due to the fact that the commutators have a slightly larger 
frequency support.

Then we consider the system \eqref{S-system=high},  where we observe that all but the 
first terms in each of the equation can be treated perturbatively at frequencies $\gg \xi_0$. Hence we obtain the bound
\[
\| (w_{j,\alpha},q_{j,\alpha}) \|_{\dH^\frac14} \lesssim 
\alpha_0^{-1}  \| (g_j,k_j) \|_{\dH^\frac14}
\lesssim t^{-1} \xi_0^\frac12  \| (g,k) \|_{\dH^\frac14} + t^{-\frac12} \|(w_{j-1,\alpha},q_{j-1,\alpha}) \|_{\dH^\frac14}).
\]
Reiterating this bound several times, we are eventually able to use 
our a-priori bound on $(w,q)$ to conclude that for some large fixed $j$ (e.g. $j = 5$)
we obtain 
\[
\| (w_{j,\alpha},q_{j,\alpha}) \|_{\dH^\frac14} \lesssim t^{-1} \xi_0^\frac12
\left(\| (w,q)\|_{\dot \H^\frac14 \cap \dot H^\sigma} 
+ \| (g,k)\|_{\dot \H^\frac14} \right),
\]
thus proving \eqref{ell-hi}.

\medskip

ii) Low frequency, $ \xi \ll \xi_0$. Here the leading component 
is the linear $t$ component, therefore we would like to show that
\begin{equation}\label{ell-lo}
\| P_{\ll \xi_0} (w,q) \|_{\dot \H^\frac34} \lesssim \frac{1}{t}\left(\| (w,q)\|_{\dot \H^\frac14 \cap \dot H^\sigma} 
+ \| (g,k)\|_{\dot \H^\frac14} \right).
\end{equation}
The argument is identical to the one in case (i) above, so the details are 
omitted.

\medskip

(iii) Once we have the bounds in the elliptic region, we can truncate
in frequency to the hyperbolic region $\xi \approx \xi_0$, using the elliptic 
bounds to estimate the truncation errors in the system \eqref{S-system}. Thus we will assume from here on that both $(w,q)$ are localized at frequency $\xi_0$. This localization will 
destroy the spatial localization, but we will neglect this in the  analysis that follows since the generated tails are of size $t^{-N}$ and rapidly decreasing.

\bigskip

\emph{Step 4(b): The hyperbolic analysis.}
Here, as discussed in (iii) above, we assume that $(w,q)$ are frequency localized at dyadic frequency $\xi_0$,  spatially localized in the dyadic region  $|\alpha| \approx \alpha_0$,
and the right hand side in the equation satisfies the same bounds 
as in the theorem.  It suffices to prove the desired pointwise bound 
for $w$, as $q_\alpha$ can then be obtained directly from either of the 
equations in \eqref{S-system}. 

Our next step is to eliminate $q_\alpha$ from the two equations.
To do this we use the paraproduct product and commutator formulas 
from our previous paper \cite[Lemmas 2.4, 2.5]{AIT}. This gives
\begin{equation}\label{wa-eqn}
(4 \alpha^2    - 8\alpha t  T_{\Re \tQ_\alpha} + 4t^2 T_{(\Re \tQ_\alpha)^2}) w_\alpha  + i t^2(1 -  T_{2 \Re \tW_\alpha}) w =  2 \alpha g_1,
\end{equation}
where $g_1$ is given by
\[
2 \alpha g_1 =  2 (\alpha - t T_{ \Re \tQ_\alpha}) g + t(1-2T_{\Re \tW_\alpha}) k 
+ t^2 (L w_\alpha +  M q_{\alpha}),
\]
with
\[
L = 4 (T_{ \Re \tQ_\alpha} T_{\Re \tQ_\alpha} - T_{(\Re \tQ_\alpha)^2}), \qquad 
M = [T_{2 \Re \tW_\alpha}, T_{2 \Re \tQ_\alpha}].
\]
We can show that $g_1$ and $g$ are essentially equivalent:

\begin{lemma}
The function $g_1$ satisfies the same bounds as $g$,
\begin{equation}
\| g_1 \|_{\dH^\frac14} \lesssim    \left(\| (w,q)\|_{\dot \H^\frac14 \cap \dot H^\sigma} 
+ \| (g,k)\|_{\dot \H^\frac14} \right).
\end{equation}
\end{lemma}

\begin{proof} Here it is easiest to use a result from \cite{AIT}, precisely Lemma 2.5 there. 
Applied with $\gamma_1 = \gamma_2 = \frac34$ and using our bootstrap bound \eqref{boot-nf} it yields
\[
\| L P_{\xi_0}\|_{L^2 \to L^2} \lesssim \xi_0^{-\frac32} \|D^\frac34 \tQ_\alpha \|_{BMO}^2  \lesssim 
\frac{1}{t} \xi_0^{-\frac32} A_{1/4}^2.
\]
This is exactly as needed since in our case the argument of $L$ is spatially localized in the 
region $\alpha \approx \alpha_0$, so the output has a similar localization modulo tails which 
decrease rapidly on the $\xi_0^{-1}$ scale.

For $M$ we also use \cite[Lemma 2.5]{AIT} but now with $\gamma_1 = \frac14$ and $\gamma_2 = \frac34$,
where the former corresponds to $\tW_\alpha$ and the latter to $\tQ_\alpha$. We obtain
\[
\| L P_{\xi_0}\|_{L^2 \to L^2} \lesssim \xi_0^{-\frac32} \|D^\frac14 \tW_\alpha \|_{BMO}
\|D^\frac34 \tQ_\alpha \|_{BMO}  \lesssim 
\frac{1}{t} \xi_0^{-1} A_{1/4}^2,
\]
which again suffices.
\end{proof}

Consider now \eqref{wa-eqn}, which we rewrite in a shorter form
\begin{equation}\label{wa-eqn-re}
(1- T_{V_1}) w_\alpha  + i \frac{t^2}{\alpha_2} (1 -  T_{V_2}) w =  (2 \alpha)^{-1} g_1,
\end{equation}
where the potentials $V_1$ and $V_2$ are given by 
\[
V_1 = -  \frac{2t}{\alpha} \Re \tQ_\alpha + \frac{ t^2}{\alpha^2} (\Re \tQ_\alpha)^2,
\qquad V_2 = 2 \Re W_\alpha.
\]

We carry out another reduction, which is to eliminate the paracoefficient of $w_\alpha$.
This is achieved by applying the operator $1 + T_{\frac{V_1}{1-V_1}}$ in \eqref{wa-eqn-re}.
Using again paraproduct calculus exactly as in the above lemma,  \eqref{wa-eqn-re} is rewritten as
\begin{equation}\label{wa-eqn-re+}
 w_\alpha  + i \frac{t^2}{\alpha^2} w  -  i \xi_0 T_V w =  2 \alpha^{-1} g_2,
\end{equation}
where $g_2$ satisfies the same bound as $g_1$ and $V$ is given by 
\[
V = \frac{t^2}{\alpha^2} \xi_0^{-1} \frac{V_2-V_1}{1-V_1}.
\]
Here we pulled out the $\xi_0$ factor because in the region of interest $|\alpha| \approx \alpha_0$
we have $\frac{t^2}{\alpha^2} \xi_0^{-1} \approx 1$. The contributions of $V$ outside a size $\alpha_0$
neighbourhood of this region have size $O(t^{-N})$ and can be harmlessly discarded.

In view of \eqref{boot-nf}, in the above region the potential $V$ is \emph{real valued} and 
has the following properties:
\begin{enumerate}[label=\roman*)]
\item  small size, 
\[
\| V\|_{L^\infty} \ll   \xi_0^\frac12 t^{-\frac12} \lesssim  t^{-\frac14},
\]

\item smaller gradient
\[
\| P_{< \xi_0} V_{\alpha}\|_{L^\infty} \ll t^{-\frac12} \xi_0^\frac34. 
\]

\end{enumerate}

For solutions to the equation \eqref{wa-eqn-re+}  we seek to prove a uniform bound of the form
\begin{equation}\label{cxi0}
\| w\|_{L^\infty} \lesssim  c(\xi_0) t^{-\frac12} \xi_0^{-\frac54} (\| w\|_{\dot H^{\frac{11}4}}+ \|w\|_{\dot H^\frac14}
+ \| g_2\|_{\dot H^\frac14}),
\end{equation}
where the $\xi_0^{-\frac54}$ factor corresponds to the $A_{\frac14}$ norm while
$c(\xi_0)$ denotes any additional gain, as required by the $X^\sharp_{\alpha_0}$ norm.

Here we distinguish two cases depending on the size of $\xi_0$. 

\bigskip

\emph{ Case 1: $\xi_0 > 1$.} Then we will establish a  bound based 
on the $\dot \H^{\frac{11}4}$ norm for $w$, and will instead show that
\begin{equation}\label{large-xi}
\| w\|_{L^\infty} \lesssim t^{-\frac12} \xi_0^{-\frac54} \| w\|_{\dot H^{\frac{11}4}}^\frac12 \| g_2\|_{\dot H^\frac14}^\frac12.
\end{equation}
Since $w$ is localized at frequency $\xi_0$, replacing the $\dot \H^{\frac{11}4}$ norm 
with $\dot H^\sigma$ with $\sigma > \frac{11}{4}$ yields a gain of $c(\xi_0)
= \xi_0^{\frac12(\frac{11}{4}-\sigma)}$ in \eqref{cxi0}, which exactly corresponds
to our choice of  $b$ in the $X^\sharp$ norm. 

Taking into account the localization at frequency $\xi_0$, we can replace 
the Sobolev norms by $L^2$ norms in \eqref{large-xi}, and rewrite it as 
\begin{equation*}
\| w\|_{L^\infty} \lesssim t^{-\frac12} \xi_0^{-\frac54} (\xi_0^\frac{11}4 \| w\|_{L^2})^\frac12 (
t \xi_0^{-\frac14} \| \alpha^{-1} g_2\|_{L^2})^\frac12, 
\end{equation*}
or equivalently, as a bound for solutions to \eqref{wa-eqn-re+}, as
\begin{equation}\label{w-need}
\| w\|_{L^\infty} \lesssim  \|w\|_{L^2}^\frac12
 \| \alpha^{-1} g_2\|_{L^2}^\frac12.
\end{equation}
We postpone the proof of this bound in order to discuss the second case.

\bigskip

\emph{ Case 2: $\xi_0 < 1$.} Then we will establish a  bound based 
on the $\dot \H^{\frac{1}4}$ norm for $w$, and will show that
\begin{equation}\label{small-xi}
\| w\|_{L^\infty} \lesssim t^{-\frac12} \| w\|_{\dot H^{\frac{1}4}}^\frac12 \| g_2\|_{\dot H^\frac14}^\frac12. 
\end{equation}
This corresponds to choosing $c(\xi_0) = \xi_0^\frac54$ in \eqref{cxi0}, which in turn 
corresponds to our choice of  $a$ in the $X^\sharp$ norm. 
Taking into account the localization at frequency $\xi_0$, and the spatial localization at $|\alpha|\approx \alpha_0$, this bound also reduces
to \eqref{w-need}.

It remains to prove the bound \eqref{w-need} for solutions to \eqref{wa-eqn-re+}.
Here the paradifferential coefficients are nonperturbative.
Part of the difficulty is also the fact that these  coefficients
are in paradifferential form. If that were not the case, then we could simply take advantage of the \emph{critical} fact that they are real, calculate 
\[
\partial_\alpha |w|^2 = 2 \Re w \cdot  \frac{1}{2 \alpha} g_2,
\]
and integrate to get 
\[
\| w \|_{L^\infty}^2 \lesssim 
\| w \|_{L^2} \| \alpha^{-1} g_2 \|_{L^2},
\]
as needed.

To prove \eqref{w-need}, we discard the spatial localization and restate the result in a simpler form:

\begin{lemma}
Suppose that the function $u \in L^2$ is localized at frequency $\xi_0$ and solves the equation
\begin{equation}
u_\alpha + i T_V^w u = f,    
\end{equation}
where the potential $V$ is real and satisfies 
\begin{equation}
 V \approx \xi_0
\end{equation}
and
\begin{equation}
\| \nabla V\|_{L^\infty} \lesssim M  \ll \xi_0^2.  
\end{equation}
Then we have the pointwise bound
\begin{equation}
\|u\|_{L^\infty}^2 \lesssim  \|u\|_{L^2}\|f\|_{L^2}. 
\end{equation}
\end{lemma}

\begin{proof}

For a suitable smooth, bounded and nondecreasing function $\chi$ we multiply the 
equation by $\chi u$ and integrate by parts.
We get 
\[
\frac12 \int \chi' |u|^2 \, d\alpha = 
t^{-1} \Re \int \chi u \bar f \,  d \alpha + \Re \int i [T_V^w, \chi] u \cdot  \bar u \, d\alpha.    
\]
To insure that the term on the left nonnegative we choose $\chi$
increasing from $0$ to $1$ in an interval $I$ of a fixed length $r$,
and constant elsewhere. Here $r$ is chosen 
above the uncertainty principle threshold $r > \xi_0^{-1}$. Then we have 
\[
|\chi'| \lesssim r^{-1},
\]
and $\chi'$ is further supported in $I$.
Then the commutator has $L^2$ size 
\[
\| [T_V^w, \chi] P_{\xi_0} \|_{L^2 \to L^2} \lesssim \xi_0^{-2} \| P_{<\xi_0} V_\alpha \|_{L^\infty}
\| \chi'\|_{L^\infty} \lesssim 
\xi_0^{-2} M r^{-1}.
\]
Further, we observe that the commutator is essentially localized in $2I$, modulo rapidly decreasing tails on the $\xi_0^{-1}$ scale. We can account for the rapidly decreasing tails using translates of the interval $I$, which has size at least $\xi_0^{-1}$. Then we arrive at the estimate
\[
\| [T_V^w, \chi] P_{\xi_0} u \|_{L^2} \lesssim 
\xi_0^{-2} M r^{-1} \sup_{c \in \R} \| u\|_{L^2(I+c)}
\]
Hence from the previous integral identity we obtain 
\[
 r^{-1} \|u\|_{L^2(I)}^2 \lesssim \| u\|_{L^2} \|f\|_{L^2} + 
\xi_0^{-2} M r^{-1} \sup_{c \in \R} \| u\|_{L^2(I+c)}^2,
\]
But we can also apply this bound with $I$ replaced by translates of $I$. This yields
\[
 r^{-1} \sup_{c \in \R}  \|u\|_{L^2(I+c)}^2 \lesssim \| u\|_{L^2} \|f\|_{L^2} + 
\xi_0^{-2} M r^{-1} \sup_{c \in \R} \| u\|_{L^2(I+c)}^2.
\]
Since $M \xi_0^{-1} \ll 1$, we can absorb the second term on the right on the left, 
to obtain
\[
r^{-1} \sup_{c \in \R} \|u\|_{L^2(I+c)}^2 \lesssim \| u\|_{L^2} \|f\|_{L^2}. 
\]
Using the frequency localization of $u$ as well as the bound $r > \xi_0^{-1}$, 
this yields a similar bound for the derivative of $u$, namely
\[
r^{-1} \sup_{c \in \R} \|u_\alpha\|_{L^2(I+c)}^2 \lesssim \xi_0 \| u\|_{L^2} \|f\|_{L^2}. 
\]
One may obtain $L^\infty$ bounds for $u$ 
in any interval $I+c$ from $L^2$
bounds for $u$ and $u_\alpha$ in the same interval,
\[
\| u \|_{L^\infty(I+c)}^2 \lesssim  r^{-1} \|u\|_{L^2(I+c)}^2  +
r \| u_\alpha \|_{L^2(I+c)}^2.
\]
Then we arrive at
\[
\begin{aligned}
\sup_{c \in \R} \| u \|_{L^\infty(I+c)}^2 \lesssim  (1+r^2 \xi_0^2)  \| u\|_{L^2} \|f\|_{L^2}.
\end{aligned}
\]
Choosing $r$ as small as possible, 
\[
r \approx \xi_0^{-1},
\]
we finally obtain
\[
\| u\|_{L^\infty}^2 \lesssim  \| u\|_{L^2} \|f\|_{L^2},
\]
as desired, concluding the proof of the lemma.
\end{proof}

Once we have the above Lemma, we can apply it to prove \eqref{w-need}, which in turn concludes the proof 
of the proposition.
\end{proof}

To complete our discussion of the $X^\sharp$ bounds we need to compare them with the $X$
bounds. This is best carried out in terms of the elliptic-hyperbolic decomposition
\eqref{wq=ell+hyp}.

\bigskip

We begin with the hyperbolic part, for which 
we have that its $X$ size is controlled by the full $X^\sharp$
norm, with an additional gain away from unit velocity. 
To quantify this gain we use the  region $\Omega^\delta$
defined in \eqref{Omega-delta}, and denote by $\chi_{\Omega^\delta}$ a bump function which selects the region $\Omega^\delta$ and is smooth at both ends on the appropriate dyadic $\alpha$ scale. Then we have:

\begin{proposition} 
We have the bounds
\begin{equation}\label{XbelowXsharp}
\| (w,q)_{hyp,\alpha} \|_{X} \lesssim \|(w,q)\|_{X^\sharp},    
\end{equation}
respectively 
\begin{equation}\label{v-not-1}
\| (1-\bchi_{\Omega^\delta}) (w,q)_{hyp,\alpha} \|_{X} \lesssim t^{-b\delta} \|(w,q)\|_{X^\sharp}.    
\end{equation}
\end{proposition}
\begin{proof}
 For the hyperbolic part we use the decomposition in \eqref{hyp-decomp}, which we recall here:
\[
  (w_{hyp},q_{hyp}) =  \sum_{ \alpha_{lo} \lesssim \alpha_0 \lesssim \alpha_{hi}} P_{\xi_0} \bchi_{\alpha_0} (w,q).
\]
We now distinguish between small and large velocities:
\medskip

\textit{ a) Small velocities: $\alpha_{lo} < \alpha_0 \lesssim  t$.} This corresponds to dyadic velocities 
$v_0 = \alpha_0/t$ in the range
\[
 \frac{\alpha_{lo}}{t} = v_{lo} < v_0 \lesssim 1,
\]
and to frequencies $\xi_0 = v_0^{-2} \gtrsim 1$.

In this case, we use the $X^\sharp$ norm component given by the last term in \eqref{Xsharp-small}. Since $\xi_0 \gtrsim 1$, at frequency $\xi_0$ the $X$ norm agrees with the $X^0$ norm, so we obtain
\begin{equation}
\|    P_{\xi_0} \bchi_{\alpha_0} (w,q) \|_{X}
\lesssim \xi_0^{-b} \|(w,q)\|_{X^\sharp} = v_0^{2b} \|(w,q)\|_{X^\sharp}.
\end{equation}
This suffices directly for \eqref{XbelowXsharp}, while 
in \eqref{v-not-1} we capture the extra gain due to 
the truncation to the range $v_0 < t^{-\delta}$.

\medskip

\textit{ b) Large velocities: $t \lesssim  \alpha_0 <  \alpha_{hi}$.} This corresponds to dyadic velocities 
$v_0$ in the range
\[
1 \lesssim v_0 < v_{hi}= \frac{\alpha_{hi}}{t},
\]
and to frequencies $\xi_0 = v_0^{-2} \lesssim 1$.

Now we use instead the $X^\sharp$ norm component given by the last term in \eqref{Xsharp-small}. Since $\xi_0 \lesssim 1$, at frequency $\xi_0$ the $X$ norm 
is $\xi_0^{-\frac34}$ times the $X^0$ norm, so we obtain
\begin{equation}
\|    P_{\xi_0} \bchi_{\alpha_0} (w,q) \|_{X}
\lesssim \xi_0^{a - \frac34} \|(w,q)\|_{X^\sharp} = v_0^{-2(a-\frac34)} \|(w,q)\|_{X^\sharp}.
\end{equation}
This suffices directly for \eqref{XbelowXsharp}, while 
in \eqref{v-not-1} we capture the extra gain due to 
the truncation to the range $v_0 > t^{\delta}$.
\end{proof}

\bigskip
Next we consider the elliptic part of $(w,q)$, where we have a simpler objective, namely 
to show that it satisfies better bounds both in the energy sense and in the pointwise 
sense.

\begin{proposition}\label{p:Xell}
Let $(w,q)$ be a pair of functions satisfying
\begin{equation}\label{X-ell}
t^\frac12 \|(w,q)\|_{X^\sharp_{ell}} + \|(w,q)\|_{\dH^\frac14 \cap \dH^\sigma}   \leq 1.
\end{equation}
Then we have the energy bound
\begin{equation}\label{ell(wq)-energy}
\| (w_\alpha,q_\alpha) \|_{\dH^{-\frac14}} \lesssim t^{-1},
\end{equation}
as well as the uniform bound
\begin{equation}\label{ell(wq)-point}
\| (w,q)_\alpha \|_{X} \lesssim t^{-b}.     
\end{equation}
\end{proposition}

\begin{proof}
For $(w,q)$ we consider the decomposition \eqref{(wq)-dec}, and prove the desired 
bound separately for each frequency. The $L^2$ estimate \eqref{ell(wq)-energy} is trivial; we have only added it in the proposition for easy reference. For the uniform bound \eqref{ell(wq)-point}, on the other hand, we need to appropriately apply Bernstein's 
inequality.
\medskip

\emph{a) The low $\alpha$ component,} 
\[
(w_{lo},q_{lo}) = \bchi_{\ll \alpha_{lo}} (w,q).
\]

Here by by \eqref{X-ell} and by the definition of the $X^\sharp$ norm, see \eqref{X-lo}, we control
\[ 
t \|(w_{lo},q_{lo})\|_{\dH^\frac34} + \|(w_{lo},q_{lo})\|_{\dH^\frac14 \cap \dH^\sigma}   \lesssim 1.
\]
Then we can bound uniformly the dyadic pieces
of $(w_{lo},q_{lo})$ using Bernstein's inequality as follows, neglecting the $\dH^\frac14$ norm:
\[
\| P_\mu (w_{lo},|D|^\frac12 q_{lo})\|_{L^\infty} \lesssim 
\min\{ \mu^{-\sigma+\frac12}, t^{-1} \mu^{-\frac14}\}, % \| (w_{lo},q_{lo})\|_{X^\sharp_{lo}},
\]
where the first component is smaller if $\mu > \xi_{lo}= t^\frac12$.

Hence, for the high frequency part of the $X$ norm 
(i.e. at frequencies $\geq 1$) we have 
\[
\|P_{\geq 1}(|D|^\frac54 w_{lo},|D|^\frac74 q_{lo})\|_{B^{0}_{\infty,2}}^2 \lesssim
\sum_{\mu \geq 1} \min\{ \mu^{-\sigma+\frac74}, t^{-1} \mu\}^2 = t^{-1-2\delta_{lo}}, 
\qquad \delta_{lo} = \frac{\sigma - \frac{11}4}{2(\sigma - \frac34)}
\]
as needed. The bound for the low frequency part of the $X$ norm is similar
but better.

\medskip
\emph{ b) The intermediate $\alpha$ component.} Here we fix a dyadic region 
$|\alpha| \approx \alpha_0 \in [\alpha_{lo},\alpha_{hi}]$ and consider the component 
\[
(w_{mid},q_{mid}) = \bchi_{\alpha_0} (w,q),
\]
which is in turn decomposed into low frequencies ($< \xi_0$) and high frequencies ($> \xi_0$):
\medskip

\emph{b)(i). Low frequencies, $\xi < \xi_0$.} 
 Here by the second term in  \eqref{wq-ell} along with the 
$\dH^\sigma$ bound in \eqref{X-ell}, we have
\[
t \|P_{<\xi_0} (w_{mid}, q_{mid})\|_{\dH^{\frac34}}
+ \|P_{<\xi_0} (w_{mid}, q_{mid})\|_{\dH^{\sigma}} \lesssim 1
\]
We split into dyadic frequency regions $\mu < \xi_0$,  and use Bernstein's inequality to estimate
\[
\|P_\mu (w_{mid},|D|^\frac12 q_{mid})\|_{L^\infty} \lesssim 
 \min \{t^{-1} \mu^{-\frac14}, \mu^{-\sigma+\frac12}\} 
 \| (w_{mid},q_{mid})\|_{X^\sharp_{hi}}.
\]
Neglecting the $\sigma$ component, after dyadic $\mu$ summation this implies 
\[
\|P_{<\xi_0} (|D|^\frac12 w_{mid},|D| q_{mid})\|_{L^\infty} \lesssim 
 t^{-1} \xi_0^{\frac14}= t^{-1} |v_0|^{-\frac12}.
\]
Since $\xi_0 < \xi_{lo} = t^\frac12$, this suffices for the low frequency part of the $X$ norm, and completes the argument if $\xi_0 < 1$, i.e. if $v_0 > 1$.

However, if $\xi_0 > 1$ then we also need to consider the high frequency part of the $X$ norm, where  we can no longer neglect the $\sigma$ term.
Hence we write instead
\[
\|P_{<\xi_0} (|D|^\frac54 w_{mid},|D|^\frac74 q_{mid})\|_{B^0_{\infty,2}}^2 \lesssim 
\sum_{\mu < \xi_0} \min\{    t^{-1} \mu, \mu^{-\sigma+\frac74} \}^2. 
\]
We neglect the $\mu$ range and bound this by the maximum of the right hand side summand,
\[
\|P_{<\xi_0} (|D|^\frac54 w_{mid},|D|^\frac74 q_{mid})\|_{B^0_{\infty,2}}^2 \lesssim t^{-1-2\delta_{lo}} 
\]
with $\delta_{lo}$ exactly as in case (a).

\medskip

\emph{b)(ii). High frequencies, $\xi > \xi_0$.}
Here we use instead the first term in \eqref{wq-ell}; then we have 
\[
t \xi_0^{-\frac12} \|P_{>\xi_0} (w_{mid}, q_{mid})\|_{\dH^{\frac54}}
+ \|P_{<\xi_0} (w_{mid}, q_{mid})\|_{\dH^{\sigma}} \lesssim 1.
\]
Then we use Bernstein's inequality to estimate for $\mu > \xi_0$
\[
\|P_\mu (w_{mid},|D|^\frac12 q_{mid})\|_{L^\infty} \lesssim 
 \min \{t^{-1} \xi_0^\frac12 \mu^{-\frac34}, \mu^{-\sigma+\frac12}\}.
\]
Neglecting the $\sigma$ component, after dyadic $\mu$ summation this implies 
\[
\|P_{>\xi_0} (|D|^\frac12 w_{mid},|D| q_{mid})\|_{L^\infty} \lesssim 
 t^{-1} \xi_0^{\frac14}= t^{-1} |v|^{-\frac12},
\]
which, as before in case (b)(i), suffices for the low frequency part of the $X$ norm. For 
the high frequency part of the $X$ norm we again can no longer 
neglect the $\sigma$ term, so we write instead
\[
\|P_{>\xi_0} (|D|^\frac54 w_{mid},|D|^\frac74 q_{mid})\|_{B^{0}_{2,\infty}}^2 \lesssim 
\sum_{\mu > \xi_0} \min\{    t^{-1} \xi_0^\frac12 \mu^\frac12, \mu^{-\sigma+\frac74} \}^2. 
\]
Replacing $\xi_0$ by $\mu$ we arrive exactly at the same computation as 
in case (b)(i), in which the $\mu$ range was neglected.

\medskip
\emph{ c) The high $\alpha$ component,} 
\[
(w_{hi},q_{hi}) = \bchi_{\gg \alpha_{hi}} (w,q).
\]
Here we combine the expression in \eqref{X-hi} with the second term in \eqref{X-ell} to obtain
\[
t^2 \| (w_{hi},q_{hi})\|_{\dH^\frac54}+ \|(w,q)\|_{\dH^\frac14 \cap \dH^\sigma}   \lesssim  1.
\]
Then, using Bernstein's inequality we have:
\[
\| P_\mu (w_{hi},|D|^\frac12 q_{hi})\|_{L^\infty} \lesssim 
\min\{\mu^{\frac14}, \mu^{-\sigma+\frac12}, t^{-2} \mu^{-\frac34}\}  \| (w_{hi},q_{hi})\|_{X^\sharp_{hi}}.
\]
For the low frequency part of the $X$ norm we neglect the $\sigma$ component to get
\[
\|(|D|^\frac12 w_{hi},|D| q_{hi})\|_{L^\infty} \lesssim
\sum_{\mu} \min\{ \mu^\frac34, t^{-2} \mu^{-\frac14}\} = t^{-\frac32},
\]
better than needed.

The estimate for the high frequency part of the $X$ norm  is similar,
\[
\|(|D|^\frac54 w_{hi},|D|^\frac74 q_{hi})\|_{B^0_{\infty,2}}^2 \lesssim
\sum_{\mu} \min\{ \mu^\frac32, \mu^{-\sigma+\frac74}, t^{-2} \mu^{\frac12}\}^2 \lesssim t^{-\frac83},
\]
as needed. Here we have instead neglected the first term, and replaced $\sigma$ by $\frac{11}{4}$.
\end{proof}

\subsection{Back to the normal form variables}
We now return to $(\tW,\tQ)$, and we  apply the results of the  previous subsection to them, both directly and in terms of the corresponding elliptic-hyperbolic decomposition
\eqref{wq=ell+hyp}.

\begin{corollary}\label{c:KS-para}
Assume that $(W,Q)$ satisfy the bootstrap bound \eqref{pointwise-bootstrap}.
Then the following pointwise bound holds:
\begin{equation}\label{KS-para-apply}
\| (\tW_\alpha, \tQ_\alpha) \|_{X^\sharp}
\lesssim t^{-\frac12} \left(\| (\tW,\tQ)\|_{\dot \H^\frac14 \cap \dot H^\sigma} 
+ \| \tS(\tW,\tQ)\|_{\dot \H^\frac14} \right).
\end{equation}
In particular, if \eqref{KS-control} holds then
\begin{equation}
\| (\tW, \tQ) \|_{X^\sharp}
\lesssim \epsilon t^{-\frac12+C\epsilon}.
\end{equation}

Furthermore, its hyperbolic an elliptic components
satisfy bounds as follows:
\begin{equation}\label{tWQ-hyp-out}
\| (1-\chi_{\Omega^\delta}) (\tW,\tQ)_{hyp} \|_{X} \lesssim \epsilon t^{-\frac12-b\delta + C \epsilon^2},    
\end{equation}
respectively  elliptic $L^2$ and $L^\infty$ bounds
\begin{equation}\label{tWQ-ell2}
\| (\tW_{ell}, \tQ_{ell}) \|_{X^\sharp_{ell}}
\lesssim \epsilon^2 t^{-\frac12+C\epsilon^2},
\end{equation}
\begin{equation}\label{tWQ-ell3}
\| (\tW_{ell,\alpha}, \tQ_{ell,\alpha}) \|_{\dH^{-\frac14}}
\lesssim \epsilon^2 t^{-1+C\epsilon^2},
\end{equation}
and
\begin{equation}\label{tWQ-ell4}
\| (\tW_{ell,\alpha}, \tQ_{ell,\alpha}) \|_{X}
\lesssim \epsilon^2 t^{-\frac12 - b/2 +C\epsilon^2}.
\end{equation}

\end{corollary}

The reason we care about the better bounds for the elliptic part is that its 
contribution to the analysis of the normal form equation \eqref{tWQ-system} is 
mostly perturbative. This is fully the case for the cubic source terms
$(\tG^{(3)},\tK^{(3)})$, but also to some extent for the paradifferential quadratic terms. 
Indeed, an interesting observation is that in all paradifferential interactions 
in the original system \eqref{KS-system} 
for $(\tW,\tQ)$, at least one of the two inputs 
has to be in the elliptic region. Precisely, we have

\begin{proposition}\label{p:errors}
a) Denote by $(\tG^{(3)}_{ell},\tK^{(3)}_{ell})$ any expression with one elliptic entry, 
i.e. of the form
\[
\tG^{(3)}(\tW_{ell,\alpha}, \tW_\alpha, \tQ_\alpha), \qquad \tG^{(3)}(\tW_\alpha, \tW_\alpha, \tQ_{ell,\alpha}),
\]
and similarly for $\tK^{(3)}_{ell}$. Then we have
\begin{equation}\label{GK3-ell}
\| (\tG^{(3)}_{ell},\tK^{(3)}_{ell})\|_{\dH^\frac14} \lesssim ( \| (\tW_{ell},\tQ_{ell})
\|_{X^\sharp_{ell}}+ t^{-\frac12}\| (\tW_{ell},\tQ_{ell})
\|_{\dH^\frac14 \cap \dH^\frac{11}{4}}) \|(\tW,\tQ)\|_{X}^2.
\end{equation}

b) For the paradifferential term we have the improved bound
\begin{equation}\label{better-para}
\|(T_{\Re \tW_\alpha} \tQ_\alpha - T_{\Re \tQ_\alpha} \tW_\alpha , T_{\Re \tQ_\alpha} \tQ_\alpha) \|_{\dH^\frac14} \lesssim \epsilon^2 t^{-\frac54+2C \epsilon^2}.    
\end{equation}
\end{proposition}

\begin{proof}

a) For the elliptic entry we use only the translation invariant part of the 
$X^\sharp_{ell}$ norm, i.e. 
\[
\| (\tW_{ell},\tQ_{ell}) \|_{\dH^{\frac34}} \lesssim t^{-\frac12} 
\| (\tW_{ell},\tQ_{ell})\|_{X^\sharp_{ell}}.
\]
Interpolating the $\dH^{\frac34}$ norm with the $\dH^\frac14 \cap \dH^\frac{11}{4}$ norm we obtain
\[
\| (\tW_{ell,\alpha},\tQ_{ell,\alpha}) \|_{\dH^{s}} \lesssim t^{-\frac12} 
\| (\tW_{ell},\tQ_{ell})\|_{\dH^\frac14 \cap \dH^\frac{11}{4}} +
\| (\tW_{ell},\tQ_{ell})\|_{X^\sharp_{ell}}, \qquad - \frac12 \leq s \leq \frac34.
\]
It remains to show that
\begin{equation}\label{GK3-ell+}
\| (\tG^{(3)}_{ell},\tK^{(3)}_{ell})\|_{\dH^\frac14} \lesssim 
\| (\tW_{ell},\tQ_{ell})
\|_{\dH^\frac12 \cap \dH^\frac{7}{4}}  \|(\tW,\tQ)\|_{X}^2.
\end{equation}
We consider the following three cases:

\bigskip

\emph{i) The elliptic variable is the lowest frequency.} Beginning with $\tG^{(3)}$, we have following three prototypical terms in $\tG^{(3)}$,
\[
T_{T_{\tQ_\alpha} \tW_\alpha} \tW_\alpha, \qquad T_{T_{\tW_\alpha} \tQ_\alpha} \tW_\alpha
\qquad T_{T_{\tW_\alpha} \tW_\alpha} \tQ_\alpha,
\]
noting that the cases when two of the frequencies are matched are entirely similar to these.

For the first of these cases, we estimate
\begin{equation*}
\begin{aligned}
\| T_{T_{\tQ_\alpha} \tW_{\alpha}} \tW_{\alpha} \|_{\dot H^\frac14} &\lesssim \| \tQ_{\alpha}\|_{L^2} \| \tW_{\alpha} \|_{L^\infty} \|D^\frac14 \tW_\alpha\|_{BMO}, 
\end{aligned}
\end{equation*}
as needed. For the second of these cases, we estimate
\begin{equation*}
\begin{aligned}
\| T_{T_{\tW_\alpha} \tQ_{\alpha}} \tW_{\alpha} \|_{\dot H^\frac14} &\lesssim \|\tW_{\alpha}\|_{L^2} \|\tQ_{\alpha} \|_{L^\infty} \|D^\frac14 \tW_\alpha\|_{BMO},
\end{aligned}
\end{equation*}
and likewise the third,
\begin{equation*}
\begin{aligned}
\| T_{T_{\tW_\alpha} \tW_{\alpha}} \tQ_{\alpha} \|_{\dot H^\frac14} &\lesssim \|\tW_{\alpha}\|_{L^2} \|\tW_{\alpha} \|_{L^\infty} \|D^\frac14 \tQ_\alpha\|_{L^\infty}.
\end{aligned}
\end{equation*}
All terms of $\tG^{(3)}$ are similar to one of these cases, or may have an additional derivative falling on the elliptic variable. This last situation is estimated in the same way as one of the above cases, since we are free to rebalance the derivative.

We continue with $\tK^{(3)}$. Here we have four prototypical terms, 
\[
T_{T_{\tQ_\alpha} \tW_\alpha} \tQ_\alpha, \qquad T_{T_{\tW_\alpha} \tQ_\alpha} \tQ_\alpha, \qquad T_{T_{\tW_\alpha} \tW_\alpha} \tW, \qquad T_{\D_\alpha T_{\tQ_\alpha} \tQ_\alpha} \tW,
\]
observing that terms of the form $T_{T_{\tQ_\alpha} \tQ_\alpha} \tW_\alpha$ cancel. The analysis of the first two terms is analogous to the first two terms discussed for $\tG^{(3)}$. For the third term,
\begin{equation*}
\begin{aligned}
\| T_{T_{\tW_\alpha} \tW_{\alpha}} \tW \|_{\dot H^\frac34} &\lesssim \|\tW_{\alpha}\|_{L^2} \| \tW_{\alpha} \|_{L^\infty}\||D|^{3/4} \tW \|_{L^\infty}
\end{aligned}
\end{equation*}
suffices, and for the last term,
\begin{equation*}
\begin{aligned}
\| T_{\D_\alpha T_{\tQ_\alpha} \tQ_{\alpha}} \tW \|_{\dot H^\frac34} &\lesssim \| \tQ_{\alpha}\|_{L^2} \| |D|^\frac12 \tQ_{\alpha} \|_{L^\infty} \|\tW_\alpha\|_{B^{\frac14}_{\infty,2}}.
\end{aligned}
\end{equation*}

\bigskip

\emph{ii) The elliptic variable is the middle frequency.} The analysis in this case is similar to the analysis in the first case, except we measure in each case the middle frequency term in $L^2$.
 
\bigskip

\emph{iii) The elliptic variable is the highest frequency.} For $\tG^{(3)}$, we directly measure the two lower frequency variables in $X$. For instance,
\[
\| T_{T_{\tW_\alpha} \tQ_\alpha}  \tW_{\alpha} \|_{\dot H^\frac14}
\lesssim  \|\tW_{\alpha}\|_{L^\infty} \| \tQ_{\alpha}\|_{L^\infty} \| \tW_{\alpha} \|_{\dot H^\frac14}.
\]
The analysis of $\tK^{(3)}$ is similar. For instance (using here the boundedness of $\D_\alpha P$),
\[
\| T_{ P[|\tQ_\alpha|^2]_\alpha} \tW \|_{\dot H^\frac14}
\lesssim  \| \tQ_{\alpha}\|_{L^\infty} \| \tQ_{\alpha}\|_{L^\infty} \| \tW_{\alpha} \|_{\dot H^\frac14}.
\]

\bigskip

b) We use the elliptic-hyperbolic decomposition of $(\tW,\tQ)$, noting that the above 
expressions only allow for low-high interactions, therefore the hyperbolic $\times$ hyperbolic case is forbidden. We separately consider each of the three remaining cases:
\bigskip

\emph{i) The elliptic-hyperbolic case.}
We consider a dyadic region $|\alpha| \approx \alpha_0$, and the 
corresponding localized components of $(\tW_{ell},\tQ_{ell})$, respectively
$(\tW_{hyp},\tQ_{hyp})$.  There we need to estimate the quadratic terms:
\[
\| T_{2\Re \tQ_\alpha} \tW_{\alpha,\xi_0} \|_{\dot H^\frac14}
\lesssim \xi_0^{\frac14} \| \tQ_{\alpha,< \xi_0}\|_{L^4} \| \tW_{\alpha,\xi_0} \|_{L^4}
\lesssim  t^{-\frac12} 
\| \tQ \|_{X_{ell}^\sharp} \| \tW_\alpha \|^\frac12_{X} \| \tW \|_{\dH^{\frac54}}^\frac12
\lesssim \epsilon^2 t^{-\frac54+ 2C\epsilon},
\]
which suffices.

The bound for $T_{\Re \tQ_\alpha} \tQ_\alpha$ is identical. Finally,
\[
\| T_{\Re \tW_\alpha} \tQ_\alpha \|_{\dot H^\frac14} \lesssim  
\xi_0^{\frac14} \| \tW_{\alpha,< \xi_0}\|_{L^2} \| \tQ_{\alpha,\xi_0} \|_{L^\infty} \lesssim
t^{-\frac12}  \| \tW \|_{X_{ell}^\sharp} \| \tQ_\alpha \|_{X}
\lesssim \epsilon^2 t^{-\frac32+2C\epsilon}. 
\]
\bigskip

\emph{ii) The hyperbolic-elliptic case.}
After  localizing the high frequency factor at a frequency $\mu > \xi_0$, here 
we need to bound  the dyadic $\xi_0$ - $\mu$ interactions as follows:
\[
\| \Re \tQ_{\alpha,{\xi_0}} \tW_{\alpha,\mu} \|_{\dot H^\frac14}
\lesssim  \| \tQ_{\alpha,\xi_0}\|_{L^\infty} \| \tW_{\alpha,\mu} \|_{\dot H^\frac14}
\lesssim t^{-\frac12} 
\| \tQ_\alpha \|_{X} \| \tW \|_{X_{ell}^\sharp},
\]
which suffices. The bound for $T_{\Re \tQ_\alpha} \tQ_\alpha$ is again identical. Finally,
\[
\| T_{\Re \tW_{\alpha,\xi_0}} \tQ_{\alpha,\mu} \|_{\dot H^\frac14} \lesssim  
 \| \tW_{\alpha,\xi_0}\|_{L^\infty} \| \tQ_{\alpha,\mu} \|_{\dot H^\frac14} \lesssim 
t^{-\frac12} \| \tW \|_{X} \| \tQ \|_{X_{ell}^\sharp}.
\]

\bigskip

\emph{iii) The elliptic-elliptic case.} Here on one hand there are more subcases, but on the other 
hand the gains are also larger, and one only needs to use the $X^\sharp_{ell}$ norm and Bernstein's inequality. This case is left for the reader.

\end{proof}

The second part of the previous proposition
allows us to reiterate, ultimately eliminating the paradifferential terms
from the system \eqref{KS-system}:

\begin{proposition}\label{p:reiterate-KS}
The functions $(\tW,\tQ)$ also are solutions for a system of the form
\begin{equation}\label{KS-system-lin}
\left\{
\begin{aligned}
& 2 \alpha \D_\alpha \tW - t\D_\alpha \tQ = G,
\\
& 2 \alpha \D_\alpha \tQ + it\tW  = K,
\end{aligned}  
\right.
\end{equation}
where we control
\begin{equation}\label{KS-control+}
\| (\tW,\tQ)\|_{\dH^\frac14 \cap \dH^\sigma} 
+ \|(G,K)\|_{\dH^\frac14} 
\lesssim \epsilon t^{C \epsilon^2 }.
\end{equation}
\end{proposition}

\begin{proof}
We know that $(\tW,\tQ)$ solve the system \eqref{KS-system}, and satisfy \eqref{KS-control}. Then the estimate \eqref{KS-control+} follows directly from 
and \eqref{better-para}.
\end{proof}

As a corollary of the last Proposition, it follows that similar bounds apply
to the components of the  hyperbolic part given by \eqref{hyp-decomp}:

\begin{corollary}
The summands in \eqref{hyp-decomp} applied to $(\tW_{hyp},\tQ_{hyp})$ satisfy 
the bounds:
\begin{equation}
\| P_{\xi_0} \bchi_{\alpha_0}(\tW,\tQ)  \|_{\dH^\frac14 \cap \dH^\sigma}   \lesssim  
\epsilon t^{C \epsilon },
\end{equation}
and solve an equation of the form \eqref{KS-system-lin} with source terms
$(G_{\alpha_0},K_{\alpha_0})$ with 
\begin{equation}
\|  (G_{\alpha_0},K_{\alpha_0}) \|_{\dH^\frac14}   \lesssim 
\epsilon t^{C \epsilon }.
\end{equation}
\end{corollary}

This can be seen by applying the results of Section~\ref{ss:wq} to $(\tW,\tQ)$ as in Proposition~\ref{p:reiterate-KS}. This is of course an overkill, as the analysis simplifies 
considerably when the paradifferential coefficients vanish, and one could also essentially cite the results of \cite{IT-global}.

\section{Wave packets and long time pointwise bounds}\label{s:wp2}

The goal of this section is to close the circle of ideas in this paper, i.e to use the bootstrap assumption and the energy estimates, along with the vector field Sobolev bounds in the previous section, in order to derive the long time pointwise bound \eqref{pointwise-nf} on the solutions
at the level of the normal form variables. This is accomplished by studying an appropriate \emph{asymptotic equation},
which is captured using the method of testing by wave packets developed earlier by the 
last two authors, see \cite{NLS}, \cite{IT-global}. The main result of this section is

\begin{proposition}
Assume that the normal form variables $(\tW,\tQ)$ satisfy the pointwise 
bootstrap bounds \eqref{pointwise-bootstrap-nf} 
as well as the energy bounds 
\eqref{ee-global-nf}. Then they satisfy \eqref{pointwise-nf}. 
\end{proposition}

As a starting point for the proof of this proposition we recall the properties that we have available for $(\tW, \tQ)$. 
First of all, $(\tW,\tQ)$ solves the system \eqref{tWQ-system} with the a cubic nonlinearity $(\tG^{(3)}, \tK^{(3)})$ given by \eqref{tGK3}, and a source term $(\tG^{(4+)}, \tK^{(4+)})$ satisfying the bound 
\begin{equation}
\label{tsource-re}
\Vert (\tG^{(4+)}, \tK^{(4+)})\Vert_{ \dH^{\frac{1}{4}} }\lesssim \epsilon \epsilon^4\langle t \rangle ^{C_1\epsilon^2-\frac{3}{2}}.
    \end{equation}
    
    For $(\tW, \tQ)$ we recall the energy estimates from Proposition~\ref{p:reiterate-KS}:
\begin{equation}
\label{te-re}
\Vert (\tW, \tQ)\Vert_{\dH^{\sigma} \cap \dH^{\frac{1}{4}} }\lesssim \epsilon \langle t \rangle ^{2C\epsilon}.
\end{equation}
    and
\begin{equation}
\label{tvf-re}
\Vert (2\alpha\partial_{\alpha} \tW-t\partial_{\alpha} \tQ, 2\alpha\partial_{\alpha} \tQ +it\tW) \Vert_{ \dH^{\frac{1}{4}}}\lesssim \epsilon  \langle t \rangle^{2C\epsilon}.
    \end{equation}  
    Given this starting point, our objective is to show that we have the pointwise bound 
\begin{equation}
\label{tpoint-re-better}
\| (\tW_\alpha, \tQ_\alpha) \|_{X}
\lesssim \epsilon \langle t \rangle^{-\frac12}.
\end{equation}

For $(\tW, \tQ)$ we take advantage of the analysis in the previous section, where $(\tW, \tQ)$ are decomposed into an elliptic and hyperbolic parts
\[
(\tW, \tQ) =(\tW, \tQ)_{ell} + (\tW, \tQ) _{hyp}.
\]
For the elliptic part we can use the bounds \eqref{tWQ-ell4} from Corollary~\ref{c:KS-para} to conclude that 
\[
 \Vert (\tW_{\alpha}, \tQ_{\alpha})_{ell}\Vert_{X}
\lesssim \epsilon^2 \langle t \rangle^{-\frac12 -\frac{b}{2}+C\epsilon^2}, 
\]
which suffices for $\epsilon$ small enough. Hence it remains to prove that \eqref{tpoint-re-better} holds for the hyperbolic part
\begin{equation}
\label{tpoint-re-better-hyp}
\| (\tW_\alpha, \tQ_\alpha)_{hyp} \|_{X}
\lesssim \epsilon \langle t\rangle^{-\frac12}.
\end{equation}
On the other hand for the hyperbolic part 
we have the pointwise bounds from Corollary~\ref{c:KS-para}:
\begin{equation}
\label{tpoint-re}
\| (\tW_\alpha, \tQ_\alpha)_{hyp} \|_{X^\sharp}
\lesssim \epsilon t^{-\frac12+C\epsilon^2},
\end{equation}
which are not good enough because of the $t^{C\epsilon^2}$ loss.
However, the $X^\sharp$ norm includes an additional gain 
away from dyadic velocity $1$, which is captured by the bound
\eqref{tWQ-hyp-out} which we recall here 
\begin{equation}
\| (1-\bchi_{\Omega^\delta}) (\tW,\tQ)_{hyp} \|_{X} \lesssim \epsilon t^{-\frac12 -b\delta + C \epsilon^2} .   
\end{equation}
This gives enough decay outside the region $\Omega^\delta$ defined in \eqref{Omega-delta}.
Hence it remains to obtain a bound inside $\Omega^\delta$, and show that
\begin{equation}\label{tpoint-re-better-hyp+}
\| \bchi_{\Omega^\delta} (\tW,\tQ)_{hyp} \|_{X} \lesssim \epsilon t^{-\frac12}.    
\end{equation}

In order to establish the global pointwise decay estimates \eqref{tpoint-re-better-hyp} in $\Omega^\delta$ we use the
method of testing by wave packets,  first introduced in 
paper \cite{NLS} in the context of the one dimensional cubic NLS
equation, and then used in the water waves context in \cite{IT-global} and other subsequent works.  The construction of the wave packets is identical with the one we have used \cite{IT-global}, but for convenience we recall it here.  This method, as emphasized in all our  results, requires localization of the initial data. 

The premise of the wave packet testing is that, at leading order, nonlinear waves
travel in a linear fashion along a ray which is connected to their spatial frequency via the linear Hamilton flow.  We take the ray to be $\{\alpha = v t\}$, and we refer to $v$ as the velocity; the associated frequency will be denoted by $\xi_v = v^{-2}$.  Our goal is to establish decay for the pair $(\tW,\tQ)$ along this ray by testing it with a wave packet evolving along the ray. The wave packet testing will only see a certain frequency of $(\tW,\tQ)$ along the ray, namely $\xi_v$; but this will suffice for our uniform decay bounds.

In our context here, by a wave packet we mean an appropriately localized approximate solution, i.e. with $O(1/t)$ errors, of the linear system 
\begin{equation} \label{ww2d-0} \left\{
\begin{aligned}
 & W_{ t} +  Q_\alpha   =  0
\\
& Q_t - i W = 0.
\end{aligned}
\right.
\end{equation}

We recall some key facts about how one should envision a wave packet. The dispersion relation $\tau = \pm \sqrt{|\xi|}$ gives that a ray with velocity $v$ is associated with waves which have  spatial frequency 
\[
\xi_v = - \frac{1}{4 v^2}=-\frac{t^2}{4\alpha^2}.
\]

This is associated with the phase function
\[
\phi(t,\alpha) = \frac{t^2}{4\alpha},
\]
which can also be seen as a solution to the appropriate eikonal equation,
and is exactly the phase of the fundamental solution, as predicted by the stationary phase method.

Then our wave  packets will be combinations of functions of the 
\[
\uu(t,\alpha) = v^{-\frac32} \chi\left(\frac{\alpha - vt}{t^\frac12 v^\frac32}\right)e^{i\phi(t,\alpha)},
\]
where $\chi$ is a smooth compactly supported bump function with integral one
\begin{equation}\label{chi-int}
\int \chi (y)\, dy =1.
\end{equation}
Our packets are localized around the ray $\{\alpha = v t\}$ on the
scale $\delta \alpha = t^\frac12 v^\frac32$. This exact choice of
scale is determined by the phase function $\phi$. Precisely, the quadratic 
expansion of $\phi$ near $\alpha = vt$ reads
\[
\phi(t,\alpha) = \phi(t,vt) + (\alpha-vt) \phi_\alpha(t,vt) + O( t^{-1} v^{-3}  (\alpha-vt)^2),
\]
and our scale $\delta \alpha$ represents exactly the scale on which $\phi$ is well
approximated by its linearization. We further remark that there is a threshold
$v \approx t$ above which $\phi$ is essentially zero, and the above considerations
are no longer relevant. By contrast, the above phase blows up at $\alpha = 0$.
In order to avoid proximity to either of these 
extreme scenarios, we confine our analysis to a region of the form 
\begin{equation}
\label{vbound}
\Omega^0:= \left\{ t^{-\frac{1}{100}}\leq |v| \leq  t^{\frac{1}{100}}\right\},
\end{equation}
which contains the smaller region $\Omega^\delta$. These powers of $t$ in the definition of $\Omega^0$ are chosen  rather arbitrary; they need to be universal small enough constants.

Under this assumption, the function $\uu$ is strongly localized at frequency $\xi_v$.
For later use, we record here some ways to express this localization.
We recall here some of the results in \cite{HIT}, and \cite{IT-global} that we will rely on without further adjustments.
\begin{lemma}
\label{l:uu}
a) Let $\uu$ be defined as above. Then its Fourier transform and that
of $\partial_v \uu$ have the form
\begin{equation}\label{hatq}
\hat \uu (\xi)=  t^\frac12 \chi_1\left(\frac{\xi + (4v^2)^{-1}}{ t^{-\frac12} v^{-\frac32}} \right) e^{- i t \sqrt{|\xi|}},
\qquad \partial_v \hat{\uu}(\xi) =  t v^{-\frac32}  \chi_2\left(\frac{\xi + (4v^2)^{-1}}{ t^{-\frac12} v^{-\frac32}} \right) e^{- i t \sqrt{|\xi|}} ,
\end{equation}
 where $\chi_1$ and $\chi_2$ are Schwartz functions so that in addition, 
\begin{equation}\label{chi1-int}
\int \chi_i(\xi) \, d\xi = 1 + O(v^{\frac12}t^{-\frac12}), \quad i=\overline{1,2}.
\end{equation}

b) For $s \geq 0$, $\lambda_v = (4v^2)^{-1}$ and $P_{\lambda_v}$ the associated dyadic frequency projector  we have
\begin{equation} \label{dsuu}
P_{\lambda_v}(|D|^s   - (4v^2)^{-s}) \uu(\alpha,t)  =  (4v^2)^{-s} t^{-\frac12} v^{\frac12}
\chi_3\left(\frac{\alpha - vt}{t^\frac12 v^\frac32}\right)e^{i\phi(t,\alpha)},
\end{equation}
where $\chi_3$ is also a Schwartz function.
\end{lemma}

\bigskip

Our use of the method of testing by wave packets proceeds in a similar fashion as in \cite{IT-global}. The linear correlation between our unknowns $(\tW,\tQ)$ makes it easier to chose one wave packet for one of the variables, and then match it for the second variable.  
As our linear system \eqref{ww2d-0} is simple enough, it suffices to first choose the $\tQ$ component and then use the second of the two linear equations in \eqref{ww2d-0}  to match $\tW$,
\begin{equation*}
%\label{defpac}
(\ww,\qq) = ( -i v \partial_t \uu, v \uu),  
\end{equation*}
where $\ww$ and $\qq$ are the wave packets associated to $\tW$, and $\tQ$ respectively.

Then we have
\begin{equation}
\label{w-bold}
\ww = \frac12 \uu + \left( \frac{vt -\alpha}{2 \alpha} \chi\left(\frac{\alpha - vt}{t^\frac12 v^\frac32}\right)
 + \frac{i (vt+\alpha)}{2t^\frac32 v^\frac12}  \chi'\left(\frac{\alpha - vt}{t^\frac12 v^\frac32}\right)\right) v^{-\frac32} e^{i\phi(t,\alpha)}.
\end{equation}
The second term above is better by a $v^\frac12 t^{-\frac12}$ factor, so it will play a negligible role in 
most of our analysis. However, it is crucial in improving the error in the first linear equation
in  \eqref{ww2d-0}, which is given by 
\begin{equation}
\label{g-eq}
\ggg  := \partial_t \ww + \partial_\alpha \qq = v (\partial_\alpha - i \partial_t^2) \uu .
\end{equation}
 Indeed,  computing the  error in \eqref{g-eq} we obtain
\begin{equation}
\label{erori}
\begin{aligned}
(\partial_\alpha - i \partial_t^2) \uu &=\frac{e^{i\phi}}{v^{\frac{3}{2}}} \partial_{\alpha}\left[ \frac{(\alpha-vt) }{2\alpha}\chi-i \frac{ (\alpha +vt)^2}{4v^{\frac{3}{2}}t^{\frac{5}{2}}}\chi '\right]+ \frac{e^{i\phi}}{v^{\frac{3}{2}}} \left[ \frac{(\alpha-vt)}{2\alpha^2}\chi -i \frac{ (\alpha -vt)}{4v^{\frac{3}{2}}t^{\frac{5}{2}}}\chi'\right].
\end{aligned}
\end{equation}
The leading term is the first one, which, as expected, has size $t^{-1}$
times the size of $\ww$; the presence of $\partial_\alpha$ endows it with a critical 
structural property which we will take advantage of later on.  The second term is better by another $t^{\frac12}$ factor, and will only play a perturbative role in the sequel.

The reader is cautioned that one should not think about the above wave packets
as a global approximate solution for the linear system.  Instead, as in \cite{NLS} and as in \cite{IT-global}, our test packets $(\ww,\qq)$ are good approximate solutions for the linear
system associated to our problem only on the dyadic time scale $\delta t\leq t$.  

The outcome of testing the normal form solutions to the water wave
system with the wave packet $(\ww,\qq)$ is  the
scalar complex valued function $\gamma(t,v)$, defined by
\begin{equation*}
%\label{def-gamma}
\gamma(t,v) = \langle (\tW,\tQ),(\ww,\qq)\rangle_{\dH^0},
\end{equation*}
which we will use as a good measure of the size of $(\tW,\tQ)$ along
our chosen ray.  Here it is important that we use the complex pairing
in the inner product. Note that here we are following \cite{IT-global} and using the original energy space $\dH^0$, and not the fractional Sobolev space $\dH^{\frac14}$.

While the above asymptotic profile $\gamma$ is defined everywhere, we will only use it in the region $\Omega^0$  in \eqref{vbound}. This is  because we  already have sufficient decay outside this region, indeed outside the smaller region $\Omega^\delta$. Furthermore, we will see that $\gamma$ primarily carries information about the hyperbolic part of $(\tW, \tQ)$, but this is all that is needed.

\medskip
Now we have two
tasks. Firstly, we need to show that $\gamma$ is a good representation of the 
pointwise size of $(\tW,\tQ)_{hyp}$ and their derivatives:

\begin{proposition}
\label{p:diff}
Assume that \eqref{te-re} and \eqref{tvf-re} hold. Then in  $\Omega^0$ we have 
the following bounds for $\gamma$:
\begin{equation}\label{gamma-hiv}
\| v^{-\frac{1}{2}}\gamma \|_{L^2_v}  + \|v^{\frac{1}{2}} \partial_v \gamma \|_{L^2_v} 
+  \|\gamma \|_{L^\infty}  \lesssim 
\epsilon t^{C\epsilon^2},\quad v\gtrsim 1,
\end{equation}
\begin{equation}\label{gamma-lowv}
\|v^{-2\sigma}\gamma \|_{L^2_v}  + \|v^{\frac{1}{2}} \partial_v \gamma \|_{L^2_v} 
+  \| v^{\frac14-\sigma}\gamma \|_{L^\infty}  \lesssim 
\epsilon t^{C\epsilon^2},\quad v\lesssim 1,
\end{equation}
as well as the approximation bounds for $(\tW,\tQ)_{hyp}$ and their derivatives:
\begin{equation} \label{pactest}
\begin{split}
&(|D|^s \tW,|D|^{s+\frac12} \tQ)_{hyp} (t,vt) =  \ |\xi_v|^{s}
t^{-\frac12} e^{i\phi(t,vt)} \gamma(t,v) ( 1, \sgn{v}) + \err_s, 
\\  
%& (\hat \tW,|\xi|^{\frac12}\hat \tQ)(t,\xi) =   \sum_{\xi = -(4v^2)^{-1}}  |v|^{-\frac32} (e^{i t \sgn{v}  |\xi|^\frac12}  \gamma(t,v)(1,\sgn{v}) + \hat{\err}),
\end{split}
\end{equation}
where
\begin{equation}\label{pacerr}
\begin{split}
&\|v^{2s-1} \err_{s}\|_{L^2_v}\lesssim \epsilon t^{-1}, \qquad  \|v^{2s-\frac{1}{4}} \err_s\|_{L^\infty} \lesssim  
\epsilon t^{-\frac{3}{4}}.
\\
%&\|v^{-1} (1+v^{-2})^2 \hat \err\|_{L^2_v}  + \| (1+v^{-2})^2 \hat \err\|_{L^\infty} \lesssim  
%\epsilon t^{-\frac{1}{18}+C^2_{*}\epsilon^2}.
\end{split}
\end{equation}
\end{proposition}

Compared to the work in \cite{IT-global}, here we do not limit the range for $s$ because we are only comparing the profile $\gamma$ with $(\tW,\tQ)_{hyp}$, and not with the full pair $(\tW, \tQ)$. Also \cite{IT-global} contains similar relations between the Fourier transforms of $(\tW, \tQ)$ and $\gamma$, which for brevity we omit here.

\medskip

Secondly, we need to show that $\gamma$ stays bounded,  which we do by 
establishing a differential equation for it:

\begin{proposition}
\label{p:ode}
Assume that \eqref{te-re}, \eqref{tvf-re}, \eqref{tpoint-re} and \eqref{tsource-re}  hold.
Then within the set $\Omega^0$ the function $\gamma$ solves an asymptotic ordinary differential equation of the form 
\begin{equation}\label{gamma-ode}
\dot\gamma = \frac{i}{2t (2v)^5}   \gamma |\gamma|^2 + e, 
\end{equation}
where $e$ satisfies the $L^2$ and $L^\infty$ bounds
\begin{equation}\label{e-hiv}
 \| v^{\frac{19}{8}} e\|_{L^{\infty}}  \lesssim 
 \epsilon^2 t^{-\frac{9}{8}+C\epsilon^2}, 
\end{equation}
\begin{equation}\label{e-lowv}
\| v^{-\frac12} e\|_{L^2_v} \lesssim 
 \epsilon^2 t^{-\frac54} t^{C\epsilon^2}.
\end{equation}
\end{proposition}

We now use the two propositions to conclude the proof of \eqref{tpoint-re-better-hyp+}. 
By virtue of \eqref{pactest} and \eqref{pacerr}, in order to prove \eqref{tpoint-re-better-hyp+}
it suffices to establish its analogue for $\gamma$, namely 

\begin{equation}\label{need}
|\gamma(t,v)| \lesssim \epsilon  \min \left\{v^{1^-}, v^{\frac52^+}\right\} \qquad \text{  in  $\Omega^{\delta}$}.
\end{equation}
Here by $1^-$, respectively $\frac52^+$ we denote universal constants slightly smaller than $1$,
respectively slightly larger than $5/2$; these are needed in order to insure dyadic frequency summation in the Besov norms in the definition of the $X$ norm. 

On the other hand,  from \eqref{gamma-hiv} and \eqref{gamma-lowv} we directly obtain

\begin{equation}\label{have}
|\gamma(t,v)| \lesssim \epsilon \min \left\{ 1, v^{\sigma-\frac14}\right\}  t^{C \epsilon^2} \text{ in $\Omega_{0}$}.
\end{equation}

Our goal now is to use the ode \eqref{gamma-ode} in order to transition from 
\eqref{have} to \eqref{need} along rays $\alpha = vt$. We consider three cases for $v$:

(i)  Suppose first that
$v \approx 1$, i.e., $|\alpha| \approx t$. Then we initially have 
\[
|\gamma(t)| \lesssim \epsilon, \qquad t \approx 1.
 \]
Integrating \eqref{gamma-ode} we conclude that 
\[
|\gamma(t)| \lesssim \epsilon, \qquad t \geq 1,
\]
and then \eqref{need} follows.

(ii) Assume now that $v \ll 1$, i.e., $|\alpha| \ll t$. Then, as $t$ increases, the ray $\alpha = vt$
enters $\Omega^0$ at some point $t_0$ with $v \approx t_0^{-\frac{1}{100}}$. Then by 
 \eqref{have} we obtain
\[
|\gamma(t_0,v)| \lesssim \epsilon v^{\sigma -\frac14} t^{C \epsilon^2}  \lesssim \epsilon v^{\frac52^+}.
\] 
We use this to initialize $\gamma$. For larger $t$ we use \eqref{gamma-ode}
to conclude that 
\[
|\gamma(t)| \lesssim \epsilon v^{\frac52^+} + \int_{t_0}^\infty \epsilon s^{-\frac98 +C^2\epsilon^2} v^{-\frac{19}{8}} \, ds \approx \epsilon v^{\frac52^+} +
 \epsilon  t_0^{-\frac1{8}+C \epsilon^2 } v^{-\frac{19}{8}}\lesssim \epsilon v^{\frac52^+} , \qquad t > t_0.
\]
Then  \eqref{need} follows.

(iii) Finally, consider the case $v \gg 1$, i.e., $|\alpha| \gg t$.
Again, as $t$ increases, the ray $\alpha = vt$ enters $\Omega^0$ at some point $t_0$  with
 $v \approx t_0^{\frac{1}{100}}$, therefore by \eqref{have} we obtain
\[
|\gamma(t_0,v)| \lesssim \epsilon  t_0^{C \epsilon^2}  \lesssim \epsilon v^{1^{-}}. 
\] 
We use this to initialize $\gamma$. For larger $t$ we use \eqref{gamma-ode}
to conclude that 
\[
|\gamma(t)| \lesssim \epsilon v^{1^-} + \int_{t_0}^\infty \epsilon s^{-\frac98 +C\epsilon^2} v^{-\frac{19}{8}} \, ds  \approx 
 \epsilon v^{1^-} +
 \epsilon  t_0^{-\frac1{8}+C^2\epsilon^2 } v^{-\frac{19}{8}}\lesssim \epsilon v^{1^-}  ,  \qquad t > t_0.
\]
Then  \eqref{need} again follows.

\bigskip

We remark that a more precise conclusion of the above analysis is the fact that 
as $t \to \infty$, the asymptotic profile $\gamma(t,v)$ is well approximated by 
solutions to the exact asymptotic equation,
\[
 \gamma(t,v) = \gamma_{\infty}(v) e^{i c(v) \ln t |\gamma_\infty(v)|^2} + err_\gamma,
\]
where the error $\err_\gamma$ decays to $0$ in both weighted $L^2$ and in weighted $L^\infty$ norms.
This leads to a good asymptotic representation of the solutions $(W,Q)$ in terms 
of its scattering data represented by $\gamma_\infty$. We do not pursue this here,
but instead we refer the reader to the similar analysis already carried out in \cite{IT-global}.

\bigskip

The remainder of the paper is devoted to the proof of the two propositions above.

\subsection{Approximation errors.}

Here we prove Proposition~\ref{p:diff}. In what follows in this subsection, the analysis happens all at fixed time, based on the 
elliptic/hyperbolic decomposition of $(tW,\tQ)$ in the previous section.
We first recall the decomposition of $(\tW, \tQ)$ from the previous section into localized components
\[
(\tW, \tQ) =\sum_{\alpha_0} \bchi _{\alpha_0}(\tW, \tQ),
\]
which we only need in the region $\Omega^0$. Because the bump functions $\chi_{\alpha_0}$ have essentially disjoint supports, it suffices to consider a single one of them, which is supported in the region $\alpha \approx \alpha_0$, and corresponds to velocities
$v \approx v_0 = \alpha_0/t$. The hyperbolic frequencies associated with this component are comparable to $\xi_0 =v_0^{-2}$. On the other hand, because of the spatial localization, such a component will interact with our wave packet only if the velocity of the wave packet is also comparable with $v_0$. Hence, the wave packet is also essentially supported at frequencies comparable to $\xi_0$.

For this component we consider the decomposition of the pair $(\tW,\tQ)$ into elliptic and hyperbolic parts
\[
\bchi _{\alpha_0}(\tW, \tQ)=  (\tW_{\alpha_0, ell}, \tQ_{\alpha_0, ell} )+  (\tW_{\alpha_0, hyp}, \tQ_{\alpha_0, hyp} ),
\]
where
\[
(\tW_{\alpha_0, ell}, \tQ_{\alpha_0, ell} )=(1-P_{\xi_0}) \bchi _{\alpha_0}(\tW, \tQ), \quad (\tW_{\alpha_0, hyp}, \tQ_{\alpha_0, hyp} )=P_{\xi_0} \bchi _{\alpha_0}(\tW, \tQ).
\]

At this point we observe that the elliptic part is frequency separated from our wave packet so its contribution to $\gamma$ is of size $O(t^{-N})$, with $N$ large, and thus negligible.
We conclude
that 
\[
\gamma(t,v) = \langle (\tW_{hyp},\tQ_{hyp}),(\ww,\qq)\rangle_{\dH^0} 
+ O(t^{-N}).
\]
So from here on we focus on the hyperbolic component only, which is fully localized in frequency, at dyadic frequency $\xi_0$.

\medskip

Borrowing an idea from \cite{IT-global}, we  symmetrize the problem by introducing the normalized
variables 
\[
(w,r) = (\tW_{\alpha_0, hyp}, \vert D\vert ^{\frac12}\tQ_{\alpha_0, hyp} ),
\]
 which satisfy the bounds
\begin{equation*}
\| (w,r)\|_{\dot{H}^{\frac14}\cap \dot{H}^{\sigma}} \leq \epsilon t^{C^2 \epsilon^2}, \qquad
\| (2 \alpha \partial_\alpha w -  i t  |D|^\frac12 r ,  
2\alpha \partial_\alpha  r - it |D|^\frac12 w )\|_{\dot{H}^{\frac14}_{\alpha}} \lesssim \epsilon t^{C^2 \epsilon^2},
\end{equation*}
or equivalently, using the frequency localization at $\xi_0 \approx v^{-2}$,
\begin{equation*}
\| (w,r)\|_{L^2} \leq \epsilon t^{C^2 \epsilon^2} \min \left\{ v^{\frac12},v^{2\sigma}  \right\}, \quad
\| (2 \alpha \partial_\alpha w -  i t  |D|^\frac12 r ,  
2\alpha \partial_\alpha  r - it |D|^\frac12 w )\|_{L^2_{\alpha}} \lesssim \epsilon t^{C^2 \epsilon^2} v^{\frac12}.
\end{equation*}

Then we rewrite $\gamma$ in terms of these variables
as
\[
\gamma = \int w \bar \ww + r D^\frac12 \bar \qq \, d\alpha .
\]
Here, following \cite{IT-global},  we  discard acceptable errors, and redefine  $\gamma$  as
\begin{equation}
\label{rewrite-gamma}
\gamma(t,v) = \frac{1}{2} \int (w\pm r) \bar{\textbf{u}} \, d\alpha.
\end{equation}
Then Proposition~\ref{p:diff}  is a consequence of the following Lemma:

\begin{lemma}
Let $\gamma$ be defined as in \eqref{rewrite-gamma} in the region 
$\Omega^0$, where $(w,r)$ are holomorphic functions, localized at frequency $\xi_0$,
which satisfy
\begin{equation}
\label{H}
\| (w,r)\|_{L^2} \leq \min \left\{  v^{\frac12}, v^{2\sigma}\right\}, \qquad \| (2 \alpha \partial_\alpha w -  i t  |D|^\frac12 r ,  2\alpha \partial_\alpha  r - it |D|^\frac12 w )\|_{L^2_{\alpha}} \lesssim v^{\frac12}.
\end{equation}
Then $\gamma$ satisfies the bounds
\begin{equation}
\label{point-gamma}
 \| v^{-\frac{1}{2}}(1+v^{-2})^{\sigma -\frac14}\gamma\|_{L^2_v}  + \| v^{\frac12} \partial_v \gamma \|_{L^2_v} \lesssim  1, 
\qquad |\gamma| \lesssim (1+v^{-2})^{-(\frac{\sigma}{2}-\frac18) }.
\end{equation}
Moreover, the following error bounds for $\gamma$ also hold:
\begin{equation}\label{pactest1}
\begin{split}
&|D|^s(w,r)(t,vt) =   |\xi_v|^s t^{-\frac12}e^{i\phi(t,vt)} \gamma(t,v)(1,\sgn{v}) + \err_s, \\
%&  (\hat w,\hat r)(t,\xi) =  \ \sum_{\xi = - (4v^2)^{-1}}  |v|^{-\frac32}(  
%e^{i t \sgn{v}  |\xi_v|^\frac12}  \gamma(t,v)(1,\sgn{v}) + \hat{\err}),
\end{split}
\end{equation}
where
\begin{equation}\label{pacerr1}
\begin{split}
&\| \err_{s} \|_{L^2_v} \lesssim  v_0^{1-2s}t^{-1}, \qquad  \| \err_s\|_{L^\infty} \lesssim  v_0^{\frac14-2s}t^{-\frac34} , 
\qquad 0 \leq s . \\
%&\|v^{-1} (1+v^{-2})^{2} \hat\err \|_{L^2_v} \lesssim  t^{-\frac1{18}}, \ \, \qquad  \|(1+v^{-2})^{2} \hat\err\|_{L^\infty} \lesssim  t^{-\frac1{18}} , 
%\quad \qquad 0 \leq s \leq 2. 
\end{split}
\end{equation}
\end{lemma}

\begin{proof}
The proof is similar with the argument in \cite{IT-global}, but simpler. The reason for this is that the pair of functions $(w,r)$ are already frequency localized in the hyperbolic region. In order to fix signs, we first need to differentiate between the two symmetric cases $v_0 > 0$ and $v_0 < 0$. Without any restriction in generality we take 
$v_0 > 0$. We express everything in terms of $w-r$ and $y=w+r$. Then $w-r$ does not contribute to $\gamma$, but it contributes to the error. In addition,  subtracting the two components in the second term
in \eqref{H} we obtain 
\[
\|(2\alpha \vert D\vert +t\vert D\vert^{\frac{1}{2}} )(w-r)\|_{L^2} \lesssim v_0^{\frac12}.
\]
The operator above is elliptic in $\{\alpha \approx v_0 t\}$, therefore 
we obtain
\[
\| w-r\|_{L^2} \lesssim t^{-1}v_0^{\frac12} \xi_0^{-\frac12}=t^{-1}v_0^{\frac32}.
\]
Thus, we can directly bound its contribution $|D|^{s}(w-r)$ to the error term in $L^2$ and in $L^{\infty}$ by Bernstein's inequality. We note that the exponents will not match
with \eqref{pacerr1}; instead, here we obtain a gain, which is akin to the similar gain for the elliptic component of $(\tW,\tQ)$.

We now consider the contribution of $y$, noting that $\gamma$ is already expressed in terms
of $y$. To reduce the problem to an estimate for $y$ we
need one last step. Combining again the two components in the second term
in \eqref{H} we obtain 
\[
\| |D|^{\frac12}  (4\alpha^2 \partial_{\alpha} + it^2)(w,r) \|_{L^2_{\alpha}} \lesssim t v_0^{\frac12},
\]
which yields the same bound for $y$. In view of frequency localization at frequency $\approx \xi_0$ we  conclude that 
\begin{equation}\label{z1}
\left\|  L y \right\|_{L^2_{\alpha}}
 \lesssim v_0^{-\frac12}t^{-1}, \qquad L = \partial_{\alpha} +\frac{it^2}{4\alpha^2}.
\end{equation}
On the other hand, from the first relation in \eqref{H} we obtain
\begin{equation}\label{z2}
\| y \|_{L^2_\alpha} \lesssim v^{\frac12}(1+ v_0^{-2})^{\frac14 -\sigma}.
\end{equation}
From here on we will work only with the function $y$.

Following \cite{IT-global} we rewrite the bounds on $y$ in terms of the auxiliary function
$u:= e^{-i\phi} y$, which satisfies
$ \partial_{\alpha}u=e^{-i\phi}(\partial_{\alpha}+\frac{it^2}{4\alpha^2})y$.
Then for $u$ we have 
\begin{equation}\label{u1}
\left\|   \partial_\alpha u \right\|_{L^2_{\alpha}}
 \lesssim v_0^{-\frac12}t^{-1}, \qquad \| u \|_{L^2_\alpha} \lesssim v_0^{\frac12}(1+ v_0^{-2})^{\frac14-\sigma}.
\end{equation}
Combining these bounds we get by interpolation
\begin{equation*}
\label{ui}
\|   u \|_{L^\infty} \lesssim t^{-\frac12} (1+ v_0^{-2})^{\frac18-\frac{\sigma}{2}},
\end{equation*}
which also is transferred back to $y$,
\begin{equation}\label{zi}
\|  y \|_{L^\infty} \lesssim t^{-\frac12} (1+ v_0^{-2})^{\frac18-\frac{\sigma}{2}}.
\end{equation}

The bounds \eqref{z2} and \eqref{zi} lead directly to $L^2$ and $L^\infty$ bounds for $\gamma$,
\begin{equation}\label{game}
\|\gamma \|_{L^2_v} \lesssim v_0^{\frac12}(1+ v_0^{-2})^{\frac14-\sigma}, 
\qquad\|\gamma \|_{L^\infty} \lesssim  (1+ v_0^{-2})^{\frac18-\frac{\sigma}{2}}.
\end{equation}
To estimate $\partial_v \gamma = \langle y, \partial_v \uu
\rangle_{L^2}$ we write $\partial_v \uu$ in the form
\[
\partial_v \uu = 
- v^{-\frac32} e^{i \phi} \left(  t \partial_\alpha \chi\left( \frac{\alpha-vt}{t^\frac12 v^{\frac32}}\right)
+ \frac32 \frac{\alpha-vt}{t^\frac12 v^\frac52} \chi'\left( \frac{\alpha-vt}{t^\frac12 v^{\frac32}}\right)
\right),
\]
and compute using integration by parts
\[
\partial_v \gamma = \int v^{-\frac32}  t \partial_\alpha u(t,\alpha) \chi\left( \frac{\alpha-vt}{t^\frac12 v^{\frac32}}\right) \, d\alpha -  \int v^{-\frac32}  u(t,\alpha)  \frac32 \frac{\alpha-vt}{t^\frac12 v^\frac52} \chi'\left( \frac{\alpha-vt}{t^\frac12 v^{\frac32}}\right) d\alpha .
\]
Now we can bound the two  integrals using \eqref{u1} to obtain
\[
\| \partial_v \gamma \|_{L^2_v} \lesssim v_0^{-\frac12} ,
\]
which, together to \eqref{game}, concludes the proof of \eqref{point-gamma}.

%%%%%%%%%%%%%%%%%%%%%%%%%%%       %%%%%%%%%%%%%%%%%%%%%   

It remains to estimate the $L^2$ and $L^{\infty}$ norms of the error
in \eqref{pacerr}.  We begin with the case $s=0$, were we bound the 
the difference $$\err= y(t,vt) - t^{-\frac12}
e^{i\phi(t,vt)} \langle y, \uu \rangle_{L^2}$$ in both $L^2_v$ and
$L^\infty$ in terms of $\|y\|_{L^2_\alpha}$ and $\|Ly\|_{L^2_\alpha}$, exactly as in \cite{IT-global}
\begin{equation}
\label{err-est}
\| \err\|_{L^\infty} \lesssim v_0^\frac34 t^{\frac{1}{4}}\Vert
  Ly \Vert_{L^2_{\alpha}}, \qquad \| \err\|_{L^2_v} \lesssim 
v_0^{\frac32} \Vert Ly \Vert_{L^2_{\alpha}}.
 \end{equation}
 This is exactly what we need for \eqref{pacerr} in the case $s=0$. Due to the frequency localization for $y$, adding extra derivatives simply adds factors of $\xi_0^s =v_0^{-2s}$ to the bound.
 
 \subsection{The asymptotic equation for  \texorpdfstring{$\gamma$}{}}
Here we track the evolution of $\gamma(t,v)$ and prove
Proposition~\ref{p:ode}.  The computation is based on the energy conservation relation for
the linear system \eqref{ww2d-0}. If both $(\tW, \tQ)$ and $(\ww, \qq)$ were solutions to the homogeneous linear system \eqref{ww2d-0}, then we would get $\dot{\gamma}=0$. As it is, $\dot{\gamma}$ depends on the source terms in the linear equation \eqref{ww2d-0} applied to $(\tW, \tQ)$, respectively $(\ww, \qq)$. The source term in the $(\ww,\qq)$ equation is $(\ggg, 0)$ with $\ggg$ given by \eqref{g-eq}. The source term in the similar $(\tW, \tQ)$ equation comes  from \eqref{tWQ-system}. Thus we obtain the relation
\begin{equation}
\label{gamma dot}
\dot \gamma(t) = \int \left( \tG -T_{2\Re \tW_{\alpha}}\tQ_{\alpha} + T_{2\Re \tQ_{\alpha}}\tW_{\alpha}\right) \bar \ww + \tW \bar \ggg + i \left( \tK + T_{2\Re \tQ_{\alpha}}\tQ_{\alpha} \right)_\alpha \bar \qq \, d\alpha  .
\end{equation}
We successively consider all terms on the right. With the exception of a single 
term, namely the resonant part of $\tG$, see below, all contributions will be placed 
into the error term $\sigma$. 

We need to estimate at fixed time the terms in $\dot \gamma$ in the region $\Omega^0$. The wavepacket components are localized on the scale $t^{\frac12}v^{\frac32}$ around the ray $\alpha =vt$. Therefore we can harmlessly regard  $(\tW, \tQ)$ as being also localized in the corresponding dyadic region $\alpha \approx vt$.

To estimate the error terms it is convenient to begin with a lemma that captures the main computations that lead to the error bound:

\begin{lemma}
\label{l:theta}
Let $f$ be supported in a dyadic region $v\approx v_0$ and 
\begin{equation}
    \theta (v):=\int f(\alpha) \uu (\alpha) \, d\alpha .  
\end{equation}
Then the following bounds hold
\begin{equation}
    \Vert \theta \Vert_{L^2_v}\lesssim \Vert f\Vert_{L^2_{\alpha}} ,
\end{equation}
respectively
\begin{equation}
    \Vert \theta \Vert_{L^{\infty}_v}\lesssim t^{\frac12}\left(t^{\frac12}v^{\frac32} \right)^{-\frac{1}{p}} \Vert f\Vert_{L^p_{\alpha}}.
\end{equation}
\end{lemma}
The result does not depend on the choice of the bump function $\chi$ in the definition of $\uu$. 
\begin{proof}
Here we only use the size of the function $\uu$ which is a bump function on the scale $t^{\frac{1}{2}}v^{\frac{3}{2}}$ with norms
\[
\Vert \uu \Vert_{L^1_{\alpha}} \lesssim t^{\frac12},\quad \Vert \uu \Vert_{L^{\infty}_{\alpha}} \lesssim v^{-\frac32}.
\]
Then the bounds  are obtained akin to Young's inequality with the minor difference that the integral defining $\theta$ is not an exact convolution, but can be bounded by one (in absolute value). 
\end{proof}

\bigskip

{\bf A. The contribution of $\bar \ggg$.}  This is 
\[
I_1 = v^{-\frac32} \int \tW e^{-i\phi} \left( \partial_{\alpha}\left[ \frac{(\alpha-vt) }{2\alpha}\chi-i \frac{ (\alpha +vt)^2}{4v^{\frac{3}{2}}t^{\frac{5}{2}}}\chi '\right]+  \left[ \frac{(\alpha-vt)}{2\alpha^2}\chi -i \frac{ (\alpha -vt)}{4v^{\frac{3}{2}}t^{\frac{5}{2}}}\chi'\right]\right) \, d\alpha.
\]
We use \eqref{pactest}  and \eqref{tWQ=ell+hyp} to replace $\tW$ in terms of $\gamma$ 
\[
\tW - t^{-\frac12}e^{i\phi} \gamma(t,v)
= \tW_{ell}+ \left(\tW_{hyp} - t^{-\frac12}e^{i\phi} \gamma(t,v)\right).
\]
The elliptic part $\tW_{ell}$ is mismatched with $\bar\ggg$ in frequency, so its contribution is $O(t^{-N})$. The contribution of the second term above
is directly estimated in both $L^2$ and $L^\infty$ via  \eqref{pacerr}.

The contribution of $\gamma$, on the other hand, is written using integration by parts as 
\[
\tilde{I}_1 := v^{-\frac32} t^{-\frac12} \int  - \gamma_\alpha \left[ \frac{(\alpha-vt) }{2\alpha}\chi-i \frac{ (\alpha +vt)^2}{4v^{\frac{3}{2}}t^{\frac{5}{2}}}\chi '\right]+  \gamma \left[ \frac{(\alpha-vt)}{2\alpha^2}\chi -i \frac{ (\alpha -vt)}{4v^{\frac{3}{2}}t^{\frac{5}{2}}}\chi'\right] \, d\alpha.
\]
Now we can easily bound the two terms using \eqref{gamma-hiv}, \eqref{gamma-lowv} and Lemma~\ref{l:theta} to obtain
\[
\Vert \tilde{I}_1\Vert_{L^2_v}\lesssim t^{-1}v^{\frac12}\Vert \gamma_{\alpha}\Vert_{L^2_{\alpha}} +t^{-2}v^{-\frac12} \Vert \gamma \Vert_{L^2_{\alpha}}\lesssim t^{-\frac32} \epsilon t^{C\epsilon ^2},
\]
respectively
\[
\begin{aligned}
\Vert \tilde{I}_1\Vert_{L^{\infty}_v}&\lesssim t^{-1}v^{\frac12} t^{\frac{1}{4}}v^{-\frac34}\Vert \gamma_{\alpha}\Vert_{L^2_{\alpha}} +t^{-2}v^{-\frac12}t^{\frac12} \Vert \gamma \Vert_{L^{\infty}_{\alpha}}\\
&\lesssim t^{-\frac54} v^{-\frac34}\epsilon t^{C\epsilon ^2} +t^{-\frac32} v^{-\frac12}\epsilon t^{C\epsilon ^2}\\
&\lesssim t^{-\frac54}v^{-\frac34} \epsilon t^{C\epsilon ^2}.
\end{aligned}
\]
Here, at the last step, we used that we are in the region $\Omega^0$, given by \eqref{vbound}.
\bigskip

{\bf B. The contribution of the paradifferential source terms.} This is given by
\[
I_2:= \int \left(  -T_{2\Re \tW_{\alpha}}\tQ_{\alpha} + T_{2\Re \tQ_{\alpha}}\tW_{\alpha}\right) \bar \ww + i \left(  T_{2\Re \tQ_{\alpha}}\tQ_{\alpha} \right)_\alpha \bar \qq \, d\alpha . 
\]
As before, the goal is to estimate $I_2$ in $L^2_v$ and $L^{\infty}_v$. For the $L^2_v$ bound it suffices to use the estimate in \eqref{better-para}, with the observations that up to rapidly decaying tails only the frequencies of size $\xi_0$ will contribute.  Combining this observation with Lemma~\eqref{l:theta}, we have
\[
\Vert I_2\Vert_{L^2_v}\lesssim \xi_0^{-\frac14}\left\| \left(  -T_{2\Re \tW_{\alpha}}\tQ_{\alpha} + T_{2\Re \tQ_{\alpha}}\tW_{\alpha} ,  T_{2\Re \tQ_{\alpha}}\tQ_{\alpha} \right) \right\|_{\dot{H}^{\frac14}_v}\lesssim v_0^{\frac12} \epsilon^2 t^{-\frac54} t^{C\epsilon^2}.
\]

Unfortunately \eqref{better-para}  is no longer sufficient to estimate the $L^{\infty}_v$ bound of $I_2$, so we need a more refined analysis.

The three terms in $I_2$ are mostly similar, with the first one being a little bit better in terms of the time decay. We will discuss the second one in detail, and the third one will be identical with the second one. Hence in what follows we seek to estimate 
\[
I_{2,2}:=\int  T_{2\Re \tQ_{\alpha}}\tW_{\alpha} \bar \ww \, d\alpha.
\]
We start with a simple observation, namely that $\bar{\ww}$ is localized at frequency $\xi_0$ which means that  the only nontrivial contribution arises from the component $\tW_{\alpha}$ which is also localized at frequency $\xi_0$. This in turn implies that for $\tQ_{\alpha}$ we only use frequencies $\ll \xi_0$.  Thus we can replace $\tW_{\alpha}$ by $\tW_{hyp, \alpha}$, and $\tQ_{\alpha}$ is replaced by $\tQ_{ell, \alpha}$, to write
\[
I_{2,2}=\int  T_{2\Re \tQ_{ell, \alpha}}\tW_{hyp, \alpha} \bar \ww \, d\alpha +O(t^{-N}).
\]
The advantage of working with the low frequency elliptic component of $\tQ$ is that it satisfies a better $L^2$ type bound which is part of the $X^{\sharp}$ bound in Corollary~\ref{c:KS-para}. Precisely we have
\begin{equation}
\label{q-ell}
\Vert\tQ_{ell,\alpha} \Vert_{\dot{H}^{\frac14}}\lesssim \epsilon t^{-1}t^{C\epsilon^2},
\end{equation}
which is the only bound we will need for the paradifferential coefficient. Using Bernstein inequality this also gives the pointwise bound
\begin{equation}
\label{q-ell-point}
\Vert\tQ_{ell,\alpha} \Vert_{L^{\infty}}\lesssim \epsilon v_0^{-\frac12}t^{-1}t^{C\epsilon^2}.
\end{equation}
This is better than the $t^{-\frac12}$ decay in the hyperbolic region, but still not enough. 

The next step is to use the $\tW_{hyp, \alpha}$ representation in  \eqref{pactest} which gives 
\[
\tW_{hyp, \alpha} \approx t^{-\frac12}\xi_0P_{\xi_0}\left[\gamma e^{i\phi}\right] + \err_1,
\]
where $\err_1$ satisfies the pointwise bound    
\[
\vert \err_1 \vert \lesssim \epsilon t^{-\frac34}v_0^{-\frac74}.
\]
The contribution of $\err_1$ to $I_{2,2}$ is estimated via Lemma~\ref{l:theta} to obtain
\[
I_{2,2}=t^{-\frac12}\xi_0\int  T_{2\Re \tQ_{ell, \alpha}}P_{\xi_0}\left[ \gamma e^{i\phi}\right] \bar \ww \, d\alpha +O(\epsilon^2t^{-\frac54}v_0^{-\frac94}t^{2C\epsilon^2}).
\]
Our next simplification is to freeze $\gamma$ to its values at the center of the packet, which we denote by $\gamma_0$. Within the support of the packet the difference can be estimated  by H\"older's inequality and \eqref{gamma-hiv}-\eqref{gamma-lowv}
\[
\vert \gamma -\gamma_0\vert\lesssim \int_ {\vert v-v_0\vert \lesssim t^{-\frac12}v_0^{\frac32}} |\partial_v\gamma |\, dv \lesssim t^{-\frac14}v_0^{\frac34} \Vert \partial_v\gamma \Vert_{L^2_v} \lesssim \epsilon t^{C\epsilon^2}t^{-\frac14} v_0^{\frac14}.
\]
Estimating directly the corresponding error, and using the expression of $\bar{\ww}$, we arrive at
\[
I_{2,2}=v^{-\frac32}t^{-\frac12}\xi_0 \gamma_0\int  T_{2\Re \tQ_{ell, \alpha}}P_{\xi_0}\left[  e^{i\phi}\right] \chi e^{-i\phi} \, d\alpha +O(\epsilon^2 t^{-\frac54}v_0^{-\frac94}t^{2C\epsilon^2}),
\]
where we can easily drop the projector $P_{\xi_0}$ because the exponential is already frequency localized around  the same frequency $\xi_0$. It remains to bound the following integral in $L^{\infty}$
\[
I_{2,2}'=v^{-\frac32}t^{-\frac12}\xi_0 \gamma_0\int  [T_{2\Re \tQ_{ell, \alpha}}  e^{i\phi}] \chi e^{-i\phi} \, d\alpha .
\]

We decompose $\tQ_{ell}$ in low and high frequencies in comparison to the $\xi_0$ frequency:
\[
\tQ_{ell}:= \tQ_{ell}^{low} +\tQ_{ell}^{high},
\]
where the truncation threshold comes from the wave-packet frequency scale:
\[
\tQ_{ell}^{low}:= (\tQ_{ell})_{<v^{-\frac32}t^{-\frac12}}, \quad \tQ_{ell}^{high}:= (\tQ_{ell})_{\geq v^{-\frac32} t^{-\frac12}}.
\]
Same notation will apply to a similar decomposition in frequencies for  $\tQ_{ell,\alpha}$.

It is easier to first estimate the contribution of lower frequencies in $\tQ_{ell}$
\[
I_{2,2}'=v^{-\frac32}t^{-\frac12}\xi_0 \gamma_0\int  [T_{2\Re (\tQ^{low}_{ell, \alpha})}  e^{i\phi}] \chi e^{-i\phi} \, d\alpha .
\]
which we bound directly as follows using Lemma~\eqref{l:theta}
\[
\begin{aligned}
\Vert I_{2,2}'\Vert_{L^{\infty}_{\alpha}} & \lesssim \xi_{0}\gamma_0 \Vert T_{2\Re (\tQ^{low}_{ell, \alpha})}  e^{i\phi}\Vert_{L^{\infty}_{\alpha}}  \lesssim   \xi_{0}\gamma_0  \Vert \tQ^{low}_{ell, \alpha}\Vert_{L^{\infty}_{\alpha}}\lesssim \epsilon v^{-2}t^{2C\epsilon^2}\Vert \tQ^{low}_{ell, \alpha}\Vert_{L^{\infty}_{\alpha}} .
\end{aligned}
\]
We bound the last term separately by means of Bernstein's inequality and \eqref{tWQ-ell3} to get
\[
\Vert \tQ^{low}_{ell, \alpha}\Vert_{L^{\infty}_{\alpha}} \lesssim v^{-\frac38}t^{-\frac18}\Vert \tQ^{low}_{ell, \alpha}\Vert_{\dot H^{\frac14}_{\alpha}}\lesssim \epsilon v^{-\frac38}t^{-\frac98} t^{C\epsilon^2}.
\]
The final estimate is the contribution of the high frequencies of $\tQ$ to $L^{\infty}_{\alpha}$ bound. It involves the $L$ operator defined in \eqref{z1}. We begin by observing the representation
\[
\begin{aligned}
T_{\tQ_{ell,\alpha}^{high}}&=[\partial_{\alpha}, T_{\tQ_{ell}^{high}}]\\
&=[L, T_{\tQ_{ell}^{high}}] - \left[i\frac{t^2}{\alpha^2}, T_{\tQ_{ell}^{high}}\right].\\
\end{aligned}
\]
We now estimate separately the two contributions. For the first one we integrate by parts
\[
I_{2,3}'=v^{-\frac32}t^{-\frac12}\xi_0 \gamma_0\int [L, T_{\tQ_{ell}^{high}}]e^{i\phi}\chi e^{-i\phi}\, d\alpha = v^{-\frac32}t^{-\frac12}\xi_0 \gamma_0 \int T_{\tQ_{ell}^{high}}[e^{i\phi}]\chi_{\alpha} e^{-i\phi}\, d\alpha .
\]
Here we used 
\[
L e^{i\phi}=0, \quad L^{*}e^{-i\phi}=0,
\]
where $L^*$ is the adjoint operator. We have
\[
\begin{aligned}
\Vert I_{2,3}'\Vert_{L^{\infty}_{\alpha}} & \lesssim \xi_{0}\gamma_0 v^{-\frac32}t^{-\frac12}\Vert T_{2\Re (\tQ^{high}_{ell})}  e^{i\phi}\Vert_{L^{\infty}_{\alpha}} \Vert  \chi_{\alpha} \Vert_{L^1}\\
& \lesssim v^{-\frac32}t^{-\frac12} \xi_{0}\gamma_0  \Vert \tQ^{high}_{ell}\Vert_{L^{\infty}_{\alpha}}\\
&\lesssim \epsilon v^{-\frac32}t^{-\frac12} v^{-2}v^{\frac98}t^{\frac{3}{8}+2C\epsilon^2}\Vert \tQ_{ell, \alpha}\Vert_{\dot H^{\frac{1}{4}}_{\alpha}} \\
&\lesssim \epsilon^2 v^{-\frac{19}{8}}t^{-\frac{9}{8}+2C\epsilon^2}.
\end{aligned}
\]
The last integral is

\[
I_{2,4}'=v^{-\frac32}t^{-\frac12}\xi_0 \gamma_0\int \left[i\frac{t^2}{\alpha^2}, T_{\tQ_{ell}^{high}}\right]e^{i\phi}\chi e^{-i\phi}\, d\alpha.
\]
Thus,using Lemma $2.5$ from \cite{AIT} we get
\[
\begin{aligned}
\Vert I_{2,4}'\Vert_{L^{\infty}_{\alpha}} & \lesssim \xi_{0}\gamma_0 \left\| \left[i\frac{t^2}{\alpha^2}, T_{\tQ_{ell}^{high}}\right]e^{i\phi}\right\|_{L^{\infty}_{\alpha}} \Vert v^{-\frac32}t^{-\frac12} \chi \Vert_{L^1}\\
& \lesssim \xi_{0}^{-1}\gamma_0  \Vert t^2\alpha^{-3} \Vert_{L^{\infty}_{\alpha}} \Vert \tQ^{high}_{ell, \alpha}\Vert_{L^{\infty}_{\alpha}}\\
&\lesssim \epsilon v^{-1}t^{-1+2C\epsilon^2}\Vert \tQ^{high}_{ell, \alpha}\Vert_{L^{\infty}_{\alpha}}\\
& \lesssim \epsilon^2 v^{-1}t^{-\frac{13}{8}+2C\epsilon^2}.
\end{aligned}
\]

Adding up all contributions we conclude that in the region $\Omega^0$ we have the bound 
\[
\Vert I_2\Vert_{L^{\infty}_v} \lesssim \epsilon^2 v^{-\frac{19}{8}}t^{-\frac{9}{8}+2C\epsilon^2}.
\]

\bigskip
 
{\bf C. The contribution of $\tG$ and $\tK$.}  
For this we consider in more detail the structure of $\tG$ and
$\tK$. We will successively peel off favorable terms until we are left 
only with the leading resonant part. We decompose them into cubic and higher terms,
\[
\tG = \tG^{(3)} + \tG^{(4+)}, \qquad \tK = \tK^{(3)} + \tK^{(4+)}.
\]
To start with we decompose them into quartic and higher order terms. 

\medskip

{\bf C1. Quartic and higher order terms.} We denote their contribution by
\[
I_3:=\int   \tG^{(4+)} \bar \ww +  i \tK^{(4+)}_\alpha \bar \qq \, d\alpha.
\]

In view of \eqref{tGK4-est} and Lemma~\eqref{l:theta},  we can estimate the contribution of the quartic and higher
terms in $L^\infty$,  
\[
\Vert I_3\Vert_{L^{\infty}}
\lesssim t^{\frac12} t^{-\frac14}v^{-\frac34} \xi_0^{-\frac14} \Vert (\ww, \qq)\Vert_{\dH^{\frac14}}
\lesssim   \epsilon^4 v^{-\frac14} t^{-\frac54+3C^2 \epsilon^2},    
\]
which suffices in $\Omega^0$. The $L^2$ bound is similar, using again \eqref{tGK4-est} and Lemma~\eqref{l:theta}.

\bigskip
{\bf C2. Cubic terms.} 
It remains to consider the contributions arising from the cubic terms, which can be viewed as  translations invariant trilinear forms
\[
\tG^{(3)} =\tG^{(3)}(\tW_{\alpha}, \tW_{\alpha}, \tQ_{\alpha}), \quad \tK^{(3)} =\tK^{(3)}(\tW_{\alpha}, \tQ_{\alpha}, \tQ_{\alpha}).
\]
These trilinear expressions include also the complex conjugates. Here, we first peel off some perturbative terms by substituting in $(\tG^{(3)}, \tK^{(3)})$ the following sequence of transformations
\[
(\tW_{\alpha},  \tQ_{\alpha}) \rightarrow (\tW_{hyp,\alpha},  \tQ_{hpy, \alpha})\rightarrow
i\xi_v P_{\xi_0}[\gamma t^{-\frac12}e^{i\phi}(1, \vert \xi_v\vert ^{-\frac12}\sgn v)],
\]
where we denote the final outcome by
\begin{equation}
\label{WQ0}
i\xi_v P_{\xi_0}[ \gamma t^{-\frac12}e^{i\phi}(1, \vert \xi_v\vert ^{-\frac12}\sgn v)]=:(\tW^0_{\alpha}, \tQ^0_{\alpha}).
\end{equation}
To control the errors we need to estimate the difference
\[
(\tG^{(3),0}_{ell}, \tK^{(3),0}_{ell})\!\! : =\!\! (\tG^{(3)}\!(\tW_{\alpha}, \tW_{\alpha}, \tQ_{\alpha}),\! \tK^{(3)}\!(\tW_{\alpha}, \tQ_{\alpha}, \tQ_{\alpha})) -(\tG^{(3)}\!(\tW^0_{\alpha}, \tW^0_{\alpha}, \tQ^0_{\alpha}),\! \tK^{(3)}\!(\tW^0_{\alpha}, \tQ^0_{\alpha}, \tQ^0_{\alpha})).
\]
For this transition we have the unlocalized elliptic difference bounds 
\begin{equation}
\Vert (\tW_{\alpha}, \tQ_{\alpha}) -(\tW^0_{\alpha}, \tQ^0_{\alpha})\Vert_{X^{\sharp}_{ell}}\lesssim \epsilon t^{C\epsilon^2-\frac12},
\end{equation}
which are a consequence of \eqref{tWQ-ell2}, \eqref{pactest}, \eqref{pacerr}.

Then we can estimate $(\tG^{(3),0}_{ell}, \tK^{(3),0}_{ell}) $ using \eqref{GK3-ell} to obtain 
\[
\Vert (\tG^{(3),0}_{ell}, \tK^{(3),0}_{ell}) \Vert_{\dH^{\frac14}} \lesssim \epsilon^4 t^{-\frac32}t^{C\epsilon^2}.
\]
This allows us to conclude the bound as in the case of the quartic bound. 

Now we consider one last transition from 
\[
(\tW^0_{\alpha}, \tQ_{\alpha}^0)= i\xi_vP_{\xi_0}[\gamma t^{-\frac12}e^{i\phi}(1, \vert \xi_v\vert ^{-\frac12}\sgn v)] \rightarrow i\xi_0\gamma_0 t^{-\frac12}e^{i\phi}(1, \vert \xi_0\vert ^{-\frac12}\sgn v):=(\tW^1_{\alpha}, \tQ^1_{\alpha}),
\]
where we emphasize that $(\tW^1, \tQ^1)$ depend also on the wavepacket parameters $\xi_0$ and $v_0$.

Now we need to estimate the difference 
\[
(\tG^{(3),1}_{ell}, \tK^{(3),1}_{ell})\!\! : = \!\! (\tG^{(3)}\!(\tW^0_{\alpha}, \tW^0_{\alpha}, \tQ^0_{\alpha}), \! \tK^{(3)}\!(\tW^0_{\alpha}, \tQ^0_{\alpha}, \tQ^0_{\alpha}))
-
(\tG^{(3)}\!(\tW^1_{\alpha}, \tW^1_{\alpha}, \tQ^1_{\alpha}),\!
 \tK^{(3)}\!(\tW^1_{\alpha}, \tQ^1_{\alpha}, \tQ^1_{\alpha})). \!
\]
Here it is important that these differences are only needed within the support of the wave packet $(\ww,\qq)$. There the leading contribution comes from the difference between $(\tW^0,\tQ^0)$ and $(\tW^1,\tQ^1)$, within a slightly larger region of comparable size $\left\{ \vert \alpha-vt\vert \leq c v^{\frac32} t^{\frac12}\right\}$, with $c$ a large positive constant.

This difference can be estimated by H\"older's inequality as follows
\[
\begin{aligned}
\| \gamma - \gamma_0\|_{L^2} \lesssim  & \ t^\frac12 v^\frac32 \|\gamma_\alpha \|_{L^2}
\lesssim \epsilon t^{C\epsilon^2} v
\qquad \| \gamma - \gamma_0\|_{L^\infty} \lesssim (t^\frac12 v^\frac32)^\frac12 \|\gamma_\alpha \|_{L^2} \lesssim \epsilon t^{C\epsilon^2} v^\frac14 t^{-\frac14}
\end{aligned}
\]
where at the last step we have used the bound for  $\partial_v \gamma$ in \eqref{gamma-hiv}, \eqref{gamma-lowv}. Using these bounds, the contribution of $(\tG^{(3),1}_{ell}, \tK^{(3),1}_{ell})$  can be estimated as in the quartic case.  

\medskip

We are now left with the task of estimating the contribution to $\dot \gamma$
of the cubic expressions $(\tG^{(3)}(\tW^1_{\alpha}, \tW^1_{\alpha}, \tQ^1_{\alpha}), \tK^{(3)}(\tW^1_{\alpha}, \tQ^1_{\alpha}, \tQ^1_{\alpha}))$. To achieve this, we 
need to consider the structure of the cubic terms.

Following \cite{HIT}, we  have the following classification of the terms in $(\tilde
G^{(3)}, \tilde K^{(3)})$:

\begin{itemize}
\item[A.] Nonresonant trilinear terms: these are either 
(A1) terms with no complex conjugates, or  
(A2) terms with two complex conjugates.

\item[B.] Resonant trilinear terms: terms with exactly one
  conjugation.  For such terms one may further define a notion of
  principal symbol, which is the leading coefficient in the expression
  obtained by substituting the factors in the trilinear form by the
  expressions in \eqref{pactest} \footnote{Which corresponds to all
    three frequencies being equal.}  Thus one can isolate a linear subspace
  of resonant terms for which this symbol vanishes, which we call {\em
    null terms}.  Hence on the full class of resonant trilinear terms
  we can further define an equivalence relation, modulo null terms.
\end{itemize}

Based on this, we reorganize $\tG^{(3)}$ in resonant, nonresonant and null terms:
\begin{equation}
\left\{
\begin{aligned}
    &\tG^{(3)}_r \, := \Pi((\tQ_\alpha \tW_\alpha)_\alpha,  \bar\tW)+ \Pi(\bar \tQ_\alpha, \tW_\alpha^2)\\
& \tG^{(3)}_{nr}\, := T_{\tW_\alpha} (\tQ_\alpha \tW_\alpha) +  T_{(\tQ_\alpha \tW_\alpha)_\alpha}\tW + \Pi(\tW_\alpha, 2\Re [\tQ_\alpha \tW_\alpha]) +\Pi(\tQ_\alpha \tW_\alpha)_\alpha,  \tW)\\
    &\qquad \quad  + T_{2 \Re (T_{\tW_\alpha}\tW + \Pi(\tW_\alpha,  \tW))_\alpha} \tQ_\alpha -T_{2 \Re \tW_\alpha} ( \tQ_\alpha \tW_\alpha )\\
    &\qquad  \quad +T_{2 \Re \tW_\alpha}(T_{\tQ_\alpha} \tW + \Pi(\tQ_\alpha, 2\Re \tW))_\alpha  + T_{2 \Re (\tQ_\alpha \tW_\alpha - (T_{\tQ_\alpha} \tW + \Pi(\tQ_\alpha,  \tW))_\alpha)}\tW_\alpha \\
    &\qquad \quad - T_{2 \Re \tQ_\alpha} ( T_{\tW_\alpha}\tW + \Pi(\tW_\alpha, 2\Re \tW))_\alpha  - \Pi(\bar \tW_\alpha^2, \tQ_\alpha) \\
     &\qquad \quad  - T_{\bar \tW_\alpha^2}\tQ_\alpha - T_{\bar \tW_\alpha} (\tF^{(2)}) + T_{\bar \tQ_\alpha}\tW_\alpha^2\\
& \tG^{(3)}_{null}\, :=  T_{\tF^{(2)}_\alpha}\tW- \tW_\alpha \tF^{(2)}  + \Pi (\tF^{(2)}_\alpha), 2\Re W) + \Pi( \tF^{(2)} , W_\alpha) +T_{2 \Re \tW_\alpha} \tF^{(2)}  \\
& \qquad \qquad +T_{ 2\Re(\Pi(\tW_\alpha, \bar \tW))_\alpha} \tQ_\alpha + T_{2 \Re (  \Pi(\tQ_\alpha, \bar \tW)_\alpha)}\tW_\alpha  + \Pi(\tW_\alpha, \bar \tF^{(2)}) .\\
    \end{aligned}
    \right.
\end{equation}

We do the same for $\tK^{(3)}$:
\begin{equation}
    \left\{
    \begin{aligned}
 &\tK^{(3)}_{r}\, :=  0 \\ 
 &\tK^{(3)}_{nr}\, :=  i T_{\tW_\alpha^2} \tW   - T_{2\Re  (T_{\tQ_\alpha} \tW + \Pi(\tQ_\alpha, \tW))_\alpha}\tQ_\alpha \\
 &\qquad \quad - T_{2\Re \tQ_\alpha} (T_{\tQ_\alpha} \tW + \Pi(\tQ_\alpha, 2\Re \tW))_\alpha  + T_{2\Re (\tQ_\alpha \tW_\alpha)} \tQ_\alpha 
 + T_{\bar \tQ_\alpha} (\tQ_\alpha \tW_\alpha) \\
 &\qquad \quad+ T_{\tQ_\alpha} T_{\tQ_\alpha}\tW_\alpha + T_{ \tQ_\alpha  \tQ_{\alpha\alpha} } \tW+ T_{\tQ_\alpha} (T_{\tW_\alpha} \tQ_\alpha + \Pi(\tW_\alpha, \tQ_\alpha))
 \\
 &\qquad \quad  +  \Pi(\tQ_\alpha, 2\Re [\tQ_\alpha \tW_\alpha]) - \Pi(\tW_\alpha \tQ_\alpha, \tQ_\alpha) - T_{\tQ_\alpha \tW_\alpha} \tQ_\alpha\\
 &\tK^{(3)}_{null}\, := - T_{2\Re ( \Pi(\tQ_\alpha, \bar \tW))_\alpha}\tQ_\alpha - T_{\tF^{(2)}} \tQ_\alpha + \Pi(\tQ_\alpha,  \bar \tF^{(2)}) + T_{ P[\vert \tQ_\alpha\vert^2]_{\alpha}} \tW\\
 &\qquad \qquad   +2\Pi (\Re W, \tQ_{\alpha}\tQ_{\alpha \alpha} +i\tW_{\alpha}^2 ) +2\Pi (\Re W, \partial_{\alpha}P [\vert \tQ_{\alpha}\vert ^2])
  \end{aligned}
    \right.
\end{equation}
 In these expressions we will substitute $(\tW_{\alpha},\tQ_{\alpha})$ by $(\tW_{\alpha}^1,\tQ_{\alpha}^1)$.

We will place all cubic contributions into  the error term $e$,  
except for  the contribution of the resonant part  $\tG^{(3)}_{r}$.

We note that for the most part the
exact form of the expressions above is irrelevant.  The only
significant matter is the coefficient of the terms in $\tG^{(3)}_{r}$,
which needs to be real\footnote{ A similar constraint would be
  required of the coefficients in $\tK^{(3)}_{r}$, if they were
  nonzero.}.

We also remark that the leading projection in all terms can be harmlessly discarded, 
since it can be moved onto the wave packets, which decay rapidly at positive frequencies,
\[
\| (\ww,\qq) - P (\ww,\qq) \|_{H^N} \lesssim t^{-N}.
\]

\medskip

\medskip
 
{\bf C2(a) The contribution of the null terms.} This is given by $\tG^{(3)}_{null}(\gamma)$ and
  $\tK^{(3)}_{null}(\gamma)$:
  \[
  \tG^{(3)}_{null}(\gamma):=\int_{\mathbb{R}}\tG^{(3)}_{null} (\tW^1_{\alpha},\tQ^1_{\alpha})\bar \ww\,d\alpha  \qquad  \tK^{(3)}_{null}(\gamma):=-i\int_{\mathbb{R}} \tK^{(3)}_{null}(\tW^1_{\alpha},\tQ^1_{\alpha}) \bar \qq_{\alpha}\, d\alpha .
  \]  Here we simply note that
$\tG^{(3)}_{null}(\gamma)= 0$ and $\tK^{(3)}_{null}(\gamma)= 0$, so
after the previous step there is nothing left to do.  We
remark that cancellation actually occurs at the bilinear level for the
``null expressions'' of type
\[
\tW_\alpha \bar \tQ_\alpha - \bar \tW_\alpha  \tQ_\alpha, \qquad 
(| \tQ_{\alpha}|^2)_\alpha, \qquad \tQ_\alpha \tQ_{\alpha \alpha}+i\tW_\alpha^2.
\]

\medskip

{\bf  C2(b) The contribution of the nonresonant terms.} This is given by 
$\tG^{(3)}_{nr}(\gamma)$ and $\tK^{(3)}_{nr} (\gamma)$:
\[
  \tG^{(3)}_{nr}(\gamma):=\int_{\mathbb{R}}\tG^{(3)}_{nr}(\tW^1_{\alpha},\tQ^1_{\alpha})\bar \ww\,d\alpha , \qquad  \tK^{(3)}_{nr}(\gamma):=-i\int_{\mathbb{R}} \tK^{(3)}_{nr} (\tW^1_{\alpha},\tQ^1_{\alpha})\bar \qq_{\alpha}\, d\alpha .
  \]

Here it is important that we integrate against $\ww$ and $\qq$, as
that fixes the frequency of the output at $\xi = -
\frac{1}{4v^2}$. On the other hand the nonresonant trilinear
expression will be concentrated at frequency $3 \xi$ if no complex
conjugate occur, respectively at frequency $-\xi$ if two conjugates
occur. Thus, because of this mismatched  the frequencies the only contributions here arise due to rapidly decaying tails, 
\[
(\tG^{(3)}_{nr} (\gamma), \tK^{(3)}_{nr}(\gamma)) =O(t^{-N}).
\]
\medskip

{\bf C2(c) The contribution of the resonant 
term.} This is given by $\tG^{(3)}_{r}(\gamma)$
\[
\tG^{(3)}_{r}(\gamma) :=\int_{\mathbb{R}}\tG^{(3)}_{r}(\tW^1_{\alpha},\tQ^1_{\alpha})\bar \ww\,d\alpha .
\]
Given the expression above we have
\[
\tG^{(3)}_{r}(\gamma) :=\int_{\mathbb{R}}\left[ \Pi((\tQ_\alpha^1 \tW^1_\alpha)_\alpha,  \bar\tW^1)+\Pi(\bar \tQ^1_\alpha, (\tW_\alpha^{1})^2)\right]\bar \ww\,d\alpha .
\]
Replacing $P$ by $I-P$ or $\Pi$ by $I-\Pi$ yields nonresonant terms  with have size $O(t^{-N})$.
Hence, we obtain
\[
\tG^{(3)}_{r}(\gamma) =\int_{\mathbb{R}}\left[ (\tQ_\alpha^1 \tW^1_\alpha)_\alpha  \bar\tW^1+\bar \tQ^1_\alpha (\tW_\alpha^{1})^2\right]\bar \ww\,d\alpha  +O(t^{-N}).
\]
Substituting $(\tW^1_{\alpha}, \tQ^1_{\alpha})$ from \eqref{WQ0}
we obtain the
integral
\[
\tG^{(3)}_r(\gamma)  = \gamma(t,\alpha/t)  |\gamma(t,\alpha/t)|^2  \int \frac{i}{t^\frac32} \left(\frac{t}{2\alpha}\right)^5 
  e^{-i\phi} \bar \ww \, d\alpha.
\]
Here $e^{-i\phi} \bar \ww$ has the form
\[
e^{-i\phi} \bar \ww = \frac12 v^{-\frac32} \chi\left(\frac{\alpha - vt}{t^\frac12 v^\frac32}\right)
+ v^{-1} t^{-\frac12} 
\tilde \chi\left(\frac{\alpha - vt}{t^\frac12 v^\frac32}\right)
\]
with Schwartz functions $\chi$ and $\tilde \chi$ so that $\int \chi = 1$. Thus we obtain
\[
\tG^{(3)}_r(\gamma) = \frac{i}{2 t} (2v)^{-5} \gamma(t,v) |\gamma(t,v)|^2 + O(t^{-\frac{3}{2}}),
\]
as needed.

\end{proof}

\bibliography{wwbib}

\bibliographystyle{plain}

\end{document}